\documentclass[11pt,a4paper,reqno]{amsart} 

\usepackage [english]{babel}  %sillabazione in inglese~

\DeclareFontFamily{U}{mathb}{\hyphenchar\font45}
\DeclareFontShape{U}{mathb}{m}{n}{
      <5> <6> <7> <8> <9> <10>
      <10.95> <12> <14.4> <17.28> <20.74> <24.88>
      mathb10
      }{}
\DeclareSymbolFont{mathb}{U}{mathb}{m}{n}

\DeclareMathSymbol{\sqbullet}{1}{mathb}{"0D}

\allowdisplaybreaks %Poter spezzarer una formula su più pagine

\usepackage{amsmath,mathtools}
\usepackage{amssymb}
\usepackage{euscript}
\usepackage{mathrsfs} 
\usepackage{bm}
\usepackage{graphicx}
\usepackage[all]{xy}
\usepackage[dvipsnames]{xcolor}
\usepackage[colorlinks=true,linktocpage=true,pagebackref=false, citecolor=black,linkcolor=black]{hyperref}
\usepackage{pifont}
\usepackage{tikz,pgfplots}
\pgfplotsset{compat=1.10}
\usepgfplotslibrary{fillbetween}
\usetikzlibrary{cd,arrows,decorations.pathmorphing,backgrounds,automata,positioning,fit,matrix}
\usepackage{caption,subcaption}
\usepackage{comment}
\usepackage{xcolor}

\usetikzlibrary{matrix}
\usetikzlibrary{decorations.pathreplacing, calc, fit, backgrounds}

\pgfplotsset{soldot/.style={color=black,only marks,mark=*}} \pgfplotsset{holdot/.style={color=black,fill=white,only marks,mark=*}}

\newtheorem{thm}{Theorem}[section]

\newtheorem{lem}[thm]{Lemma}
\newtheorem{prop}[thm]{Proposition}

\newtheorem{cor}[thm]{Corollary}

\newtheorem{assumption}[thm]{Assumption}
\newtheorem{quest2}[thm]{Question}

\theoremstyle{definition}
\newtheorem{defn}[thm]{Definition}

\newtheorem{remark}[thm]{Remark}

\newtheorem{example}[thm]{Example}

\theoremstyle{remark}

\numberwithin{equation}{section}
\numberwithin{figure}{section}

 \newcommand{\N}{{\mathbb N}}
 \newcommand{\R}{{\mathbb R}}
 \newcommand{\C}{{\mathbb C}}
 
\newcommand{\HH}{{\mathbb H}} 
\newcommand{\sph}{{\mathbb S}} \newcommand{\D}{{\mathbb D}}
 
 \renewcommand{\O}{{\mathbb O}}

\newcommand{\Cont}{\mathscr{C}}
 \newcommand{\I}{{\mathcal I}}

\newcommand{\Tt}{{\EuScript T}}
\newcommand{\Ee}{{\EuScript E}}

\newcommand{\Bb}{{\EuScript B}}
\newcommand{\Cc}{{\EuScript C}}

\newcommand{\Qq}{Q}

\newcommand{\Gg}{{\EuScript G}}

\newcommand{\Nn}{{\EuScript N}}

\newcommand{\Aa}{{\EuScript A}}

\newcommand{\Rr}{{\EuScript R}}

\newcommand{\Kk}{{\EuScript K}}

\newcommand{\Vv}{{\EuScript V}}
\newcommand{\Ww}{{\EuScript W}}

\newcommand{\inv}{\operatorname{inv}}

\newcommand{\x}{{\tt x}}  
 \renewcommand{\t}{{\tt t}}
 
 \newcommand{\e}{{\tt e}}

\newcommand{\eps}{\epsilon}

\newcommand\m[1]{[#1]_4}

\newcommand{\mc}{\mathcal}

\newcommand{\ds}{\bigtriangleup}

\begin{document}

\vspace{-2em}

\title[Reconstruction of a slice regular function]{Reconstruction of a slice regular function from some of its real components}

\date{\today}
\author{Cinzia Bisi}
\address{Cinzia Bisi, Dipartimento di matematica e informatica,
Via Machiavelli 30, Università di Ferrara, 44121 Ferrara (ITALY). {\tt ORCID: 0000-0002-4973-1053}}
\email{cinzia.bisi@unife.it}
\thanks{The first author is partially supported by GNSAGA of INdAM and by PRIN \textit{Variet\'a reali e complesse: geometria, topologia e analisi armonica}.}

\author{Antonio Carbone}
\address{Antonio Carbone, Dipartimento di Scienze dell'Ambiente e della Prevenzione, Palazzo Turchi di Bagno, C.so Ercole I D'Este, 32, Università di Ferrara, 44121 Ferrara (ITALY)}
\email{antonio.carbone@unife.it}
\thanks{The second and third author are partially supported by GNSAGA of INdAM.}

\author{Riccardo Ghiloni}
\address{Riccardo Ghiloni, Dipartimento di Matematica, Via Sommarive, 14, Universit\'a di Trento, 38123 Povo (ITALY)}
\email{riccardo.ghiloni@unitn.it}
%\thanks{}

\begin{abstract}
A basic property of holomorphic functions $f:D\to \C$ defined on domains $D$ of $\C$ is that $f$ is uniquely determined by its real part up to an additive constant. The same is true for slice regular functions defined on circular slice domains of the division algebra of quaternions or octonions. The aim of this paper is to extend the latter result to more general classes of algebras, including, among many others, the Clifford algebras $\mathbb{R}_{p,q}$ and the split octonions $\mathbb{S}\mathbb{O}$. New phenomena appear, as well as unexpected connections with graph theory and the theory of slice-Nash functions.
\end{abstract}

\maketitle 

\tableofcontents

\section{Introduction}

\subsection{Hypercomplex analysis} Since the beginning of the last century, there have been several attempts to determine a suitable class of functions of one quaternionic variable that would had played the same role as the holomorphic functions of one complex variable. One of the first successful notions of `regular' quaternionic functions is due to Fueter \cite{fu}. Unfortunately, the class of Fueter regular functions does not include polynomials and power series. Following some ideas of Cullen \cite{cu}, Gentili and Struppa introduced in \cite{gs} a notion of regularity for quaternionic functions, called \textit{slice regularity}. The class of slice regular functions includes (suitable) polynomials and convergent power series. After the work of Gentili and Struppa, the theory of slice regular functions has been widely developed and studied in the last twenty years. This theory has been extended to octonions \cite{gs3} and to Clifford algebras \cite{css}. Using a different approach to slice regularity, based on the concept of stem function, the third author of this paper in a joint work with Perotti \cite{gp} extended the theory to any finite dimensional real alternative $*$-algebra with unity. This approach also allows one to characterise slice regularity using a system of global differential equations \cite{gp5} and to extend the theory of slice regular functions in one variable to functions of several variables \cite{gp2}.

Slice regular functions share many properties with holomorphic functions - they satisfy the identity principle, Cauchy formulas, and admit (locally) suitable expansions as convergent power series. Several classical results of complex analysis have their quaternionic and octonionic counterpart. For instance, the open mapping theorem \cite{gst3,gst2}, the Weierstrass factorisation theorem \cite{gv}, the Schwarz-Pick lemma \cite{bco, bs}, the Landau theorem \cite{bs2}, the Picard theorem \cite{bw, ren1}, the Runge approximation theorem \cite{bw2}, the Brolin theorem \cite{bd}, the Jensen formula \cite{alb,bw3}, the classification of the automorphisms of the algebra of functions \cite{bw6, bw7}, Julia theory \cite{ren2}, and many others. We refer the reader to the monograph \cite{gsst} and the references therein for further information and details. Recently, slice regular functions have also been studied on domains that are more general than the natural slice domains: see \cite{GhSt, ren3} and the references mentioned therein. A unified theory of regularity in one hypercomplex variable was proposed in \cite{GhSt2} (see also \cite{GhSt1,GhSt3}).

\subsection{Reconstruction of a slice regular function} It is well known that if two holomorphic functions $f$ and $g$, defined on a domain $D$ of $\C$, have the same real part, then they coincide on $D$ up to an additive constant. In \cite[Cor.7.11]{gsst}, the authors show the analogous result for slice regular functions of one quaternionic variable, namely, they show that if two slice regular functions $f$ and $g$, defined on a circular slice domain $\Omega$ of the quaternions $\HH$, have the same real part, then they coincide on $\Omega$ up to an additive constant. Even if not explicitly mentioned in \cite{gsst}, the same result, with the obvious modification to the proof, holds true for slice regular functions defined on a circular slice domain $\Omega$ of the octonions $\O$. Thus, if the involved algebra is the algebra of quaternions $\HH$ or the algebra of octonions $\O$, then the real part (or any other of its real components) of a slice regular function $f$, defined on a circular slice domain $\Omega$, allows one to `reconstruct' $f$ up to an additive constant, i.e., $f$ is uniquely determined by its real part up to an additive constant.

The purpose of this paper is to extend this kind of results to a more general class of algebras, including, among many others, the Clifford algebras $\R_{p,q}$ and the split octonions $\mathbb{S}\O$.

\textit{Let $A$ be a finite dimensional real alternative $*$-algebra with unity $1\neq 0$ endowed with a fixed (real vector) basis $\Bb:=\{\e_0:=1,\e_1,\ldots,\e_m\}$}. Let $\Omega\subset \Qq_A$ be a circular slice domain and $f:\Omega\to A$ a slice regular function (we refer the reader to \S\ref{prelimi} for the precise definitions). We may write
$$
f=f_0+f_1\e_1+\ldots+f_{m}\e_{m},
$$
where $f_0,\ldots,f_{m}$ are the real components of $f$ with respect to the basis $\Bb$. The purpose of this paper is to address the following:

\begin{quest2}\label{minimo}
What is the minimum number of real components $f_k$ of $f$ needed to `reconstruct' f up to an additive constant?
\end{quest2}

In order to address this problem, we introduce the following:

\begin{defn}[Reconstructing number]\label{reconstructing}
We define the \textit{reconstructing number $\Rr(A)$ of $A$} as the minimum integer $\ell$ that satisfies the following property: there exists a subset $K$ of $\{0,\ldots,m\}$ of cardinality $\ell$ such that, for each circular slice domain $\Omega\subset \Qq_A$ and each slice regular function $f:\Omega\to A$, if the real components $f_k$ of $f$ are constant for each $k\in K$, then $f$ is constant. $\sqbullet$
\end{defn}

Clearly, $\Rr(A)$ is well-defined: if all the real components of a slice regular function $f$ are constant, then $f$ is also constant. We can say that determining the integer $\Rr(A)$ formalises and resolves Question \ref{minimo}. In fact, if $K$ is a subset of $\{0,\ldots,m\}$ as in Definition \ref{reconstructing} and $f$ and $g$ two slice regular functions on $\Omega$ such that $f_k=g_k$ for each $k\in K$, then $f-g$ is constant, so $f$ and $g$ coincide on $\Omega$ up to an additive constant. 

Answering Question \ref{minimo}, that is to determine $\Rr(A)$, is, in general, a difficult task. Moreover, the class of all finite dimensional real alternative $*$-algebra with unity is too `wide' and too `general' to obtain satisfactory results. That is why, in what follows we focus on a smaller class of $*$-algebras, that we call \textit{Cartan $*$-algebras}, and that we introduce in \S\ref{S3}. The class of Cartan $*$-algebras includes, among many others, the division algebras $\C$, $\HH$ and $\O$, the Clifford algebras $\R_{p,q}$, and the split octonions $\mathbb{S}\O$. 

Also for Cartan $*$-algebras determining the reconstructing number $\Rr(A)$ is, in general, a difficult task. Thus, we introduce another integer $\Cc(A)$, called the \textit{Cartan number} (see Definition \ref{CN}), which is easier to calculate and that allows one to give estimations for $\Rr(A)$. In order to introduce the integer $\Cc(A)$, we associate to the fixed basis of $A$ some combinatorial objects $\Cc$ called systems of Cartan pairs (see Definition \ref{SCP}). To a system of Cartan pairs $\Cc$ and a slice regular function $f$, we associate a new slice regular function $f^{\Cc}$, called the \textit{Cartan transform of}~$f$ (see Definition \ref{Ctransform}), which encodes many of the properties of the original slice regular function~$f$. Thanks to the Cartan transform (and some modifications of it, see Proposition \ref{Srecon}), we are able to obtain the following results for Cartan $*$-algebras:
\begin{itemize}
\item a `reconstruction' result for slice regular functions (see Corollary \ref{Crecon}),
\item a new maximum module principle (see Theorem \ref{Cmaxmodule}),
\item a splitting lemma (see Lemma \ref{SCPlitting}).
\end{itemize}
The Cartan transform $f^{\Cc}$ also allows us to make explicit the relationship between the two integers $\Rr(A)$ and $\Cc(A)$. In fact, we show the following (which corresponds to Theorem \ref{upper}):

\begin{thm}[Upper bound]
Let $A$ be a Cartan $*$-algebra. Then,
$$
\Rr(A)\leq \Cc(A)\leq\frac{\dim(A)}{2}.
$$
\end{thm}

By the previous inequality, $\Rr(A)=1$ when $\Cc(A)=1$, so in this case we can easily determine the reconstructing number $\Rr(A)$. Unfortunately, the only Cartan $*$-algebras that satisfies the equality $\Cc(A)=1$ are the division algebras $\C$, $\HH$ and $\O$ (see Theorem \ref{division1}). 

To a Cartan $*$-algebra $A$ can be naturally associated a regular graph $\Gg_A$ that we call the \textit{Cartan graph of} $A$ (see \S\ref{Sgrafo}). The Cartan number $\Cc(A)$ of the algebra $A$ is equal to the \textit{domination number} $\gamma(\Gg_A)$ of the graph $\Gg_A$ (see Proposition \ref{grafo} and \S\ref{Sgrafo} for the precise definitions). The domination number is an invariant of graphs studied by many authors and in the existing litera\-ture there are several estimations for it. In particular, the equality between the integers $\Cc(A)$ and $\gamma(\Gg_A)$ allows one to provide estimations for the reconstructing number $\Rr(A)$ using the estimations for the dominating number $\gamma(\Gg_A)$. For instance, we have the following (that corresponds to Proposition \ref{domin}, while the integer $s_A$ is introduced in \eqref{SA}):

\begin{prop}
Let $A$ be a Cartan $*$-algebra such that $s_A\geq 6$. Then,
\begin{equation}\label{stimagrafi}
\Rr(A)\leq \frac{127}{418}\dim(A)\simeq 0.30382775\dim(A).
\end{equation}
\end{prop}

In \S\ref{Scliff}, we study the reconstructing number $\Rr(\R_{p,q})$ of the Clifford algebras $\R_{p,q}$. First, we obtain the following general estimation (that corresponds to Corollary \ref{reconRpq}):

\begin{thm}[Reconstructing number of the Clifford algebras $\R_{p,q}$]
For each $p,q\in\N$ with $n:=p+q\geq 2$, it holds
\begin{equation}\label{betterstima2}
\Rr(\R_{p,q})\leq \frac{n^3+5n+6}{6}.
\end{equation}
\end{thm}

Considering all the Clifford algebras $\R_{p,q}$ together is a too general situation to obtain `good' estimations, and, in fact, the previous estimation is quite `rough'. Nevertheless, it is worthwhile to notice that, if $n\geq 5$, then $A=\R_{p,q}$ satisfies \eqref{stimagrafi} (see Corollary \ref{stimaRpq}) and that \eqref{betterstima2} is better than \eqref{stimagrafi} already for $n\geq 9$. Making $n$ grow, \eqref{betterstima2} becomes `much sharper' than \eqref{stimagrafi}, in fact, the right hand side of \eqref{betterstima2} goes like $O(n^3)$, while the one of \eqref{stimagrafi} like $O(2^n)$. Specialising to the Clifford algebras $\R_n:=\R_{0,n}$, we are able to significantly improve \eqref{betterstima2}. In fact, we show the following (which corresponds to Theorem \ref{reconRn}):

\begin{thm}[Reconstructing number of the Clifford algebras $\R_n$]
Let $n\geq 2$ be an integer. The Clifford algebra $\R_n$ satisfies:
\begin{itemize}
\item[{\rm(i)}] $\Rr(\R_2)=\Rr(\HH)=1$.
\item[{\rm(ii)}] $\Rr(\R_3)=\Rr(\R_4)=2$.
\item[{\rm(iii)}] If $n\equiv 1 \text{ mod } 4$, then $\Rr(\R_n)=2$.
\item[{\rm(iv)}] If $n\geq 6$ and $n\equiv 0,2,3 \text{ mod } 4$, then $3\leq \Rr(\R_n)\leq 4$.
\end{itemize}
\end{thm}

\subsection{Applications to slice-Nash functions}\label{IntroNash}
Nash functions were introduced by John Nash in his groundbreaking paper \cite{n}. A real (resp.\ complex) Nash function is a real analytic (resp.\ holomorphic) function which is algebraic over the ring of polynomials with real (resp.\ complex) coefficients (see Definition \ref{defNash} for the precise definition). Let $D$ be an open subset of $\C$ and $f:=f_0+if_1$ a holomorphic function on $D$. After identifying $\C$ with $\R^{2}$ in the usual way, the function $f_0$ can be regarded as a real analytic function. A natural question is the following:

\begin{quest2}\label{q1}
Is it true that, if the real part $f_0$ is a real Nash function, then $f$ is a complex Nash function?
\end{quest2} 

The second author of this paper answered the previous question in \cite{c}, showing the following:

\begin{thm}[{\cite[Thm.1.3, Cor.1.4]{c}}]\label{carbone}
The real part $f_0$ is a real Nash function if and only if the holomorphic function $f$ is a complex Nash function. 
\end{thm}

In \cite{bc}, the first and second author of this paper introduced slice-Nash functions. Even if in~\cite{bc} the focus is only on the quaternions and octonions, the definitions and most of the results of~\cite{bc} extends verbatim to slice regular functions defined on open circular subsets of the quadratic cone $\Qq_A$ of a finite dimensional real alternative $*$-algebra $A$ with unity $1\neq0$ (see \S\ref{SsliceNash}).

\textit{Let $\Omega\subset \Qq_A$ be an open circular set and $f:\Omega\to A$ a slice regular function. Assume that $A$ is endowed with a fixed (real vector) basis $\Bb:=\{\e_0:=1,\e_1,\ldots,\e_m\}$}. After the identification of $A$ with $\R^{\dim(A)}$ induced by the basis $\Bb$, we can regard the real components $f_k$ of $f$ as real analytic functions. Thus, we may ask the analogous of Question \ref{q1} in this contest, namely, to ask the following (see Definition \ref{defslicenash} for the definition of slice-Nash function):

\begin{quest2}\label{questNAsh}
What is the minimum number of real components $f_k$ of $f$ that have to be real Nash functions in order to guarantee that $f$ is a slice-Nash functions?
\end{quest2}

In order to address this question, we introduce the following:

\begin{defn}[Nash number]\label{Nashnumber}
We define the \textit{Nash number $\Nn(A)$ of $A$} as the minimum integer $\ell$ that satisfies the following property: there exists a subset $K$ of $\{0,\ldots,m\}$ of cardinality~$\ell$ such that, for each open circular set $\Omega\subset \Qq_A$ and each slice regular function $f:\Omega\to A$, if the real components $f_k$ of $f$ are real Nash functions for each $k\in K$, then $f$ is a slice-Nash function.~$\sqbullet$
\end{defn}

By Theorem \ref{carbone}, we know that $\Nn(\C)=1$. Making use of this fact and the results of \cite{bc}, we are able to show the following (which corresponds to Theorem \ref{mainNash}):

\begin{thm}
Let $A$ be a Cartan $*$-algebra. Then, $\Nn(A)\leq \Cc(A)$.
\end{thm}

By the previous result and using the same arguments (with minor modifications to the proofs that we will explain in \S\ref{NashRpq}) used to determine the reconstructing number $\Rr(\R_{p,q})$ of the Clifford algebras $\R_{p,q}$, we obtain the analogous results for the Nash number $\Nn(\R_{p,q})$. In fact, we have the followings (which correspond, respectively, to Theorem \ref{Nash1} and Theorem \ref{Nash2}):

\begin{thm}[Nash number of the Clifford algebras $\R_{p,q}$]
For each $p,q\in\N$ with $n:=p+q\geq 2$, it holds
$$
\Nn(\R_{p,q})\leq \frac{n^3+5n+6}{6}.
$$
\end{thm}

\begin{thm}[Nash number of the Clifford algebras $\R_n$]
Let $n\geq 2$ be an integer. The Clifford algebra $\R_n$ satisfies:
\begin{itemize}
\item[{\rm(i)}] $\Nn(\R_2)=\Nn(\HH)=1$.
\item[{\rm(ii)}] $\Nn(\R_3)=\Nn(\R_4)=2$.
\item[{\rm(iii)}] If $n\equiv 1 \text{ mod } 4$, then $\Nn(\R_n)=2$.
\item[{\rm(iv)}] If $n\geq 6$ and $n\equiv 0,2,3 \text{ mod } 4$, then $3\leq \Nn(\R_n)\leq 4$.
\end{itemize}
\end{thm}

\begin{remark}
In this article, we have used Cartan's name in many of our fundamental defi\-nitions. The reason is that this article, like the articles \cite{c, cs}, was inspired by a brilliant calculation of Cartan \cite[{\S IV.3.5}]{ca}. $\sqbullet$
\end{remark}

\subsection{Structure of the article} We summarise here the organisation of the article. In order to keep this work as self-contained as possible, in \S\ref{prelimi}, we collect the notions and results that we need in the remaining sections. In \S\ref{S3}, we introduce Cartan $*$-algebras and study their main properties. In \S\ref{S4}, we introduce the systems of Cartan pairs, which are combinatorial objects associated to the involved algebra $A$. Moreover, we show that systems of Cartan pairs always exist for Cartan $*$-algebras. In \S\ref{S5}, we study the relationship between the systems of Cartan pairs and slice regular functions. We introduce the Cartan transform of a slice regular function and show that it encodes many of the properties of the original slice regular function. In particular, the Cartan transform will allow us to `reconstruct' a slice regular function from some of its real components and to show a version of the maximum module principle for slice regular functions defined on certain domains of Cartan $*$-algebras. At the end of \S\ref{S5}, we modify the definition of system of Cartan pairs to introduce Schwarz systems of Cartan pairs. This will allow us to `reconstruct' more explicitly a slice regular function from some of its real components and to show a version of the splitting lemma for Cartan $*$-algebras. In \S\ref{S6}, we define the Cartan number of a Cartan $*$-algebra and we show its relationship with the reconstructing number. In addition, we show that the only Cartan $*$-algebra with Cartan number equal to $1$  are the division algebras $\C$, $\HH$ and $\O$, and we provide a sufficient condition in order to guarantee that $\Rr(A)\geq 2$. Finally, we define the Cartan graph of a Cartan $*$-algebra and show that the dominating number of this graph is equal to the Cartan number of the involved algebra. In \S\ref{Scliff}, we study the reconstructing number of the Clifford algebras $\R_{p,q}$. We start by finding `rough' estimations for the general case and then we significantly improve them for the particular case of the Clifford algebras $\R_n:=\R_{0,n}$. In \S\ref{S8}, we propose an application of our results to slice-Nash functions, namely, we make use of the techniques developed in this paper to address the problem of determining how many real components of a slice regular function have to be real Nash functions in order to guarantee that the involved slice regular function is a slice-Nash function.

\section{Preliminaries}\label{prelimi}

In this section, we collect some preliminary concepts and results that will be used freely along this article. We include them for the sake of completeness and to ease the reading of the article. 

\subsection{Alternative $*$-algebras}\label{Salg}

Let $A$ be a finite dimensional (real) algebra with unity $1\neq 0$. We identify the field of real number $\R$ with the subalgebra of $A$ generated by $1$. The algebra $A$ is \textit{alternative} if the \textit{associator} of three elements $(x,y,z):=(xy)z-x(yz)$ is an alternating function of its arguments. Equivalently, $A$ is alternative if and only if $x(xy)=x^2y$ and $(xy)y=xy^2$ for each $x,y\in A$. Clearly, if $A$ is associative, then it is also alternative. If $A$ is alternative, then the subalgebra generated by two elements of $A$ is associative:

\begin{thm}[Artin's theorem]
If $A$ is alternative, then the subalgebra generated by two elements $x$ and $y$ is associative for each $x,y\in A$.
\end{thm}

Alternative algebras also satisfy the so-called \textit{(middle) Moufang identity} 
\begin{equation}\label{Mou}
(xy)(zx)=x(yz)x
\end{equation}
for each $x,y,z\in A$.

An algebra $A$ is called a \textit{$*$-algebra} if it is endowed with an involutive real linear map $x\mapsto x^c$, called \textit{$*$-involution}, that satisfies the following properties:
\begin{itemize}
\item[{\rm(i)}] $(x^c)^c=x$ for each $x\in A$,
\item[{\rm(ii)}] $(xy)^c=y^cx^c$ for each $x,y\in A$,
\item[{\rm(iii)}] $x^c=x$ for each $x\in \R$.
\end{itemize}

Division algebras $\C$, $\HH$ and $\O$ endowed with the standard conjugations are examples of alternative $*$-algebras. 

\begin{example}[Division algebras]
The $*$-algebras of \textit{complex numbers} $\C$, \textit{quaternions} $\HH$ and \textit{octonions} $\O$ can be built from the field of real numbers $\R$ by means of the so-called Cayley-Dickson construction:
\begin{itemize}
\item $\C=\R+i\R$, $(\alpha+i\beta)(\gamma+i\delta):=(\alpha\gamma-\beta\delta)+i(\alpha\delta+\beta\gamma)$, $(\alpha+i\beta)^c:=\alpha-i\beta$ for each $\alpha,\beta,\gamma,\delta\in\R$.
\item $\HH:=\C+ j\C$, $(\alpha+j\beta)(\gamma+j\delta):=(\alpha\gamma-\delta\beta^c)+j(\alpha^c\delta+\beta\gamma)$, $(\alpha+\beta j)^c:=\alpha^c-j\beta $ for each $\alpha,\beta,\gamma,\delta\in \C$.
\item $\O:=\HH+ \ell\HH$, $ (\alpha+\ell\beta)(\gamma+\ell\delta):=(\alpha\gamma-\delta\beta^c)+\ell(\alpha^c\delta+\gamma\beta)$, $(\alpha+\ell\beta)^c:=\alpha^c-\ell\beta$ for each $\alpha,\beta,\gamma,\delta\in \HH$. $\sqbullet$
\end{itemize}
\end{example}

Complex numbers $\C$ and quaternions $\HH$ are associative, while octonions $\O$ are alternative, but not associative. Another example of associative $*$-algebra is given by the dual quaternions.

\begin{example}[Dual quaternions]\label{dualquaternions}
The associative $*$-algebra $\D\HH$ of \textit{dual quaternions} is defined as $\HH+\eps\HH$, where $(\alpha+\eps\beta)(\gamma+\eps\delta):=\alpha\gamma+\eps(\alpha\delta+\beta\gamma)$ and $(\alpha+\eps\beta)^c:=\alpha^c+\eps\beta^c$ for each $\alpha,\beta,\gamma,\delta\in \HH$. In particular, $\eps$ commutes with every element of $\D\HH$ and $\eps^2=0$. $\sqbullet$
\end{example} 

The algebra of split octionions is an example of alternative $*$-algebra, which is not associative. 

\begin{example}[Split octonions]\label{splitoctonions}
The alternative $*$-algebra $\mathbb{S}\O$ of \textit{split octonions} is defined as $\HH+\ell\HH$, where $(\alpha+\ell\beta)(\gamma+\ell\delta):=(\alpha\gamma+\delta\beta^c)+\ell(\alpha^c\delta+\gamma\beta)$ and $(\alpha+\ell\beta)^c:=\alpha^c-\ell\beta$ for each $\alpha,\beta,\gamma,\delta\in \HH$. $\sqbullet$
\end{example} 

On our alternative $*$-algebra $A$ we consider the \textit{trace} and the \textit{(squared) norm} defined, respe\-ctively, as $t(x):=x+x^c$ and $n(x):=xx^c$ for each $x\in A$. We call the elements of 
$$
\sph_{A}:=\{J\in A: t(J)=0, n(J)=1\}=\{J\in A: J^c=-J, J^2=-1\}.
$$
the \textit{imaginary units of} $A$. Let $J\in\sph_A$. The $*$-subalgebra generated by $J$, i.e.\ the complex plane $\C_J:=\R[J]$, is $*$-isomorphic to $\C$ through the $*$-isomorphism 
\begin{equation}\label{*iso}
\phi_J: \C\to \C_J, \quad \alpha+i \beta\mapsto \alpha+J\beta.
\end{equation} 
The complex plane $\C_J$ is said to be a \textit{(complex) slice} of $A$. The \textit{quadratic cone of} $A$ is the subset of $A$ defined as
$$\textstyle
\Qq_A:=\bigcup_{J\in\sph_A}\C_J=\bigcup_{J\in\sph_A}\phi_J(\C).
$$
Clearly, $\Qq_A\neq \varnothing$ if and only if $\sph_A\neq \varnothing$. Moreover, the quadratic cone $\Qq_A$ is property contained in $A$, unless $A$ is $*$-isomorphic to one of the division algebras $\C$, $\HH$ or $\O$ endowed with the standard conjugation \cite[Prop.1]{gp}. It holds $\C_J\cap \C_I=\R$ for each $I,J\in\sph_A$ such that $J\neq\pm I$. As a consequence, every element $x\in \Qq_A\setminus\R$ can be written uniquely as $x=\alpha+\beta J$, where $\alpha,\beta\in \R$, $\beta>0$ and $J\in\sph_A$. If $x\in \R$, then $x=\alpha$, $\beta=0$ and $J$ can be chosen arbitrarily in $\sph_A$. 

In what follows, we make the following:

\begin{assumption}\label{ass1}
We assume that $A$ is a finite dimensional alternative $*$-algebra with unity $1\neq 0$ such that $\sph_A\neq \varnothing$. Moreover, we endow $A$ with the natural topology and real analytic structure induced by its finite dimensional real vector space structure.
\end{assumption}

\subsection{Clifford algebras}\label{SClifford}
\textit{Let $n\geq 1$ be an integer. Denote by $\mc{P}(n)$ the power set of $\{1,\ldots,n\}$ and, for each $K\in\mc{P}(n)$, denote by $|K|$ the cardinality of $K$}. Let $p$ and $q$ be non-negative integers such that $n=p+q$. The \textit{Clifford algebra} $\R_{p,q}$ is the associative $*$-algebra constructed by taking the real vector space $\R^{2^n}$ with the following properties:
\begin{itemize}
\item $\{e_K\}_{K\in \mc{P}(n)}$ denotes the standard basis of $\R^{2^n}$,
\item $e_{\varnothing}:=1$,
\item if $K=\{k_1,\ldots,k_s\}$ with $1\leq k_1<\ldots<k_s\leq n$, then the element $e_K$ is also denoted as $e_{k_1\ldots k_s}$ and the product $e_{k_1}\cdots e_{k_s}$ is defined to be $e_K=e_{k_1\ldots k_s}$,
\item $e_k^2:=1$ for each $k\in\{1,\ldots,p\}$ and $e_k^2:=-1$ for each $k\in\{p+1,\ldots,n\}$,
\item $e_ke_h=-e_he_k$ for each $k,h\in \{1,\ldots,n\}$ with $k\neq h$,
\item $e_K^c=e_K$ if $|K|\equiv 0,3\text{ mod } 4$ and $e_K^c=-e_K$ if $|K|\equiv 1,2\text{ mod } 4$.
\end{itemize}
The $*$-involution $x\mapsto x^c$ is called \textit{Clifford conjugation}.
In what follows, we will denote the Clifford algebras $\R_{0,n}$ also as $\R_n$. 

The following example show that it is not restrictive, for our purposes, to restrict to the case $n=p+q\geq 2$.

\begin{example}\label{nogood}
(i) The Clifford algebra $\R_{0,1}$ is $*$-isomorphic to the complex numbers $\C$ endowed with the standard conjugation. 

(ii) The Clifford algebra $\R_{1,0}$ is isomorphic to the Clifford algebra $\mathbb{S}\C=\R+\ell\R$ of \textit{split complex number}. The product on $\mathbb{S}\C$ is defined as
$$
(\alpha+\ell\beta)(\gamma+\ell\delta):=(\alpha\gamma+\beta\delta)+\ell(\alpha\delta+\beta\gamma).
$$
for each $\alpha,\beta,\gamma,\delta\in\R$. In particular, $\{J\in\mathbb{S}\C: J^2=-1\}=\varnothing$, so the Clifford algebra $\R_{1,0}=\mathbb{S}\C$ is not suitable for our purposes. $\sqbullet$
\end{example}

In the following example, we classify the Clifford algebras $\R_{p,q}$ with $n=p+q=2$.

\begin{example}
(i) The Clifford algebra $\R_2=\R_{0,2}$ is $*$-isomorphic to the division algebra $\HH$ of quaternions endowed with the standard conjugation.

(ii) The Clifford algebra $\R_{2,0}$ is $*$-isomorphic to the Clifford algebra $\R_{1,1}$. Both the previous $*$-algebras are $*$-isomorphic to the algebra $\mathbb{S}\HH=\C+\ell\C$ of \textit{split quaternions} endowed with the $*$-involution defined as $(\alpha+\ell\beta)^c:=\overline{\alpha}-\ell\beta$ for each $\alpha, \beta\in\C$, where $\bar{\cdot}$ denotes the conjugation of $\C$. The product on $\mathbb{S}\HH$ is defined as
$$
(\alpha+\ell\beta)(\gamma+\ell\delta):=(\alpha\gamma+\overline{\beta}\delta)+\ell(\overline{\alpha}\delta+\beta\gamma)
$$
for each $\alpha,\beta,\gamma,\delta\in\C$. $\sqbullet$
\end{example}

Next, we want to characterise which elements of the basis $\{e_K\}_{K\in \mc{P}(n)}$ belong to $\sph_{\R_{p,q}}$. This will be useful for us in \S\ref{Scliff}. In order to lighten the notation, given an integer $m\geq 0$, we will denote by $\m{m}$ the remainder of $m$ divided by 4. Moreover, given $K\in\mc{P}(n)$, we will simply write $\m{K}$ instead of $\m{|K|}$ to denote the remainder of $|K|$ of divided by 4. For each $K\in\mc{P}(n)$, we define
\begin{equation}\label{+-}
K_+:=K\cap\{1,\ldots,p\} \quad \text{and} \quad K_-:=K\cap\{p+1,\ldots,n\}.%\{k\in K : 1\leq k\leq p\} \quad \text{and} \quad K_-:=\{k\in K : p+1\leq k\leq n\}.
\end{equation}
We start by determining which elements of the basis  $\{e_K\}_{K\in \mc{P}(n)}$ satisfy $e_K^2=-1$. 

\begin{lem}\label{Clifford1}
Let $K\in\mc{P}(n)$. Then, $e_K^2=-1$ if and only if 
$$
(\m{K_+},\m{K_-})\in\{(0,1), (0,2), (1,2), (1,3), (2,0), (2,3), (3,0), (3,1)\}.
$$
\end{lem}
\begin{proof}
Let $K\in\mc{P}(n)$. We may assume that $K\neq\varnothing$. Assume, in addition, that $K=\{k_1,\ldots,k_s\}$, with $1\leq k_1<\ldots<k_s\leq n$. As $\R_{p,q}$ is associative and $e_ke_h+e_he_k=0$ for each $k,h\in\{1,\ldots,n\}$ with $k\neq h$, using a simple inductive argument on $|K|=s\geq 1$, we have
$$
e_K^2=(e_{k_1}\cdots e_{k_s})^2=e_{k_1}\cdots e_{k_s}e_{k_1}\cdots e_{k_s}=(-1)^{\frac{|K|(|K|-1))}{2}}e_{k_1}^2\cdots e_{k_s}^2.
$$
As $e_k^2=1$ for each $k\in K_+$ and $e_k^2=-1$ for each $k\in K_-$, by the previous equality, we deduce that
$$
e_K^2=(-1)^{\frac{(|K_+|+|K_-|)(|K_+|+|K_-|-1)}{2}+|K_-|}=(-1)^{\frac{(|K_+|+|K_-|)(|K_+|+|K_-|-1)+2|K_-|}{2}}.
$$
In particular, the value of $e_K^2$ depends only by the pair $(\m{K_+},\m{K_-})$ and the statement follows by inspecting all the possible cases one by one.
\end{proof}

By the previous lemma and the fact that $e_K^c=-e_K$ if and only if $\m{K}\in\{1,2\}$, we deduce straightforwardly the following:

\begin{cor}\label{Clifford2}
Let $K\in\mc{P}(n)$. Then, $e_K\in\sph_{\R_{p,q}}$ if and only if 
$$
(\m{K_+},\m{K_-})\in\{(0,1),(0,2),(2,0),(2,3)\}.
$$
\end{cor}

Specialising the previous result to $(p,q)=(0,n)$, we obtain the following:

\begin{cor}\label{Clifford3}
Let $K\in\mc{P}(n)$. Then, $e_K\in\sph_{\R_n}$ if and only if $\m{K}\in\{1,2\}$. In particular, $e_K\in\sph_{\R_n}$ if and only if $e_K^2=-1$.
\end{cor}

The Clifford algebra $\R_{p,q}$ can be endowed also with (at least) another $*$-involution.

\begin{example}[Reversion]\label{exa:reversion}
The Clifford algebra $\R_{p,q}$ can be endowed with the $*$-involution defined on the elements of the basis as
$$
e_K^{c}:=
\left\{
 \begin{array}{rl}
 e_K & \text{ if } |K|\equiv 0,1 \text{ mod } 4, \\
 -e_K & \text{ if } |K|\equiv 2,3 \text{ mod } 4.
 \end{array}
\right.
$$
The $*$-involution $x\mapsto x^c$ is called \textit{reversion}. $\sqbullet$
\end{example}

In what follows, except otherwise stated, we will always assume that $\R_{p,q}$ is endowed with the Clifford conjugation.

\subsection{Slice functions}\label{S2s}

\textit{Let $A$ be an alternative $*$-algebra}. We endow the complexified algebra $A\otimes_{\R}\C$ with the product defined as
$$
(x+\iota y)\cdot(x'+\iota y'):=xx'-yy'+\iota(xy'+yx')
$$
for each $x,x',y,y'\in A$. The algebra $A\otimes_{\R}\C$ is endowed with two commutating $*$-involutions
\begin{itemize}
\item the \textit{complex conjugation} $w:=x+\iota y\mapsto \overline{w}:=x-\iota y$,
\item the \textit{complex $*$-involution} $w:=x+\iota y\mapsto w^c:=x^c+\iota y^c$. 
\end{itemize}
As $A$ is alternative, $A\otimes_{\R}\C$ is also alternative. For each $J\in\sph_A$, the $*$-isomorphism $\phi_J:\C\to \C_J$ introduced in \eqref{*iso} extends to 
$$
\widetilde{\phi}_J:A\otimes_\R\C\to A,\quad x+\iota y\mapsto x+Jy,
$$
which is still a surjective $*$-algebra morphism, but no longer injective.

For a (non-empty) set $D\subset \C$, we define:
$$\textstyle
\Omega_D:=\{\alpha+\beta J\in A : \alpha,\beta\in\R, \alpha+i \beta \in D, J\in\sph_A\}=\bigcup_{J\in \sph_A}\phi_J(D).
$$
A subset $\Omega$ of $A$ is called  \textit{circular} if there exists $D\subset \C$ such that $\Omega=\Omega_D$. In literature, a circular set is also called axially symmetric (with respect to the real axis). A circular set $\Omega=\Omega_D\subset A$ is a \textit{circular slice domain} if $D$ is open and connected in $\C$, and $D\cap\R\neq\varnothing$.

Let us recall the definition of stem function.

\begin{defn}[Stem function]
Let $D\subset \C$ be an open set. A function $F:D\to A\otimes_\R\C$ is called a \textit{stem function on} $D$ if it is complex intrinsic in the sense that
\begin{equation}\label{compint}
\overline{F(z)}=F(\overline{z})
\end{equation}
for each $z\in D$ such that $\overline{z}\in D$. $\sqbullet$
\end{defn}

If $D$ is symmetric with respect to the real axis, then \eqref{compint} holds for each $z\in D$.

\begin{assumption}\label{assSimm}
In what follows, even if not explicitly mentioned,  we will always assume that the subset $D$ of $\C$ is non-empty and symmetric with respect to the real axis.
\end{assumption}

Observe that $F=F_1+\iota F_2:D\to A\otimes_\R\C$ is a stem function if and only if the $A$-valued components $F_1$ and $F_2$ form an \textit{even-odd pair} with respect to the imaginary part of $z$, that is, if and only if $F_1(\overline{z})=F_1(z)$ and $F_2(\overline{z})=-F_2(z)$ for each $z=\alpha+i\beta\in D$, where $\overline{z}=\alpha-i\beta$. We define the \textit{conjugated of} $F$, that we denote by $F^c$, as the function defined as $F^c(z):=(F(z))^c$ for each $z\in D$.  Observe that $F$ is a stem function if and only if $F^c$ is a stem function, because the two $*$-involutions $w\mapsto\overline{w}$ and $w\mapsto w^c$ of $A\otimes_\R\C$ commute.

Next, we recall the definition of slice function.

\begin{defn}[Slice function]
Let $D\subset \C$ be an open set. A function $f:\Omega_D\to A$ is a (\textit{left}) \textit{slice function} if there exists a stem function $F:D\to A\otimes_\R\C$ such that the diagram
$$
\xymatrix{
D\ar[d]_{{\phi_J}|_{D}} \ar[r]^{F}&A\otimes_\R\C \ar[d]^{\widetilde{\phi}_J}\\
\Omega_D\ar[r]^{f}& A
}
$$
commutes for each $J\in\sph_A$. In this case, we say that the slice function $f$ is \textit{induced by} the stem function $F$ and we write $f=\mathcal{I}(F)$. $\sqbullet$
\end{defn}

Let us recall the following \textit{representation formula for slice functions} that allows one to completely reconstruct a slice function by its values on a single complex slice.

\begin{prop}[Representation formula {\cite[Prop.6]{gp}}]\label{6}
Let $\Omega_D\subset A$ be an open circular set and $f:\Omega_D\to A$ a slice function. Then, for each $I\in \sph_A$, we have
$$\textstyle
f(\alpha+\beta J)=\frac{1}{2}(f(\alpha+\beta I)+f(\alpha-\beta I))-\frac{1}{2}J(I(f(\alpha+\beta I)-f(\alpha-\beta I)))
$$
for each $\alpha,\beta\in\R$ and each $J\in \sph_A$ such that $\alpha+\beta J\in\Omega_D$. In particular, $f$ is induced by a unique stem function $F=F_1+\iota F_2$, where $F_1(\alpha+i\beta):=\frac{1}{2}(f(\alpha+\beta I)+f(\alpha-\beta I))$ and $F_2(\alpha+i\beta):=-\frac{1}{2}I(f(\alpha+\beta I)-f(\alpha-\beta I))$ for each $\alpha,\beta\in\R$ with $\alpha+i\beta\in D$.
\end{prop}

The $*$-algebra structure of $A\otimes_\R\C$ induces a $*$-algebra structure on the set of all stem functions on $D$. Let $F,G:D\to A\otimes_\R\C$ be stem functions and $f:=\mc{I}(F)$ and $g:=\mc{I}(G)$. We define $f^c:=\mc{I}(F^c)$ and $f\cdot g:=\mc{I}(F\cdot G)$. In this way, the map $\mc{I}$ is a $*$-algebra isomorphism between the set of all stem functions on $D$ and the set of all slice functions on $\Omega_D$.

\subsection{Slice regular functions}\label{S2sr}

Let $D\subset \C$ be an open set. We say that a stem function $F=F_1+\iota F_2:D\to A\otimes_\R\C$ is a \textit{holomorphic stem function on $D$} if it is of class $\Cont^1$ on $D$ and satisfies the following Cauchy-Riemann equation:
$$\textstyle
\frac{\partial F}{\partial\overline{z}}:=\frac{1}{2}\big(\frac{\partial F}{\partial\alpha}+\iota \frac{\partial F}{\partial\beta}\big)=0.
$$
Slice regular functions are defined as those slice functions induced by holomorphic stem functions. 

\begin{defn}[Slice regular function]
We say that the slice function $f=\mathcal{I}(F)$ is \textit{slice regular on} $\Omega_D$ if the stem function $F$ is holomorphic on $D$. $\sqbullet$
\end{defn}

We end this section with the following:

\begin{remark}\label{Cullen}
Let $\Omega\subset\Qq_A$ be an open circular set and $f:\Omega\to A$ a slice regular function. By \cite[Prop.8]{gp}, we know that $f$ is Cullen regular, i.e., $\big(\frac{\partial}{\partial\alpha}+J\frac{\partial}{\partial \beta}\big)f(\alpha+\beta J)=0$ for each $\alpha,\beta\in\R$ and each $J\in \sph_A$ such that $\alpha+\beta J\in\Omega$. $\sqbullet$
\end{remark}

\subsection{Real and complex Nash functions}

In this section, we recall the definition and the main properties of Nash functions. We refer the reader to \cite{bcr, t} and the references mentioned therein for further informations. Let $k$ be either the field of real numbers $\R$ or the field of complex numbers $\C$. Let $n\geq 1$ be an integer. Let $D\subset k^n$ be a (non-empty) open set and $f:D\to k$ a $k$-valued function. 

\begin{defn}\label{defNash}
We say that $f$ is a \textit{$k$-Nash function at $x_0\in D$} if there exist an open neighbourhood $U$ of $x_0$ in $D$ and a non-zero polynomial $P\in k[\x,\t]:=k[\x_1,\ldots,\x_n,\t]$ such that 
\begin{itemize}
\item $f$ is $k$-analytic on $U$,
\item $P(x,f(x))=0$ for each $x\in U$. 
\end{itemize}
The function $f$ is a \textit{$k$-Nash function on $D$} if it is a $k$-Nash function at every point of $D$. $\sqbullet$
\end{defn}

We will also call $\R$-Nash functions \textit{real Nash functions} and $\C$-Nash functions \textit{complex Nash functions}. By definition, if a function $f$ is $k$-Nash on $D$, then $f$ is also $k$-analytic on $D$. 

Let $f:D\to k$ be a $k$-analytic function. As for each polynomial $P\in k[\x,\t]$ the function $D\to k, \, x\mapsto P(x,f(x))$ is $k$-analytic, we deduce straightforwardly that, if $D$ is connected, then the following conditions are equivalent:
\begin{itemize}
\item[{\rm (i)}] $f$ is a $k$-Nash function on $D$,
\item[{\rm (ii)}] there exists $x_0\in D$ such that $f$ is $k$-Nash at $x_0$,
\item[{\rm (iii)}] there exists an irreducible polynomial $P\in k[\x,\t]$ such that $P(x,f(x))=0$ on $D$. Moreover, if $k=\C$, the polynomial $P$ is unique up to a non-zero constant (as a consequence of Hilbert's Nullstellensatz).
\end{itemize}
In particular, if $D$ has finitely many connected components, then $f$  is a $k$-Nash function on $D$ if and only if there exists a non-zero polynomial $P\in k[\x,\t]$ such that $P(x,f(x))=0$ for each $x\in D$. If $D$ has infinitely many connected components, $f$ might not be algebraic over the ring $k[\x,\t]$. That is, in general, the existence of a non-zero polynomial $P\in k[\x,\t]$ such that $P(x,f(x))=0$ for each $x\in D$ is not guaranteed \cite[Ex.2.2]{c}.

%%%

\section{Cartan $*$-algebras}\label{S3}%S3

Let $n\geq 1$ be an integer. For each $K,H\in\mc{P}(n)$, we denote by $K\ds H:=(K\setminus H)\cup(H\setminus K)$ the (usual) \textit{symmetric difference} between $K$ and $H$. It is worthwhile to notice that $\mc{P}(n)$ endowed with the symmetric difference $\ds$ is a commutative group. 

Let us recall the definition of $\ds$-algebra introduced in \cite[Def.2.28]{gp2}. \textit{Fix a (real vector) basis $\{e_K\}_{K\in\mc{P}(n)}$ of $\R^{2^n}$}. A real bilinear map $\rm{b}:\R^{2^n}\times\R^{2^n}\to \R^{2^n}$ is a \textit{$\ds$-product on $\R^{2^n}$} if there exists a function $\sigma:\mc{P}(n)\times \mc{P}(n)\to \R$ such that
\begin{itemize}
\item[{\rm(i)}] $\sigma(K,\varnothing)=\sigma(\varnothing,K)=1$ for each $K\in \mc{P}(n)$, 
\item[{\rm(ii)}] ${\rm b}(e_K, e_H)=\sigma(K,H)e_{K\ds H}$ for each $K,H\in \mc{P}(n)$.
\end{itemize}
In this case, we say that the $\ds$-product $\rm{b}$ on $\R^{2^n}$ is \textit{induced by $\sigma$}. Observe that properties $({\rm i})$ and $({\rm ii})$ imply that $e_{\varnothing}$ is the identity of $\R^{2^n}$, so we write $e_{\varnothing}=1$. If $K=\{k_1,\ldots,k_s\}$ with $1\leq k_1<\ldots<k_s\leq n$, then $e_K$ is also denoted as $e_{k_1\ldots k_s}$ (this notation is consistent with the one introduced in \S\ref{SClifford}). For short, we write $vw$ in place of ${\rm b}(v,w)$ for all $v,w\in\R^{2^n}$.

\begin{defn}[{$\ds$-algebra}]\label{def:delta-algebra}
A (real) algebra $A$ is a \textit{$\ds$-algebra} if it is isomorphic to $\R^{2^n}$ equipped with the $\ds$-product induced by some function $\sigma$. If $\varphi:A\to\R^{2^n}$ is such an isomorphism, we denote the element $\varphi^{-1}(e_K)$ of $A$ still with the symbol $e_K$ for every $K\in\mc{P}(n)$, and we say that the $\ds$-algebra $A$ is \textit{equipped with the basis $\{e_K\}_{K\in\mc{P}(n)}$}, and the product of~$A$ is \textit{induced by $\sigma$}. $\sqbullet$
\end{defn}

The algebras $\HH$, $\O$, $\R_{p,q}$, $\D\HH$ and $\mathbb{S}\O$ are all $\ds$-algebras.

\begin{example}\label{basis}
(i) The division algebra of quaternions $\HH$ endowed with the basis
$$
\{e_{\varnothing}:=1,\, e_1:=j,\, e_2:=i,\, e_{12}:=ji\}
$$
is an associative $\ds$-algebra. While, the division algebra of octonions $\O$ endowed with the basis
$$
\{e_{\varnothing}:=1,\, e_1:=\ell,\, e_2:=j,\, e_3:=i,\, e_{12}:=\ell j,\, e_{13}:=\ell i,\, e_{23}:=ji,\, e_{123}:=\ell(ji)\}
$$
is an alternative, but not associative, $\ds$-algebra. 

(ii) The Clifford algebras $\R_{p,q}$ with basis $\{e_K\}_{K\in\mc{P}(p+q)}$ are all associative $\ds$-algebras.

(iii) The algebra of dual quaternions $\D\HH$ with basis
$$
\{e_{\varnothing}:=1,\, e_1:=\eps,e_2:=j,\, e_3:=i ,\, e_{12}:=\eps j,\, e_{13}:=\eps i,\, e_{23}:=ji,\, e_{123}:=\eps ji\}
$$
is an associative $\ds$-algebra. 

(iv) The algebra of split octonions $\mathbb{S}\O$ with basis 
$$
\{e_{\varnothing}:=1,\, e_1:=\ell,\, e_2:=j,\, e_3:=i,\, e_{12}:=\ell j,\, e_{13}:=\ell i,\, e_{23}:=ji,\, e_{123}:=\ell(ji)\}
$$
is an alternative, but not associative, $\ds$-algebra. $\sqbullet$
\end{example}

The reader observes that the bases described in the previous example are the ones given by the Cayley-Dickson construction.

\begin{assumption}\label{bases}
In what follows, we will assume that the $\ds$-algebras $\HH$, $\O$, $\R_{p,q}$, $\D\HH$ and $\mathbb{S}\O$ are equipped with the bases described in Example \ref{basis}.
\end{assumption}

For our purposes, we need to strengthen the definition of $\ds$-algebra as follows:

\begin{defn}[Cartan algebra]\label{def:cartan-algebras}
Let $A$ be a $\ds$-algebra whose product is induced by a fun\-ction~$\sigma$. We say that $A$ is a \textit{Cartan algebra} if $\sigma(K,H)\in\{-1,1\}$ for each $K,H\in\mc{P}(n)$. $\sqbullet$
\end{defn}

The properties of the Cartan algebra $A$ are encoded by $\sigma$. For instance, if $\{e_K\}_{K\in\mc{P}(n)}$ is the fixed basis of $A$, then $A$ is commutative if and only if $\sigma(K,H)\sigma(H,K)=1$ for each $K,H\in \mc{P}(n)$, because
$$
e_Ke_H=\sigma(K,H)e_{K\ds H}=\sigma(K,H)e_{H\ds K}=\sigma(K,H)\sigma(H,K)e_He_K.
$$
Analogously, it is straightforward to check that $A$ is associative if and only if 
$$
\sigma(K,H)\sigma(K\ds H,W)=\sigma(K,H\ds W)\sigma(H,W)
$$
for each $K,H,W\in\mc{P}(n)$. Recall that an algebra $B$ is alternative if and only if $x(xy)=x^2y$ and $(yx)x=yx^2$ for each $x,y\in B$. Thus, the Cartan algebra $A$ is alternative if and only if
\begin{align}\label{alt}
\begin{split}
\sigma(K,K\ds H)\sigma(K,H)&=\sigma(K,K)\\
\sigma(H,K)\sigma(K\ds H,K)&=\sigma(K,K)
\end{split}
\end{align}
for each $K,H\in\mc{P}(n)$.

The only Cartan algebras $A$ such that $\dim(A)=2$ are the Clifford algebras of Example \ref{nogood}, namely, the complex numbers $\C$ and the split complex numbers $\mathbb{S}\C$. That is why, in what follows, it is not restrictive to assume that $A$ is isomorphic to $\R^{2^n}$ with $n\geq 2$ (i.e. $\dim(A)\geq 4$).

\begin{example}
The algebras $\HH$, $\O$, $\mathbb{S}\O$, and $\R_{p,q}$ with $p+q\geq 2$ are all Cartan algebras. While the algebra of dual quaternions $\D\HH$ is a $\ds$-algebra which is not a Cartan algebra. $\sqbullet$
\end{example}

Cartan algebras enjoy many nice properties:

\begin{lem}\label{lemDelta}
Let $A$ be a Cartan algebra. The following hold:
\begin{itemize}
\item[{\rm(i)}] $e_Ke_H\in\{\pm e_He_K\}$ for each $K,H\in\mc{P}(n)$.
\item[{\rm(ii)}] $(e_Ke_H)e_W\in\{\pm e_K(e_H e_W)\}$ for each $K,H,W\in \mc{P}(n)$.
\item[{\rm(iii)}] $e_K^2\in\{-1,1\}$ for each $K\in\mc{P}(n)$.
\item [{\rm(iv)}] The set $\{\pm e_K\}_{K\in\mc{P}(n)}$ is an invertible loop (i.e. a `non-associative group').
\end{itemize}
\end{lem}
\begin{proof}
As $A$ is a Cartan algebra, we have that $\sigma(K,H)\in\{-1,1\}$ for each $K,H\in\mc{P}(n)$. Property (i) follows by the fact that $e_Ke_H=\sigma(K,H)\sigma(H,K)e_He_K$ for each $K,H\in\mc{P}(n)$, and property (ii) by the fact that 
\begin{align*}
(e_Ke_H)e_W&=\sigma(K,H)e_{K\ds H}e_W=\sigma(K,H)\sigma(K\ds H,W)e_{(K\ds H)\ds W}\\
&=\sigma(K,H)\sigma(K\ds H,W)e_{K\ds(H\ds W)}\\
&=\sigma(K,H)\sigma(K\ds H,W)\sigma(K,H\ds W)e_Ke_{H\ds W}\\
&=\sigma(K,H)\sigma(K\ds H,W)\sigma(K,H\ds W)\sigma(H,W) e_K(e_He_W)
\end{align*}
for each $K,H,W\in \mc{P}(n)$. As $e_K^2=\sigma(K,K)e_\varnothing=\sigma(K,K)$ for each $K\in\mc{P}(n)$, we also have property (iii). At this point, property (iv) is straightforward. 
\end{proof}

Next, we introduce $*$-involutions on Cartan algebras compatible with the $\ds$-algebra structure. 

\begin{defn}[Cartan $*$-algebra]
Let $A$ be a Cartan algebra whose product is induced by a fun\-ction~$\sigma$, and let $\{e_K\}_{K\in\mc{P}(n)}$ be the fixed basis of $A$. We say that $A$ is a \textit{Cartan $*$-algebra} if $A$ is also endowed with a $*$-involution $x\mapsto x^c$ such that $e_k^c\in\{\pm e_k\}$ for each $k\in\{1,\ldots,n\}$.~$\sqbullet$
\end{defn}

There are several examples of Cartan $*$-algebras.

\begin{example}\label{splitoctonions2}
(i) The division algebras $\HH$ and $\O$ endowed with the standard conjugations are Cartan $*$-algebras.  

(ii) The Clifford algebras $\R_{p,q}$ with $p+q\geq 2$ are Cartan $*$-algebras whether they are equipped with the Clifford conjugation or the reversion.

(iii) The split octonions $\mathbb{S}\O=\HH+\ell\HH$ endowed with the $*$-involution introduced in Example~\ref{splitoctonions} is a Cartan $*$-algebra. The algebra $\mathbb{S}\O$ can also be endowed with the $*$-involution defined as $(\alpha+\ell\beta)^c:=\alpha^c+\ell\beta$ for each $\alpha,\beta\in\HH$.  Also with this $*$-involution $\mathbb{S}\O$ is a Cartan $*$-algebra. $\sqbullet$
\end{example}

By Lemma \ref{lemDelta}(i), we deduce the following:

\begin{lem}\label{invCart}
If $A$ is a Cartan $*$-algebra, then $e_K^c\in\{\pm e_K\}$ for each $K\in\mc{P}(n)$.
\end{lem}
\begin{proof}
Clearly, the statement is true if $K=\varnothing$. Let $K\in\mc{P}(n)\setminus\{\varnothing\}$. Write $K=\{k_1,\ldots,k_s\}$ with $1\leq k_1<\ldots<k_s\leq n$. Let us proceed by induction on $s=|K|\geq 1$. If $s=1$, then $e_{k_1}^c\in\{\pm e_{k_1}\}$ by definition of Cartan $*$-algebra. Assume that $s\geq 2$. By Lemma \ref{lemDelta}(i), there exists $\eps_1\in\{-1,1\}$ such that $e_{K\setminus\{k_1\}}e_{k_1}=\eps_1 e_{k_1}e_{K\setminus\{k_1\}}$. By inductive hypothesis, there exists $\eps_2, \eps_3\in\{-1,1\}$ such that $e_{K\setminus\{k_1\}}^c=\eps_2 e_{K\setminus\{k_1\}}$ and $e_{k_1}^c=\eps_3 e_{k_1}$. We deduce that
\begin{align*}
e_K^c&=\sigma(\{k_1\}, K\setminus\{k_1\})(e_{k_1}e_{K\setminus\{k_1\}})^c=\sigma(\{k_1\}, K\setminus\{k_1\})e_{K\setminus\{k_1\}}^ce_{k_1}^c\\
&=\sigma(\{k_1\}, K\setminus\{k_1\})\eps_2\eps_3e_{K\setminus\{k_1\}}e_{k_1}=\sigma(\{k_1\}, K\setminus\{k_1\})\eps_1\eps_2\eps_3e_{k_1}e_{K\setminus\{k_1\}}\\
&=\sigma(\{k_1\}, K\setminus\{k_1\})^2\eps_1\eps_2\eps_3e_K=\eps_1\eps_2\eps_3e_K.
\end{align*}
As $\eps:=\eps_1\eps_2\eps_3\in\{-1,1\}$ and $e_K^c=\eps e_K$, we conclude that $e_K^c\in\{\pm e_K\}$, as required.
\end{proof}

Constructing $*$-involutions on a given Cartan algebra $A$ is, in general, a difficult task. Thus, it is natural to ask the following:

\begin{quest2}
Let $A$ be a Cartan algebra. Does there exist a $*$-involution on $A$ that makes $A$ a Cartan $*$-algebra?
\end{quest2}

If the involved Cartan algebra $A$ is alternative, then the previous question has a positive answer.

\begin{prop}\label{involuzione}
Let $A$ be an alternative Cartan algebra. Then, there exists a $*$-involution on $A$ that makes $A$ a Cartan $*$-algebra. 
\end{prop}
\begin{proof}
We define 
\begin{equation}\label{*Cartan}
e_K^c:=(e_K^2)e_K=\sigma(K,K)e_K
\end{equation}
for each $K\in\mc{P}(n)$ and extend it to a map $\cdot^c:A\to A$ by linearity. As $A$ is a Cartan algebra, we have $\sigma(K,K)\in\{-1,1\}$ for each $K\in\mc{P}(n)$, so $(e_K^c)^c=e_K$ for each $K\in\mc{P}(n)$ and $e_k^c\in\{\pm e_k\}$ for each $k\in\{1,\ldots,n\}$. Thus, we only need to check that  $(xy)^c=y^cx^c$ for each $x,y\in A$. As the product of $A$ is a real bilinear map, it is enough to check the previous equality only on the element of the basis $\{e_K\}_{K\in\mc{P}(n)}$, i.e., we need to show that $(e_Le_M)^c=e_M^ce_L^c$ for each $L,M\in\mc{P}(n)$. Fix $L,M\in\mc{P}(n)$. We have:
\begin{align*}
(e_Le_M)^c&=\sigma(L,M)e_{L\ds M}^c=\sigma(L,M)\sigma(L\ds M,L\ds M)e_{L\ds M}\\
&=\sigma(L,M)^2\sigma(L\ds M,L\ds M)e_Le_M=\sigma(L\ds M,L\ds M)e_Le_M,
\end{align*}
while
$$
e_M^ce_L^c=\sigma(M,M)\sigma(L,L)e_Me_L=\sigma(M,M)\sigma(L,L)\sigma(M,L)\sigma(L,M)e_Le_M.
$$
We only need to show the following equality:
\begin{equation}\label{eq}
\sigma(L\ds M,L\ds M)=\sigma(M,M)\sigma(L,L)\sigma(M,L)\sigma(L,M).
\end{equation}
As $A$ is alternative, by \eqref{alt}, we have
$$
\sigma(M,M)\sigma(L,L)=\sigma(M,L)\sigma(L,M)\sigma(M,L\ds M)\sigma(L,L\ds M).
$$
Thus,
$$
\sigma(M,M)\sigma(L,L)\sigma(M,L)\sigma(L,M)=\sigma(M,L\ds M)\sigma(L,L\ds M).
$$
Taking $K:=L\ds M$, $H:=M$ (so $K\ds H=(L\ds M)\ds M=L$) in the second equality of \eqref{alt}, we deduce
$$
\sigma(L\ds M,L\ds M)=\sigma(M,L\ds M)\sigma(L,L\ds M).
$$
We conclude that the equality \eqref{eq} holds for each $L,M\in\mc{P}(n)$, as required. 
\end{proof}

On the algebras $\HH$, $\O$ and $\mathbb{S}\O$ the $*$-involution defined in the previous proposition gives nothing new. In fact, on $\HH$ and $\O$ the $*$-involution \eqref{*Cartan} coincides with the standard conjugation, while on $\mathbb{S}\O$ it coincides with $*$-involution defined in Example \ref{splitoctonions2}(iii). On the Clifford algebras $\R_{p,q}$ the situation is more interesting:

\begin{example}
The $*$-involution \eqref{*Cartan} coincides with the Clifford involution for the Clifford algebras $\R_n=\R_{0,n}$ and with the reversion for the Clifford algebras $\R_{n,0}$. While, for general Clifford algebras $\R_{p,q}$, with $p,q\geq 1$, it is a $*$-involution different from both the Clifford involution and the reversion. In fact, $e_{1n}^c=e_{1n}$ for the $*$-involution \eqref{*Cartan}, while, $e_{1n}^c=-e_{1n}$ both for the Clifford involution and the reversion. $\sqbullet$
\end{example}

In the following sections, we need the following:

\begin{prop}\label{sqrt-1}
Let $A$ be a Cartan $*$-algebra such that $\sph_A\neq \varnothing$. Then, there exists $K\in \mc{P}(n)$ such that $e_K\in\sph_{A}$.
\end{prop}
\begin{proof}
By Lemma \ref{invCart}, $e_K^c\in\{\pm e_K\}$ for each $K\in\mc{P}(n)$. Let $\Tt:=\{K\in \mc{P}(n) : e_K^c=-e_K\}$. If $x\in A\setminus\{0\}$ satisfies $t(x)=0$, then $\Tt\neq\varnothing$ and $
x=\sum_{K\in \Tt} x_K e_K$ for some $x_K\in \R$, because $\{e_K\}_{K\in\mc{P}(n)}$ is a basis of $A$ as a real vector space and the trace $t$ is a real linear map. As $\sph_A\neq \varnothing$, there exists $x\in \sph_A$, i.e., $
x=\sum_{K\in \Tt} x_K e_K$ for some $x_K\in \R$, and $x^2=-1$. The latter equation is equivalent to the following:
$$\textstyle
\sum_{K\in\Tt}x^2_K e_K^2+\sum_{K,H\in \Tt \, :\, K\neq H} x_Kx_H \sigma(K,H)e_{K\ds H}=-1.
$$
By Lemma \ref{lemDelta}(iii), $e_K^2\in\{-1,1\}$ for each $K\in\mc{P}(n)$. Thus, as $\mc{B}:=\{ e_K\}_{K\in\mc{P}(n)}$ is a basis of $A$ as real vector space and $e_{K\ds H}\in\mc{B}\setminus\{1\}$ for each $K,H\in\mc{P}(n)$ with $K\neq H$, we deduce 
$$\textstyle
\sum_{K\in\Tt}x^2_K e_K^2=-1.
$$
In particular, there exists $K\in \Tt$ such that $e^2_K=-1$, otherwise the left hand side of the previous equality would be non-negative. It follows that $e_K\in\sph_A$, as required.
\end{proof}

Let $A$ be a Cartan algebra such that $\{J\in A : J^2=-1\}\neq\varnothing$. Arguing as in the proof of the previous proposition, one can show that there exists $K\in\mc{P}(n)$ such that $e_K^2=-1$. In particular, by Proposition \ref{involuzione} and \eqref{*Cartan}, we deduce the following:

\begin{cor}\label{sqrtim}
Let $A$ be an alternative Cartan algebra such that $\{J\in A : J^2=-1\}\neq \varnothing$. Then, there exists a $*$-involution on $A$ that makes $A$ a Cartan $*$-algebra such that $\sph_A\neq \varnothing$.
\end{cor}

It is worthwhile to notice that $\{J\in A : J^2=-1\}=\varnothing$ if and only if the involved Cartan algebra $A$ is commutative and $e_k^2=1$ for each $k\in\{1,\ldots,n\}$.

%%%

\section{Systems of Cartan pairs}\label{S4}

\begin{assumption}\label{assumption}
For the rest of this article, except otherwise stated, $A$ is an alternative Cartan $*$-algebra, equipped with the basis $\{e_K\}_{K\in\mc{P}(n)}$ whose product is induced by a function $\sigma$ and such that $\sph_A\neq \varnothing$ (see Corollary \ref{sqrtim}).
\end{assumption}

\subsection{Cartan pairs and their systems}\label{cartanssystems}
We begin by introducing the concept of Cartan pair.

\begin{defn}[Cartan pair]\label{def:CP}
Let $L,M\in\mc{P}(n)$. We say that the pair $(e_L,e_M)$ is a \textit{Cartan pair for $A$} if $e_Le_M\in \sph_A$ or, equivalently, $e_{L\ds M}\in\sph_A$. $\sqbullet$
\end{defn}

The definition of Cartan pair depends on the involved $*$-involution.

\begin{example}
The pair $(1,e_1)$ is a Cartan pair for the Clifford algebra $\R_3$ endowed with the Clifford conjugation, while it is not a Cartan pair for $\R_3$ endowed with the reversion (see Example~\ref{exa:reversion}). In fact, with respect to the reversion, the trace of $e_1$ is $2e_1\neq 0$. $\sqbullet$
\end{example}

For each subset $\Cc$ of $\{(e_L,e_M)\}_{L,M\in\mc{P}(n)}$, we define the sets
\begin{align}\label{partition}
\begin{split}
&\Cc_1:=\{L\in \mc{P}(n): \exists M\in\mc{P}(n) \text{ such that } (e_L,e_M)\in \Cc\},\\
&\Cc_2:=\{M\in \mc{P}(n):\exists L\in\mc{P}(n) \text{ such that } (e_L,e_M)\in \Cc\}.
\end{split}
\end{align}
For each $L,M\in\mc{P}(n)$, we also define the sets
\begin{equation}\label{partition2}
\Cc_{1,M}:=\{L\in \mc{P}(n): (e_L,e_M)\in \Cc\}\quad \text{and}\quad \Cc_{2,L}:=\{M\in \mc{P}(n): (e_L,e_M)\in \Cc\}.
\end{equation}
Clearly, $\Cc_{1,M}\subset \Cc_1$ and $\Cc_{2,L}\subset \Cc_2$ for each $L,M\in\mc{P}(n)$.

We now introduce the notion of system of Cartan pairs.

\begin{defn}[System of Cartan pairs]\label{SCP}
We say that $\Cc\subset \{(e_L,e_M)\}_{L,M\in\mc{P}(n)}$ is a \textit{system of Cartan pairs of $A$}, for short \textit{SCP of $A$}, if it satisfies the following two conditions:
\begin{itemize}
\item[{\rm(i)}] $\{\Cc_1, \Cc_2\}$ is a partition of $\mc{P}(n)$, that is, $\Cc_1\cup \Cc_2=\mc{P}(n)$ and $\Cc_1\cap \Cc_2=\varnothing$.
\item[{\rm(ii)}] If $(e_L,e_M)\in \Cc$, then $(e_L,e_M)$ is a Cartan pair. $\sqbullet$
\end{itemize}
\end{defn}

The division algebras $\HH$ and $\O$ admit a SCP $\Cc$ with $|\Cc_1|=1$ (see Assumption \ref{bases}).

\begin{example}\label{SCPdivision}
(i) The set $\Cc:=\{(1,j),(1,i),(1,ji)\}$ is a SCP for $\HH$ endowed with the standard conjugation.  

(ii) The set $\Cc:=\{(1,\ell),(1,j),(1,i),(1,\ell j),(1,\ell i),(1, ji),(1,\ell (ji))\}$ is a SCP for $\O$ endowed with the standard conjugation. $\sqbullet$
\end{example}

The Clifford algebra $\R_3$ equipped with the Clifford conjugation admits a SCP $\Cc$, but in this case it is not possible to have $|\Cc_1|=1$.

\begin{example}\label{SCPR3}
The set $\Cc:=\{(1,e_1),(1,e_2),(1,e_3),(1,e_{12}),(1,e_{13}),(e_{23},e_{123})\}$ is a SCP for $\R_3$ with $|\Cc_1|=2$. As $e_Ke_{\{1,2,3\}\setminus K}\not\in\sph_{\R_3}$ for each $K\in\mc{P}(3)$, there exist no SCP with $|\Cc_1|=1$.~$\sqbullet$
\end{example}

Also the split octonions $\mathbb{S}\O$ admits no SCP $\Cc$ with $|\Cc_1|=1$.

\begin{example}\label{splitoctonions3}
Let $\mathbb{S}\O$ be the algebra of split octonions equipped with the basis $\{e_K\}_{K\in\mc{P}(3)}$ defined in Example \ref{basis}(iv). Observe that $\{e_K\}_{K\in\mc{P}(3)}\cap\sph_{\mathbb{S}\O}=\{i,j,ji\},
$ both for the $*$-involution $q+\ell p\mapsto q^c-\ell p$ introduced in Example \ref{splitoctonions} and for the $*$-involution $q+\ell p\mapsto q^c+\ell p$ introduced in Example \ref{splitoctonions2}(iii). In particular, 
$$
\Cc:=\{(1,i), (1,j), (1,ji), (\ell, \ell i), (\ell, \ell j), (\ell, \ell(ji))\}
$$
is a SCP (with respect to both these $*$-involutions). Let $H\in\mc{P}(3)$. If $H\in\{\varnothing, \{2\}, \{3\}, \{2,3\}\}$, then $(e_H,\ell)$ is not a Cartan pair, while, if $H\not\in\{\varnothing, \{2\}, \{3\}, \{2,3\}\}$, then $(e_H,1)$ is not a Cartan pair. In particular, if $\Cc'$ is any SCP, then, necessarily, $|\Cc'_1|>1$. We deduce that there exist no SCP $\Cc$ with $|\Cc_1|=1$. $\sqbullet$
\end{example}

We end this section with the following:

\begin{remark}\label{opposto}
By the fact that $e_Le_M=\sigma(L,M)\sigma(M,L)e_Me_L$ for each $L,M\in\mc{P}(n)$, we deduce that $(e_L,e_M)$ is a Cartan pair if and only if $(e_M,e_L)$ is a Cartan pair. Thus, if $\Cc$ is a SCP, then also $\Cc^{\inv}:=\{(e_M,e_L): (e_L,e_M)\in \Cc\}$ is a SCP. $\sqbullet$
\end{remark}

\subsection{Existence of systems of Cartan pairs}

We show in this section that a SCP always exists. Recall our Assumption \ref{assumption}.

\begin{prop}\label{esistenza}
There exists $\Cc\subset\{(e_L,e_M)\}_{L.M\in \mc{P}(n)}$ which is a SCP for $A$.
\end{prop}
\begin{proof}
By Proposition \ref{sqrt-1}, there exists $K\in \mc{P}(n)\setminus\{\varnothing\}$ such that $e_K\in\sph_A$. As $K\neq \varnothing$, we can choose an element $m$ of $K$. Define the sets
$$
\Cc_1:=\{L\in\mc{P}(n) : m\not\in L\} \quad \text{and} \quad \Cc_2:=\{M\in\mc{P}(n) : m\in M\}.
$$
For each $L\in\Cc_1$, as $m\not\in L$ and $m\in K$, we deduce that $m\in K\setminus L$, so 
$$
L\ds K=(L\setminus K)\cup (K\setminus L)\in \Cc_2.
$$
Define the set
$$
\Cc:=\{(e_L,e_{L\ds K}): L\in \Cc_1\}=\{(e_L,e_M): (L,M)\in \Cc_1\times \Cc_2 \text{ and } M=L\ds K\}.
$$
Clearly, $\Cc_1\cap \Cc_2=\varnothing$. Moreover, if $M\in\Cc_2=\mc{P}(n)\setminus\Cc_1$, then $M\ds K\in\Cc_1$ and $(M\ds K)\ds K=M$, so $\Cc$ satisfies property (i) of Definition \ref{SCP}. As $L\ds (L\ds K)=K$ and $\sigma(L,L\ds K)\in\{-1,1\}$ for each $L\in \Cc_1$, we deduce that 
$$
e_{L}e_{L\ds K}=\sigma(L,L\ds K) e_{L\ds (L\ds K)}=\sigma(L,L\ds K) e_K\in\{-e_K,e_K\}\subset \sph_A
$$  
for each $L\in \Cc_1$. Thus, the pair $(e_L,e_{L\ds K})$ is a Cartan pair for each $L\in\Cc_1$, and $\Cc$ satisfies also property (ii) of Definition \ref{SCP}. We conclude that $\Cc$ is a SCP, as required. 
\end{proof}

\begin{remark}\label{rem}
Thanks to the previous proof, if $e_K\in\sph_A$ and $m\in K$, then $\Cc:=\{(L,L\ds K)\in\mc{P}(n)\times\mc{P}(n):m\not\in L\}$ is a SCP for $A$. $\sqbullet$
\end{remark}

One could also introduce the definition of SCP for $\ds$-algebras. For a $\ds$-algebra $A$ which is not a Cartan algebra the existence of SCP is not always guaranteed, as shown in the following example. That is why, we have introduced Cartan algebras, rather than considering $\ds$-algebras.

\begin{example}\label{noC}
Consider the algebra of dual quaternions $\D\HH$ endowed with the basis of Example \ref{basis}(iii). Observe that $\eps i$ is an element of the basis. As $((\eps i)e)^2=(e(\eps i))^2=0$ for each $e\in \D\HH$, for any $*$-involution (for instance, the one defined in Example \ref{dualquaternions}), it holds $(\eps i)e,e(\eps i)\not\in \sph_{\D\HH}$ for each $e\in\D\HH$. Thus, $\D\HH$ does not admit a SCP. $\sqbullet$
\end{example}

\section{Reconstruction of a slice regular function from a system of Cartan pairs}\label{S5}

Recall that, except otherwise stated, \textit{$A$ denotes an alternative Cartan $*$-algebra equipped with the basis $\{e_K\}_{K\in\mc{P}(n)}$ whose product is induced by a function $\sigma$, such that $\sph_A\neq \varnothing$}.

\subsection{Slice regular functions and Cartan pairs}\label{sliceCPS}

Let $D\subset \C$ be a non-empty open subset and $F:=F_1+\iota F_2:D\to A\otimes_\R\C$ a stem function. As $\{e_K\}_{K\in \mc{P}(n)}$ is a basis of $A$ as a $\R$-vector space, there exist unique real valued functions $F_{s,K}:D\to\R$ for $s=1,2$ and $K\in\mc{P}(n)$ such that $F_s=\sum_{K\in\mc{P}(n)} F_{s,K}e_K$ for $s=1,2$. For each $K\in\mc{P}(n)$, define $F_K:=F_{1,K}+\iota F_{2,K}:D\to A\otimes_\R\C$. We have that each $F_K$ has values in $\C=\R\otimes_\R\C$ and
\begin{equation}\label{stemK}\textstyle
F=\sum_{K\in\mc{P}(n)}F_Ke_K
\end{equation}
on $D$. We say that the $F_K$ are the \emph{complex components of $F$}. Evidently, $\overline{F_K(z)}=F_K(\overline{z})$ for each $z=\alpha+i\beta\in D$, so each $F_K$ is a stem function. Observe that $F$ is a holomorphic function, i.e., $\frac{\partial F}{\partial\alpha}+\iota\frac{\partial F}{\partial\beta}=\sum_{K\in\mc{P}(n)}\big(\frac{\partial F_K}{\partial\alpha}+\iota\frac{\partial F_K}{\partial\beta}\big)e_K=0$ if and only if $\frac{\partial F_K}{\partial\alpha}+\iota\frac{\partial F_K}{\partial\beta}=0$ for each $K\in\mc{P}(n)$, i.e., each function $F_K$ is holomorphic (see \cite[Rmk.3(2)]{gp}).

In what follows, if $F:D\to A\otimes_\R\C$ is a stem function and we write $F=\sum_{K\in\mc{P}(n)}F_Ke_K$, we implicitly assume that the $F_K$ are the complex components of $F$. 

\begin{defn}
Let $F=\sum_{K\in\mc{P}(n)}F_Ke_K:D\to A\otimes_\R\C$ be a stem function. For each $L,M\in\mc{P}(n)$ with $L\neq M$, we define the stem function $F_{L,M}:D\to A\otimes_\R\C$ and the slice function $f_{L,M}:\Omega_D\to A$ by
\begin{equation}\label{eqstem}
F_{L,M}:=F_Le_L+F_Me_M
\end{equation}
and
$
f_{L,M}:=\mc{I}(F_{L,M}).  
$ $\sqbullet$
\end{defn} 

Observe that $f_{L,M}=\mc{I}(F_L)e_L+\mc{I}(F_M)e_M$ by \cite[Rmk.7]{gp}. Moreover, if $F$ is holomorphic, then so is each $F_{L,M}$, and every $f_{L,M}$ is slice regular. 

%{\color{green}
%\begin{lem}
%Let $F=\sum_{K\in\mc{P}(n)}F_Ke_K:D\to A\otimes_\R\C$ be a stem function and let $L,M\in\mc{P}(n)$ with $L\neq M$ such that $F_{L,M}$ is holomorphic. Then the stem function
%$$
%F_L+F_Me_{L\ds M}\sigma(L,L)\sigma(M,L):D\to A\otimes_\R\C
%$$
%is also holomorphic. 
%\end{lem}
%
%By Lemma \ref{somm}, we have
%$$
%F_Le_L+F_Me_M=(F_L+\sigma(L,L)\sigma(M,L)F_Me_{L\ds M})e_L
%$$
%for each $L,M\in\mc{P}(n)$. As $F$ is a stem function, {\color{red}we have that} $\overline{F_K(z)}=F_K(\overline{z})$ for each $z\in D$ and each $K\in\mc{P}(n)$ (recall that we are assuming that $D$ is symmetric with respect to the real axes). In particular, the functions 
%\begin{equation}\label{eqstem}
%F_{L,M}:=F_L+\sigma(L,L)\sigma(M,L)F_Me_{L\ds M}
%\end{equation}
%are stem functions for each $L,M\in\mc{P}(n)$, because
%\begin{equation}\label{stem2}
%F_{L,M}=(F_{L,1}+\sigma(L,L)\sigma(M,L)F_{M,1}e_{L\ds M})+\iota(F_{L,2}+\sigma(L,L)\sigma(M,L)F_{M,2}e_{L\ds M}),
%\end{equation}
%so $\overline{F_{L,M}(z)}=F_{L,M}(\overline{z})$ for each $z\in D$ and each $L,M\in\mc{P}(n)$. {\color{red}Observe that
%\begin{equation}\label{stem3}
%F_{M,L}e_L=F_Le_L+F_Me_M.
%\end{equation}
%}}

\begin{lem}\label{stemregular}
The stem function $F$ is holomorphic if and only if the stem function $F_{L,M}$ is holomorphic for each $L,M\in \mc{P}(n)$ such that $(e_L,e_M)$ is a Cartan pair.
\end{lem}
\begin{proof}
%By \cite[Rmk.3(2)]{gp}, $F$ is a holomorphic stem function if and only if the functions $F_K$ are holomorphic stem functions for each $K\in\mc{P}(n)$. In particular, by \eqref{eqstem}, we deduce that if $F$ is holomorphic, then the stem functions $F_{L,M}$ are holomorphic for each $L,M\in\mc{P}(n)$. 
If $F$ is holomorphic, so each $F_K$ is and the same is true for each $F_{L,M}=F_Le_L+F_Me_M$.

Assume now that the stem function $F_{L,M}$ is holomorphic for each $L,M\in \mc{P}(n)$ such that $(e_L,e_M)$ is a Cartan pair. By Proposition \ref{sqrt-1}, there exists $H\in \sph_A$ such that $e_H\in\sph_A$. Choose $K\in\mc{P}(n)$. As $K\ds(K\ds H)=H$, we have that $(K,K\ds H)$ is a Cartan pair. It follows that $F_{K,K\ds H}=F_Ke_K+F_{K\ds H}e_{K\ds H}$ is holomorphic or, equivalently, both $F_K$ and $F_{K\ds H}$ are. In particular, each $F_K$ is holomorphic, so $F$ is also holomorphic, as required.
\end{proof}

Let $F=\sum_{K\in\mc{P}(n)}F_Ke_K$ be a stem function and let $f=\I(F):\Omega\to A$ be the corresponding slice function. As $\{e_K\}_{K\in\mc{P}(n)}$ is a basis of $A$, there exist unique real-valued functions $f_K:\Omega\to \R$ for $K\in\mc{P}(n)$ such that 
\begin{equation}\label{equa}\textstyle
f=\sum_{K\in\mc{P}(n)} f_Ke_K.
\end{equation}
The functions $f_K:\Omega\to \R$ are the \textit{real components of $f$}. Observe that, in general, there is no relation between the functions $f_K:\Omega\to\R$ and $\I(F_K):\Omega\to A$.  

We end this section with the following:

\begin{prop}[Slice functions and Cartan pairs]\label{slicepres}
Let $f=\I(F):\Omega\to A$ be a slice function, let $L,M\in\mc{P}(n)$ be such that $(e_L,e_M)$ is a Cartan pair, and let $\Omega_{e_{L\ds M}}:=\Omega\cap\C_{e_{L\ds M}}$. We have:
\begin{itemize}
\item[$(\rm{i})$] $f_{L,M}=f_Le_L+f_Me_M$ on $\Omega_{e_{L\ds M}}$.
\item[$(\rm{ii})$] $f_{L,M}(x)e_L=f_L(x)\sigma(L,L)+f_M(x)\sigma(M,L)e_{L\ds M}\in\C_{e_{L\ds M}}$ for all $x\in\Omega_{e_{L\ds M}}$.
\item[$(\rm{iii})$] Define the function $\psi:\Omega_{e_{L\ds M}}\to\C_{e_{L\ds M}}$ by
$$
\psi(x):=f_{L,M}(x)e_L=f_L(x)\sigma(L,L)+f_M(x)\sigma(M,L)e_{L\ds M},
$$
and write $x\in\Omega_{e_{L\ds M}}$ as $x=\alpha+\beta e_{L\ds M}$ with $\alpha,\beta\in\R$. If $f$ is slice regular, then $\psi$ is holomorphic on the complex plane $\C_{e_{L\ds M}}$, that is, $\frac{\partial\psi}{\partial\alpha}+e_{L\ds M}\frac{\partial\psi}{\partial\beta}\equiv0$ on $\Omega_{e_{L\ds M}}$.
\end{itemize}
\end{prop}
\begin{proof}
Let $x=\alpha+\beta e_{L\ds M}\in\Omega_{e_{L\ds M}}$ and set $z:=(\phi_{e_{L\ds M}})^{-1}(x)=\alpha+i\beta\in D$. We have:
\begin{align*}
f_{L,M}(x)&=\I(F_Le_L+F_Me_M)(z)=\I(F_Le_L)(z)+\I(F_Me_M)(z)\\
&=F_{1,L}(z)e_L+e_{L\ds M}(F_{2,L}(z)e_L)+F_{1,M}(z)e_M+e_{L\ds M}(F_{2,M}(z)e_M)\\
&=(F_{1,L}(z)+F_{2,M}(z)\sigma(L\ds M,M))e_L+(F_{1,M}(z)+F_{2,L}(z)\sigma(L\ds M,L))e_M%\in\R e_L\oplus\R e_M
\end{align*}
and
\begin{align*}
f(x)-f_{L,M}(x)=&\textstyle\,\I\big(\sum_{K\in\mc{P}(n)\setminus\{L,M\}}F_Ke_K\big)(z)\\
=&\,\textstyle\sum_{K\in\mc{P}(n)\setminus\{L,M\}}\big(F_{1,K}(z)e_K+e_{L\ds M}(F_{2,K}(z)e_K)\big)\\
=&\,\textstyle\sum_{K\in\mc{P}(n)\setminus\{L,M\}}F_{1,K}(z)e_K\\
&\textstyle+\sum_{K\in\mc{P}(n)\setminus\{L,M\}}F_{2,K}(z)\sigma(L\ds M,K)e_{(L\ds M)\ds K}.
\end{align*}
It follows that $f_{L,M}(x)\in\R e_L\oplus\R e_M$. Moreover, as $(L\ds M)\ds K\in\mc{P}(n)\setminus\{L,M\}$ for every $K\in\mc{P}(n)\setminus\{L,M\}$, we deduce that $f(x)-f_{L,M}(x)\in\bigoplus_{K\in\mc{P}(n)\setminus\{L,M\}}\R e_K$, so
\begin{equation}\label{eq:1}
f_{L,M}(x)=f_L(x)e_L+f_M(x)e_M
\end{equation}
and
\begin{equation}\label{eq:2}
f_{L,M}(x)e_L=f_L(x)\sigma(L,L)+f_M(x)\sigma(M,L)e_{L\ds M}\in\C_{e_{L\ds M}}.
\end{equation}

Suppose now that $f$ is slice regular. By Lemma \ref{stemregular}, we have that $f_{L,M}=\I(F_{L,M})=\I(F_Le_L+F_Me_M):\Omega\to A$ is a slice regular function. Bearing in mind \eqref{eq:1} and \eqref{eq:2}, we can define the $\mathscr{C}^1$ functions $\phi_1,\phi_2:\Omega_{e_{L\ds M}}\to\R$ and $\phi:\Omega_{e_{L\ds M}}\to\R e_L\oplus\R e_M$ by
\begin{itemize}
 \item $\phi_1(x):=f_L(x)$, $\phi_2(x):=f_M(x)$,
 \item $\phi(x):=f_{L,M}(x)=\phi_1(x)e_L+\phi_2(x)e_M$,
\end{itemize}
so $\psi(x)=\phi(x)e_L=f_L(x)\sigma(L,L)+f_M(x)\sigma(M,L)e_{L\ds M}$. As $f_{L,M}$ is slice regular, Remark \ref{Cullen} implies that $\frac{\partial\phi}{\partial\alpha}(x)+e_{L\ds M}\frac{\partial\phi}{\partial\beta}(x)=0$. Thus, by Artin's theorem, we deduce:
\begin{align*}
\textstyle\frac{\partial\psi}{\partial\alpha}(x)+e_{L\ds M}\frac{\partial\psi}{\partial\beta}(x)&\textstyle=\frac{\partial\phi}{\partial\alpha}(x)e_L+e_{L\ds M}\big(\frac{\partial\phi}{\partial\beta}(x)e_L\big)\\
&=\textstyle\frac{\partial\phi}{\partial\alpha}(x)e_L+\big(\sigma(L,M)e_Le_M\big)\big(\big(\frac{\partial\phi_1}{\partial\beta}(x)e_L+\frac{\partial\phi_2}{\partial\beta}(x)e_M\big)e_L\big)\\
&=\textstyle\frac{\partial\phi}{\partial\alpha}(x)e_L+\big(\big(\sigma(L,M)e_Le_M\big)\big(\frac{\partial\phi_1}{\partial\beta}(x)e_L+\frac{\partial\phi_2}{\partial\beta}(x)e_M\big)\big)e_L\\
&\textstyle=\frac{\partial\phi}{\partial\alpha}(x)e_L+\big(e_{L\ds M}\frac{\partial\phi}{\partial\beta}(x)\big)e_L=\big(\frac{\partial\phi}{\partial\alpha}(x)+e_{L\ds M}\frac{\partial\phi}{\partial\beta}(x)\big)e_L=0.
\end{align*}
This proves that $\psi$ is holomorphic on the complex plane $\C_{e_{L\ds M}}$ in the usual sense.
\end{proof}

\subsection{Cartan transforms and Cartan reconstruction}

\emph{Let $\Omega=\Omega_D\subset \Qq_A$ be an open circular set, $F=\sum_{K\in\mc{P}(n)}F_Ke_K$ a stem function and $f=\I(F):\Omega\to A$ the corresponding slice function.}

Let us introduce the concepts of Cartan transforms.

\begin{defn}[Cartan transform]\label{Ctransform}
Let $\Cc$ be a SCP of $A$. We define the \textit{Cartan transform $F^{\Cc}:D\to A\otimes_\R\C$ of $F$ (associated to $\Cc$)} as the following stem function
$$\textstyle
F^{\Cc}:=\sum_{(e_L,e_M)\in \Cc}F_{L,M}=\sum_{(e_L,e_M)\in \Cc}(F_Le_L+F_Me_M),
$$
and the \textit{Cartan transform $f^{\Cc}:\Omega\to A$ of $f$ (associated to $\Cc$)} as the following slice function
$$\textstyle
f^{\Cc}:=\I(F^{\Cc})=\sum_{(e_L,e_M)\in \Cc}f_{L,M}. \text{ $\,\sqbullet$}
$$
\end{defn}

The stem function $F^{\Cc}$ has a simple expression in terms of $F$ and $\Cc$.

\begin{prop}\label{KeyCalculation}
The complex components $F^{\Cc}_K$ of $F^{\Cc}$ satisfy
$$
F^{\Cc}_K=\begin{cases}
|\Cc_{2,K}| F_K\quad \text{if } K\in\Cc_1,\\
|\Cc_{1,K}| F_K\quad \text{if } K\in\Cc_2.
\end{cases}
$$
for each $K\in\mc{P}(n)$, where $\Cc_1$ and $\Cc_2$ are defined in \eqref{partition}.
\end{prop}
\begin{proof}
For each $L\in \Cc_1$ and $M\in \Cc_2$, let $\Cc_{1,M}$ and $\Cc_{2,L}$ be the sets introduced in \eqref{partition2}. As
$$
\Cc=\{(e_L,e_M): L\in \Cc_1, M\in\Cc_{2,L}\}=\{(e_L,e_M): M\in \Cc_2, L\in\Cc_{1,M}\},
$$
we have
\begin{align*}
F^\Cc&\textstyle=\sum_{(e_L,e_M)\in\Cc} F_{L,M}=\sum_{(e_L,e_M)\in\Cc} (F_Le_L+F_Me_M)\\
&\textstyle=\sum_{(e_L,e_M)\in\Cc} F_Le_L+\sum_{(e_L,e_M)\in\Cc} F_Me_M\\
&\textstyle=\sum_{L\in\Cc_1}\sum_{M\in\Cc_{2,L}}F_Le_L+ \sum_{M\in\Cc_2}\sum_{L\in\Cc_{1,M}}F_Me_M\\
&\textstyle=\sum_{L\in\Cc_1}|\Cc_{2,L}|F_Le_L+\sum_{M\in\Cc_2}|\Cc_{1,M}|F_Me_M.
\end{align*}
In order to conclude, it is enough to notice that the statement follows by the previous equality because $\{\Cc_1,\Cc_2\}$ is a partition of $\Cc$, so $\Cc_1\cap \Cc_2=\varnothing$. 
\end{proof}

Some consequences are as follows.

\begin{cor}
If $e_K\in\sph_A$ and $m\in K$, then the set $\Cc:=\{(e_L,e_{L\ds K}):m\not\in L\}$ is a SCP for~$A$ and $f^{\Cc}=f$.
\end{cor}
\begin{proof}
By Remark \ref{rem}, the set $\Cc$ is a SCP for~$A$. As $|\Cc_{2,L}|=|\Cc_{1,M}|=1$ for each $L\in\Cc_1$ and $M\in\Cc_2$, Proposition \ref{KeyCalculation} implies the equality $f^{\Cc}=f$.
\end{proof}

Clearly, the Cartan transform $f^{\Cc}$ of $f$ depends on the choice of the SCP $\Cc$ of $A$.

\begin{example}
Let $f:\HH\to \HH$ be the slice polynomial function $f(q):=q(1+i)$. Observe that $f$ is induced by the stem function $F=F_\varnothing+F_1j+F_2i+F_{12}ji:\C\to\HH\otimes_\R\C$, where $F_\varnothing=F_2:=\rm{Id}_\C$ and $F_1=F_{12}:\equiv0$. Here $\rm{Id}_\C:\C\to\HH\otimes_\R\C$ denotes the function ${\rm Id}_\C(\alpha+i\beta):=\alpha+\iota\beta$. Consider the two SCP $\Cc':=\{(1,j),(1,i),(1,ji)\}$ and $\Cc'':=\{(1,i),(j,ji)\}$ of $\HH$. We have:
\begin{align*}
F^{\Cc'}&=F_{\varnothing,\{1\}}+F_{\varnothing,\{2\}}+F_{\varnothing,\{1,2\}}=F_\varnothing+F_1j+F_\varnothing+F_2i+F_\varnothing+F_{12} ji=3\,{\rm Id}_\C+{\rm Id}_\C i,\\
F^{\Cc''}&=F_{\varnothing,\{2\}}+F_{\{1\},\{1,2\}}=F_\varnothing+F_2i+F_1j+F_{12} ji=F
\end{align*}
so
\begin{align*}
f^{\Cc'}(q)&=\I(F^{\Cc'})(q)=3q+qi=q(3+i),\\
f^{\Cc''}(q)&=\I(F)(q)=f(q)=q(1+i)
\end{align*}
for all $q\in\HH$. Observe that these computations are coherent with Proposition \ref{KeyCalculation}. $\sqbullet$
\end{example}

\begin{cor}\label{const1}
The slice function $f$ is slice regular if and only if its Cartan transform $f^{\Cc}$ is slice regular. 
\end{cor}
\begin{proof}
The statement follows straightforwardly by Lemma \ref{stemregular} and Proposition \ref{KeyCalculation}.
\end{proof}

\begin{cor}\label{const2}
The slice function $f$ is constant if and only if its Cartan transform $f^{\Cc}$ is. 
\end{cor}
\begin{proof}
The slice function $f=\I(F)$ is constant if and only if its stem function $F$ is. By Proposition \ref{KeyCalculation}, $F$ is constant if and only if $F^{\Cc}$ is. Finally, $F^{\Cc}$ is constant if and only if $f^{\Cc}=\mathcal{I}(F^{\Cc})$ is. This completes the proof.
\end{proof}

The Cartan transform of a slice regular function allows to prove the following crucial result:

\begin{prop}\label{constslice}
Suppose that $\Omega$ is a circular slice domain and $f:\Omega\to A$ is a slice regular function. Let $\Cc$ be a SCP for $A$. If the real components $f_L$ of $f$ are constant on $\Omega\cap \C_{e_{L\ds M}}$ for each $L,M\in\mc{P}(n)$ such that $(e_L,e_M)\in\Cc$, then $f$ is constant.
\end{prop}
\begin{proof}
Define $\psi:\Omega_{e_{L\ds M}}\to\C_{e_{L\ds M}}$ as in the statement of Proposition \ref{slicepres}$(\rm{iii})$. Observe that the real part of $\psi$ coincides with the restriction $f_L\sigma(L,L)|_{\Omega_{e_{L\ds M}}}$, which is constant by hypothesis. Thus, the holomorphic function $\psi$ is also constant, because $\Omega_{e_{L\ds M}}$ is connected by hypothesis. In addition, as the algebra $A$ is alternative, we have that $\psi(x)\sigma(L,L)e_L=\phi(x)(\sigma(L,L)e_L^2)=\phi(x)=f_{L,M}(x)$, so $f_{L,M}$ is constant on $\Omega_{e_{L\ds M}}$. By the representation formula for slice functions, we deduce that $f_{L,M}$ is constant on the whole $\Omega$. As the latter assertion is true for all $L,M\in\mc{P}(n)$ with $(e_L,e_M)\in\Cc$, it follows that $f^\Cc=\sum_{(e_L,e_M)\in\Cc}f_{L,M}$ is also constant on $\Omega$. By Corollary \ref{const2}, we conclude that $f$ is constant, as required.
\end{proof}

By the previous proposition we deduce straightforwardly the following:

\begin{cor}[Cartan reconstruction]\label{Crecon}
Suppose that $\Omega$ is a circular slice domain. Let $\Cc$ be a SCP of $A$ and let $f,g:\Omega\to A$ be slice regular functions. If $f_L=g_L$ on $\Omega\cap \C_{e_{L\ds M}}$ for each $L,M\in\mc{P}(n)$ such that $(e_L,e_M)\in\Cc$, then $f$ and $g$ coincide on $\Omega$ up to an additive constant. 
\end{cor}
\begin{proof}
Apply Proposition \ref{constslice} to $f-g$.
\end{proof}

\begin{remark}
Thanks to the proof of Proposition \ref{constslice}, we can improve the statement of Corollary \ref{Crecon} asserting that $f-g=\sum_{M\in\Cc_2} c_Me_M$ on $\Omega$
for some constants $c_M\in\R$. $\sqbullet$
\end{remark}

\subsection{Maximum module principle}

In this section we provide a version of the maximum mo\-dule principle for slice regular functions on Cartan $*$-algebras (compare it with \cite[Thm.7.13]{gsst}).

\textit{Let $\Omega=\Omega_D\subset \Qq_A$ be an open circular set, let $f:\Omega\to A$ be a slice function and let $\{f_K\}_{K\in\mc{P}(n)}$ be its real components, see \eqref{equa}.} We start with the following:

\begin{lem}\label{harmonic}
Let $f:\Omega\to A$ be a slice regular function. Then, for each $J\in\sph_A$ and $K\in\mc{P}(n)$, the function $D\to\R$, $(\alpha,\beta)\mapsto f_K(\alpha+\beta J)$ is harmonic.
\end{lem}
\begin{proof}
As $f$ is slice regular, by Remark \ref{Cullen}, we have $(\frac{\partial}{\partial\alpha}+J\frac{\partial}{\partial\beta})f(\alpha+\beta J)\equiv0$ and $\Delta f(\alpha+\beta J)=(\frac{\partial^2}{\partial\alpha^2}+\frac{\partial^2}{\partial\alpha^2})f(\alpha+\beta J)=(\frac{\partial}{\partial\alpha}-J\frac{\partial}{\partial\beta})(\frac{\partial}{\partial\alpha}+J\frac{\partial}{\partial\beta})f(\alpha+\beta J)\equiv0$ on $D$. Thus,
$$\textstyle
\sum_{L\in\mc{P}(n)} \Delta f_L(\alpha+\beta J)e_L=\Delta \big(\sum_{L\in\mc{P}(n)}f_L(\alpha+\beta J)e_L\big)=\Delta f(\alpha+\beta J)\equiv0
$$
on $D$. As $\{e_L\}_{L\in\mc{P}(n)}$ is a basis of $A$ as a real vector space, we conclude that $\Delta f_L(\alpha+\beta J)\equiv0$ on $D$ for each $L\in\mc{P}(n)$, as required.
\end{proof}

We are ready to show the following:

\begin{thm}[Maximum module principle]\label{Cmaxmodule}
Let $\Omega$ be a circular slice domain, let $f:\Omega\to A$ be a slice regular function and let $\Cc$ be a SCP. Suppose that, for each $L,M\in\mc{P}(n)$ with $(e_L,e_M)\in\Cc$, there exists a point $x_{L,M}\in\Omega\cap\C_{e_{L\ds M}}$ such that the module $|f_L|$ of $f_L$ has a local maximum at~$x_{L,M}$. Then, $f$ is constant on $\Omega$.
\end{thm}
\begin{proof}
Let $L,M\in\mc{P}(n)$ be such that $(e_L,e_M)\in\Cc$, and let  $z_{L,M}:=(\phi_{e_{L\ds M}})^{-1}(x_{L,M})\in D$. By Lemma \ref{harmonic}, the function $\xi:D\to\R$, defined by $\xi(\alpha,\beta):=f_L(\alpha+\beta e_{L\ds M})$, is harmonic on~$D$. Observe that $D$ is connected by hypothesis, and either $\xi$ or $-\xi$ has a local maximum at $z_{L,M}\in D$. We deduce that $\xi$ is constant on $D$ or, equivalently, $f_L$ is constant on $\Omega\cap\C_{e_{L\ds M}}$. By Proposition \ref{constslice}, we conclude that $f$ is constant, as required.
\end{proof}

The following example shows that there exist non-constant slice regular functions such that the real part $f_{\varnothing}$ of $f$ has a maximum in $\Omega$. Compare it with \cite[Thm.7.13]{gsst}.

\begin{example}
On the Clifford algebra $\R_3$ consider the slice polynomial function $f(x):=x e_{123}$. By Corollary \ref{Clifford3}, each imaginary unit $J\in\sph_{\R_3}$ is of the form $J=\sum_{e_K^2=-1}x_Ke_K$
for suitable $x_K\in\R$. Thus, we have
$$\textstyle
f(\alpha+\beta J)=\alpha e_{123}+\beta\sum_{e_K^2=-1} \sigma(K,\{1,2,3\}) x_Ke_{K\ds\{1,2,3\}}\in\R e_{123}\oplus\bigoplus_{e_K^2=-1}\R e_{K\ds\{1,2,3\}}
$$
for each $\alpha,\beta\in \R$ and each $J=\sum_{e_K^2=-1}x_Ke_K\in\sph_{\R_3}$. Using again Corollary \ref{Clifford3}, we have that $K\ds\{1,2,3\}\neq \varnothing$ for each $K\in\mc{P}(n)$ such that $e_K^2=-1$. We deduce that the real part $f_{\varnothing}$ of $f$ is constantly equal to zero. In particular, each point of $\Qq_{\R_3}$ is a maximum for $f_{\varnothing}$. $\sqbullet$
\end{example}

\subsection{Schwarz reconstruction of a slice regular function}

\emph{Let $\Omega=\Omega_D\subset \Qq_A$ be an open circular set, let $F=\sum_{K\in\mc{P}(n)}F_Ke_K$ be a stem function and let $f=\I(F):\Omega\to A$ be the corresponding slice function.}

In this section we present a more explicit way to `reconstruct' a slice regular function from some of its real components, without using its Cartan transform. In order to do that, we need to strengthen Definition \ref{SCP}. For each $\Cc\subset\{(e_L,e_M)\}_{L,M\in\mc{P}(n)}$ let $\Cc_1$ and $\Cc_2$ be the sets introduced in \eqref{partition}.

\begin{defn}[Schwarz system of Cartan pairs]\label{SSCP}
We say that $\Cc\subset \{(e_L,e_M)\}_{L,M\in\mc{P}(n)}$ is a \textit{Schwarz system of Cartan pairs for $A$}, for short \textit{SSCP of $A$}, if it satisfies the following properties:
\begin{itemize}
\item[{\rm(i)}] $\Cc$ is a SCP of $A$,
\item[{\rm(ii)}] For each $M\in \Cc_2$, there exists a unique $L\in \Cc_1$ such that $(e_L,e_M)\in \Cc$,
\item[{\rm(iii)}] For each $L\in \Cc_1$, there exist an odd number of $M\in \Cc_2$ such that $(e_L,e_M)\in \Cc$. $\sqbullet$
\end{itemize}
\end{defn}

Clearly, not all the SCP are SSCP. For instance, $\Cc:=\{(1,e_K) : |K|\in\{1,2\}\}\cup\{(e_{123},e_{23})\}$ is a SCP for the Clifford algebra $\R_3$ but not a SSCP. Nevertheless, it is easy to check that any SCP in Remark \ref{rem} is also a SSCP. In particular, we deduce the following:

\begin{prop}\label{esistenza2}
There exists $\Cc\subset\{(e_L,e_M)\}_{L.M\in \mc{P}(n)}$ which is a SSCP.
\end{prop}

Let $\Cc$ be a SSCP. For each $L\in\Cc_1$, let $\Cc_{2,L}$ be the set introduced in \eqref{partition2}. As $\Cc$ is a SSCP, the set $\Cc_{2,L}$ has odd cardinality for each $L\in \Cc_1$. For each $L\in\Cc_1$, let $\Aa_L$ be a subset of $\Cc_{2,L}$ of cardinality $\tfrac{|\Cc_{2,L}|+1}{2}$ (which is a positive integer, because $|\Cc_{2,L}|$ is odd). Then, $\Cc_{L,2}\setminus\Aa_L$ has cardinality $\tfrac{|\Cc_{2,L}|-1}{2}$. As for each $M\in\Cc_2$ there exists a unique $L\in\Cc_1$ such that $(e_L,e_M)\in\Cc$, for each $M\in \Cc_2$ there exists a unique $L\in\Cc_1$ such that $M\in \Cc_{2,L}$. Thus, the family $\{\Cc_{2,L}\}_{L\in \Cc_1}$ is a partition of $\Cc_2$. In particular, as $\{\Cc_1,\Cc_2\}$ is a partition of $\mc{P}(n)$, we deduce that also the family $\Cc_1\cup \{\Cc_{2,L}\}_{L\in \Cc_1}$ is a partition of $\mc{P}(n)$. For each $(e_L,e_M)\in\Cc$, let $F_{L,M}:=F_Le_L+F_Me_M$ be the stem function introduced in \eqref{eqstem} and define 
$$
\widetilde{F}_{L,M}:=F_Le_L-F_Me_M.
$$
We deduce
\begin{align}
\begin{split}\label{sw}
F&\textstyle=\sum_{K\in\mc{P}(n)}F_Ke_K=\sum_{L\in\Cc_{1}}\Big(F_Le_L+\sum_{M\in\Cc_{2,L}}F_Me_M\Big)\\
&\textstyle=\sum_{L\in\Cc_{1}}\Big(\sum_{M\in\Aa_L}\big(F_Le_L+F_Me_M\big)-\sum_{M\in\Cc_{2,L}\setminus\Aa_L}\big(F_Le_L-F_Me_M\big)\Big)\\
&\textstyle=\sum_{L\in\Cc_{1}}\Big(\sum_{M\in\Aa_L}F_{L,M}-\sum_{M\in\Cc_{2,L}\setminus\Aa_L}\widetilde{F}_{L,M}\Big).
\end{split}
\end{align}

For each $(e_L,e_M)\in\Cc$, let  $f_{L,M}=\mc{I}(F_{L,M})$ and define the function $\widetilde{f}_{L,M}:=\mc{I}(\widetilde{F}_{L,M})$. Observe that, if $f$ is slice regular, then both $f_{L,M}$ and $\widetilde{f}_{L,M}$ are slice regular for each $(e_L,e_M)\in\Cc$. By \cite[Rmk.7]{gp}, it holds 
$$
f_{L,M}=\mc{I}(F_L)e_L+\mc{I}(F_M)e_M \quad \text{and} \quad \widetilde{f}_{L,M}=\mc{I}(F_L)e_L-\mc{I}(F_M)e_M
$$
for each $(e_L,e_M)\in\Cc$. By \eqref{sw}, we deduce 
\begin{align}
\begin{split}\label{sw2}
f&=\mc{I}(F)=\textstyle\sum_{L\in\Cc_{1}}\Big(\sum_{M\in\Aa_L}\mc{I}(F_{L,M})-\sum_{M\in\Cc_{2,L}\setminus\Aa_L}\mc{I}(\widetilde{F}_{L,M})\Big)\\
&=\textstyle\sum_{L\in\Cc_{1}}\Big(\sum_{M\in\Aa_L}f_{L,M}-\sum_{M\in\Cc_{2,L}\setminus\Aa_L}\widetilde{f}_{L,M}\Big).
\end{split}
\end{align}

The functions $\widetilde{f}_{L,M}$ satisfy the following, which is the analogous of Proposition \ref{slicepres}. We omit the proof, as it is exactly the same as the one of Proposition \ref{slicepres}.

\begin{prop}\label{slicepres2}
Let $L,M\in\mc{P}(n)$ be such that $(e_L,e_M)\in\Cc$, and let $\Omega_{e_{L\ds M}}:=\Omega\cap\C_{e_{L\ds M}}$. We have:
\begin{itemize}
\item[$(\rm{i})$] $\widetilde{f}_{L,M}=f_Le_L-f_Me_M$ on $\Omega_{e_{L\ds M}}$.
\item[$(\rm{ii})$] $\widetilde{f}_{L,M}(x)e_L=f_L(x)\sigma(L,L)-f_M(x)\sigma(M,L)e_{L\ds M}\in\C_{e_{L\ds M}}$ for all $x\in\Omega_{e_{L\ds M}}$.
\item[$(\rm{iii})$] Define the function $\widetilde{\psi}:\Omega_{e_{L\ds M}}\to\C_{e_{L\ds M}}$ by
$$
\widetilde{\psi}(x):=\widetilde{f}_{L,M}(x)e_L=f_L(x)\sigma(L,L)-f_M(x)\sigma(M,L)e_{L\ds M},
$$
and write $x\in\Omega_{e_{L\ds M}}$ as $x=\alpha+\beta e_{L\ds M}$ with $\alpha,\beta\in\R$. If $f$ is slice regular, then $\widetilde{\psi}$ is holomorphic on the complex plane $\C_{e_{L\ds M}}$, that is, $\frac{\partial\widetilde{\psi}}{\partial\alpha}+e_{L\ds M}\frac{\partial\widetilde{\psi}}{\partial\beta}\equiv0$ on $\Omega_{e_{L\ds M}}$.
\end{itemize}
\end{prop}

For each $L,M\in\mc{P}(n)$ such that $(e_L,e_M)\in\Cc$, define
\begin{equation}\label{gLM}
g_{L,M}:=
\left\{
 \begin{array}{rl}
f_{L,M} &\text{ if } M\in \Aa_{L},\\
-\widetilde{f}_{L,M} &\text{ if } M\not\in \Aa_{L}.
\end{array}
\right.
\end{equation}
Thus, \eqref{sw2} becomes
$$\textstyle
f=\sum_{L\in\Cc_{1}}\Big(\sum_{M\in\Aa_L}g_{L,M}+\sum_{M\in\Cc_{2,L}\setminus\Aa_L}g_{L,M}\Big)=\sum_{L\in\Cc_{1}}\sum_{M\in\Cc_{2,L}}g_{L,M}=\sum_{(e_L,e_M)\in\Cc} g_{L,M}.
$$
Putting all the previous observations together, we derive the following:

\begin{prop}[Schwarz reconstruction]\label{Srecon}
Let $f:\Omega\to A$ be a slice regular function and $\Cc$ a SSCP. For each $(e_L,e_M)\in\Cc$ there exists a slice regular function $g_{L,M}:\Omega\to A$ such that
\begin{itemize}
\item[{\rm(i)}] either $g_{L,M}=f_{L,M}$ or $g_{L,M}=-\widetilde{f}_{L,M}$,
\item[{\rm(ii)}] $f=\sum_{(e_L,e_M)\in\Cc} g_{L,M}$.
\end{itemize}
\end{prop}

For each $(e_L,e_M)\in \Cc$, define the slice function 
\begin{equation}\label{vecchie}
f_{(e_L,e_M)}:=\mc{I}(F_L+F_M\sigma(L,L)\sigma(M,L)e_{L\ds M}).
\end{equation}
By Artin's theorem and \cite[Rmk.7]{gp}, we deduce
\begin{align}
\begin{split}\label{vecchie2}
f_{(e_L,e_M)}e_L&=\mc{I}(F_L+F_M\sigma(L,L)\sigma(M,L)e_{L\ds M})e_L\\
&=\mc{I}(F_L)e_L+\mc{I}(F_M)\sigma(L,L)\sigma(M,L)e_{L\ds M}e_L\\
&=\mc{I}(F_L)e_L+\mc{I}(F_M)e_{M}
\end{split}
\end{align}
on the slice $\Omega_{e_{L\ds M}}$. We have $e_{L\ds M}^c=-e_{L\ds M}$, because $(e_L,e_M)\in\Cc$. Thus,
$$
f_{(e_L,e_M)}^c=\mc{I}(F_L-F_M\sigma(L,L)\sigma(M,L)e_{L\ds M})=f_{(e_L,e_M)}^ce_L=\mc{I}(F_L)e_L-\mc{I}(F_M)e_{M}
$$
on the slice $\Omega_{e_{L\ds M}}$. We deduce
\begin{equation}\label{uguaSw}
f_{(e_L,e_M)}e_L=f_{L,M}\quad \text{and} \quad f_{(e_L,e_M)}^ce_L=\widetilde{f}_{L,M}
\end{equation}
on the slice $\Omega_{e_{L\ds M}}$ for each $(e_L,e_M)\in \Cc$. 

Recall that, as $\Cc$ is a SSCP, for each $M\in \Cc_2$ there exists unique $L:=L_M\in\Cc_1$ such that $(e_L,e_M)\in\Cc$. Define
\begin{equation}\label{gLM2}
h_M:=
\left\{
 \begin{array}{rl}
f_{(e_L,e_M)} &\text{ if } M\in \Aa_{L_M},\\
-f^c_{(e_L,e_M)} &\text{ if } M\not\in \Aa_{L_M}.
\end{array}
\right.
\end{equation}
Clearly, if $f$ is slice regular, then $h_M$ is slice regular for each $(e_L,e_M)\in\Cc$. Let $g_{L,M}$ be the functions introduced in \eqref{gLM2}. By \eqref{uguaSw}, we deduce
\begin{equation}\label{h=g}
g_{L,M}=h_Me_L
\end{equation}
on the slice $\Omega_{e_{L\ds M}}$ for each $(e_L,e_M)\in \Cc$. 

\begin{remark}\label{SWpres}
Let $(e_L,e_M)\in\Cc$. By \eqref{uguaSw}, we have $f_{(e_L,e_M)}\sigma(L,L)=f_{L,M}e_L$ on the slice $\Omega_{e_{L\ds M}}$. Thus, by Proposition \ref{slicepres}(ii), we deduce that the slice function $f_{(e_L,e_M)}$ preserves the slice $\C_{e_{L\ds M}}$. Thus, also the function $f^c_{(e_L,e_M)}$ preserves the slice $\C_{e_{L\ds M}}$. We deduce: \textit{The slice functions $h_M$ preserve the slice $\C_{e_{L\ds M}}$ for each $(e_L,e_M)\in \Cc$.} $\sqbullet$
\end{remark}

By Proposition \ref{Srecon}, we deduce the following version of the splitting lemma for Cartan $*$-algebras.

\begin{lem}[Splitting lemma for Cartan $*$-algebras]\label{SCPlitting}
Let $K\in\mc{P}(n)$ be such that $e_K\in\sph_A$ and let $m\in K$. If $f$ is slice regular, then for each $M\in\mc{P}(n)$ such that $m\not\in M$ there exists a holomorphic function $\varphi_M:\Omega_{e_K}\to \C_{e_K}$ such that 
$$\textstyle
f(x)=\sum_{M\not\owns m} \varphi_M(x)e_{M\ds K}
$$
for each $x\in \Omega_{e_K}$.
\end{lem}
\begin{proof}
As $e_K\in \sph_A$, by Remark \ref{rem}, the set $\Cc:=\{(e_M,e_{M\ds K}) : m\not\in M\}$ is a SCP. By Remark \ref{opposto}, also the set $\Cc^{\inv}=\{(e_{M\ds K},e_M): m\not\in M\}$  is a SCP. Clearly, $\Cc^{\inv}$ is also a SSCP. Let $\Cc^{\inv}_1$ and $\Cc^{\inv}_2$ be the set introduced in \eqref{partition}. It holds $\Cc_2^{\inv}=\{M\in\mc{P}(n) : m\not\in M\}$. For each $M\in\Cc_2^{\inv}$, let $h_M$ be the slice function introduced in \eqref{gLM2}, which is slice regular, because $f$ is slice regular. By Remark \ref{SWpres}, $h_M$ preserves the slice $\C_{e_{K}}$ for each $M\in\Cc_2^{\inv}$, because $(M\ds K)\ds M=K$. For each $M\in\Cc_2^{\inv}$, let $\varphi_M$ be the restriction of $h_M$ to the slice $\Omega_{e_{K}}$. Then, $\varphi_M:\Omega_{e_K}\to \C_{e_K}$ is a holomorphic function. By Proposition \ref{Srecon} and \eqref{h=g}, we deduce 
\begin{multline*}\textstyle
f(x)=\sum_{(e_L,e_M)\in\Cc^{\inv}} g_M(x)=\sum_{(e_L,e_M)\in\Cc^{\inv}} \varphi_M(x)e_L\\
=\textstyle\sum_{M\in\Cc^{\inv}_2} \varphi_M(x)e_{M\ds K}=\sum_{M\not\owns m} \varphi_M(x)e_{M\ds K}.
\end{multline*}
for each $x\in\Omega_{e_K}$, as required. 
\end{proof}

Compare the previous result with \cite[Lem.2.5]{gs}, \cite[Lem.2.7]{css} and \cite[Lem.2.4]{gp3}. 

If the involved Cartan $*$-algebra $A$ is associative, then \eqref{h=g} holds on the whole circular set $\Omega$. Thus, Proposition \ref{Srecon}(ii) becomes 
$$\textstyle
f(x)=\sum_{L\in\Cc_{1}}\Big(\sum_{M\in\Aa_L}f_{(e_L,e_M)}(x)-\sum_{M\in\Cc_{2,L}\setminus\Aa_L}f^c_{(e_L,e_M)}(x)\Big)e_L=\sum_{(e_L,e_M)\in \Cc}h_M(x)e_L
$$
for each $x\in\Omega$ and each function $h_M$ preserves the slice $\C_{e_{L\ds M}}$, where $L$ is the unique element of $\mc{P}(n)$ such that $(e_L,e_M)\in \Cc$. 

\begin{example}
The set $\Cc:=\{(1,e_1),(1,e_2),(1,e_3),(1,e_{12}),(1,e_{13}),(e_{23},e_{123})\}$ is a SSCP for the Clifford algebra $\R_3$. Let $f:\Omega\to \R_3$ be a slice regular function. Then
$$
f=f_{(1,e_1)}-f^c_{(1,e_2)}+f_{(1,e_3)}-f^c_{(1,e_{12})}+f_{(1,e_{13})}+f_{(e_{23},e_{123})}e_{23}
$$
and each of the functions on the right hand side preserves the corresponding slice $\C_{e_{L\Delta M}}$. $\sqbullet$
\end{example}

\section{Cartan number and its relationship with the reconstructing number}\label{S6}

Recall that, except otherwise stated, \textit{$A$ denotes an alternative Cartan $*$-algebra equipped with the basis $\{e_K\}_{K\in\mc{P}(n)}$ whose product is induced by a function $\sigma$, such that $\sph_A\neq \varnothing$}. 

\subsection{Cartan number of Cartan $*$-algebras}

In this section we introduce the Cartan number $\Cc(A)$ of $A$ and explore its relationship with the reconstructing number $\Rr(A)$. This relationship between the Cartan number $\Cc(A)$ and the reconstructing number $\Rr(A)$ is particularly important, because it will allow us to give upper bounds for $\Rr(A)$ and to make several precise calculations in concrete examples. Even if the definition of the reconstructing number $\Rr(A)$ is already contained in Definition \ref{reconstructing}, we recall it here for the reader's convenience. 

\begin{defn}[Reconstructing number]
We define the \textit{reconstructing number $\Rr(A)$ of $A$} as the minimum integer $\ell$ that satisfies the following property: there exists a subset $\Kk$ of $\mc{P}(n)$ of cardinality $\ell$ such that, for each circular slice domain $\Omega\subset \Qq_A$ and each slice regular function $f:\Omega\to A$, if the real components $f_K$ of $f$ are constant for each $K\in \Kk$, then $f$ is constant. $\sqbullet$
\end{defn}

Let $\Cc$ be a SCP and $\Cc_1$ and $\Cc_2$ the sets introduced in \eqref{partition}. We introduce the following:

\begin{defn}[Cartan number]\label{CN}
The \textit{Cartan number of} $A$ is the integer
\begin{equation*}
\Cc(A):=\min\big\{\min\{|\Cc_1|,|\Cc_2|\} : \, \Cc \text{ is a SCP}\big\}. \sqbullet
\end{equation*}
\end{defn}

By Proposition \ref{esistenza}, a SCP always exists for Cartan $*$-algebras, so the Cartan number $\Cc(A)$ is always well-defined. By Remark \ref{opposto}, it follows that 
\begin{equation}\label{C1}
\Cc(A)=\min\{|\Cc_1| : \, \Cc \text{ is a SCP}\}=\min\{|\Cc_2| : \, \Cc \text{ is a SCP}\}.
\end{equation}
In what follows, even if often without explicitly mentioning it, we will always make use of the previous equality instead of the one of Definition \ref{CN}.

\begin{example}\label{Cartannumber1}
(i) By Example \ref{SCPdivision}, it follows that $\Cc(\HH)=\Cc(\O)=1$.

(ii) By Example \ref{SCPR3}, it follows that $\Cc(\R_3)=2$.

(iii) By Example \ref{splitoctonions3}, it follows that the split octonions $\mathbb{S}\O$ (both if endowed with the $*$-involution of Example \ref{splitoctonions} or the one of Example \ref{splitoctonions2}(iii)) satisfies $\Cc(\mathbb{S}\O)=2$. $\sqbullet$
\end{example}

The following important result establishes an upper bound for $\Rr(A)$.

\begin{thm}[Upper bound]\label{upper}
Let $A$ be an alternative Cartan $*$-algebra. Then,
$$
\Rr(A)\leq \Cc(A)\leq \frac{\dim(A)}{2}.
$$
\end{thm}
\begin{proof}
Let $\Cc$ be a SCP. By Corollary \ref{constslice}, if the real components $f_L$ of a slice regular function $f$ (defined on a circular slice domain $\Omega\subset \Qq_A$) are constant for each $L\in \Cc_1$, then $f$ is constant. In particular, 
$
\Rr(A)\leq |\Cc_1|
$
for each SCP $\Cc$. Thus, $\Rr(A)\leq \Cc(A)$. By Remark \ref{rem}, we have $\Cc(A)\leq\tfrac{\dim(A)}{2}$. We conclude that
$$
\Rr(A)\leq \Cc(A)\leq \frac{\dim(A)}{2},
$$
as required.
\end{proof}

Clearly, the division algebra of complex numbers $\C$ satisfies $\Rr(\C)=\Cc(\C)=\tfrac{\dim(\C)}{2}=1$. The following example shows that there exist also examples with a more `hypercomplex flavour' (i.e. with $\dim(A)\geq 4$).

\begin{example}\label{max}
Consider the Clifford algebra $\R_{1,1}$ endowed with the $*$-involution introduced in \eqref{*Cartan}. As $e_{1}^c=e_1$, $e_2^c=-e_2$ and $e_{12}^c=e_{12}$, we have $\sph_{\R_{1,1}}=\{e_2,-e_2\}$. In particular, $\Qq_{\R_{1,1}}=\C_{e_2}$. The non-constant slice polynomial function $f:\C_{e_2}\to \R_{1,1}$ defined as $f(x):=x e_1$ for each $x\in\C_{e_2}$ satisfies $f_{\varnothing}\equiv 0$. Thus, for each $e\in \{e_{\varnothing}=1,e_1,e_2,e_{12}\}$ the non-constant slice polynomial function $g_e:=fe:\C_{e_2}\to \R_{1,1}$ satisfies $f_e\equiv 0$. We deduce that $\Rr(\R_{1,1})\geq 2$, so
$$
2\leq \Rr(\R_{1,1})\leq \Cc(\R_{1,1})\leq \frac{\dim(\R_{1,1})}{2}=2.
$$ 
We conclude that $\Rr(\R_{1,1})=\Cc(\R_{1,1})=\tfrac{\dim(\R_{1,1})}{2}$. $\sqbullet$
\end{example}

\subsection{Lower bound for the Cartan number}

In this section we want to determine a lower bound for $\Cc(A)$. We first need some preparation.

\begin{lem}\label{permutazione}
Let $V\in \mc{P}(n)$. If $(e_L,e_M)$ is a Cartan pair, then $(e_{L\ds V},e_{M\ds V})$ is a Cartan pair.
\end{lem}
\begin{proof}
For each $x\in A$, it holds $x\in\sph_A$ if and only if $-x\in \sph_A$. In particular, we may reduce to show: \textit{If $e_Le_M\in\sph_A$, then $(e_Ve_L)(e_Me_V)\in\sph_A$ for each $V\in\mc{P}(n)$.}

Recall that $\sph_A=\{J\in A : t(J)=0, n(J)=1\}=\{J\in A : J^2=-1, J^c=-J\}$. Let $L,M\in\mc{P}(n)$ be such that $e_Le_M\in\sph_A$. Let $V\in\mc{P}(n)$. As $(e_Le_M)^2=-1$, by the Moufang identity \eqref{Mou} and Artin's theorem, we deduce
\begin{align*}
((e_Ve_L)(e_Me_V))^2&=(e_V(e_Le_M)e_V)^2=(e_V(e_Le_M)e_V)(e_V(e_Le_M)e_V)\\
&=e_V((e_Le_M)e_V)(e_V(e_Le_M))e_V=e^3_V(e_Le_M)^2e_V=-e_V^4=-1.
\end{align*}
Using again the Moufang identity \eqref{Mou}, we deduce also
\begin{align*}
((e_Ve_L)(e_Me_V))^c&=(e_Me_V)^c(e_Ve_M)^c=(e_V^ce_M^c)(e_L^ce_V^c)\\
&=e_V^c(e_M^ce_L^c)e_V^c=e_V^c(e_Le_M)^ce_V^c\stackrel{(*)}{=}-e_V(e_Le_M)e_V=-(e_Ve_L)(e_Me_V),
\end{align*}
where the equality ($*$) follows by Lemma \ref{invCart} and by the fact that $(e_Le_M)^c=-e_Le_M$, because $e_Le_M\in\sph_A$. We conclude that $(e_Ve_L)(e_Me_V)\in \sph_A$, as required.
\end{proof}

In order to provide a lower bound for the Cartan number, we define the integer:
\begin{equation}\label{SA}
s_A:=|\{K\in\mc{P}(n) : e_K\in\sph_A\}|.
\end{equation}
By Proposition \ref{sqrt-1}, it follows that $s_A\geq 1$. By the previous lemma we deduce the following:

\begin{lem}\label{lemSA}
The set $\Cc(L):=\{e_M: (e_L,e_M) \text{ {\rm is a Cartan pair}}\}$ has cardinality $s_A$ for each $L\in\mc{P}(n)$.
\end{lem}
\begin{proof}
As $(1,e_M)$ is a Cartan pair if and only if $e_M\in\sph_A$, we have $s_A=|\Cc(\varnothing)|$. For each $L\in\mc{P}(n)$, define the map
$$
\theta_L:\{e_K\}_{K\in\mc{P}(n)}\to \{e_K\}_{K\in\mc{P}(n)}, \quad e_K\mapsto e_{K\ds L}.
$$
As $(K\ds L)\ds L=K$ for each $K\in \mc{P}(n)$, $\theta_L$ is a bijection, because $\theta_L\circ \theta_L$ is the identity map. By Lemma \ref{permutazione}, we have $\theta_L(\Cc(\varnothing))\subset \Cc(L)$ and $\theta_L(\Cc(L))\subset \Cc(\varnothing)$. Thus, $\theta_L|_{\Cc(\varnothing)}:\Cc(\varnothing)\to \Cc(L)$ is a bijection. We conclude that $|\Cc(L)|=s_A$ for each $L\in\mc{P}(n)$, as required.
\end{proof}

For each $x\in \R$ denote by $\lceil x \rceil$ the smallest integer greater or equal than $x$. By the previous lemma, using the pigeonhole principle, we deduce straightforwardly the following:

\begin{prop}[Lower bound]\label{lowerbound}
It holds: 
$$
\Cc(A)\geq\left\lceil \frac{\dim(A)}{s_A+1} \right\rceil.
$$
\end{prop}

The next example shows that there exist alternative Cartan $*$-algebras $A$ that achieve this minimum.

\begin{example}
By Example \ref{Cartannumber1}(iii), we have that the Clifford algebra $\R_3$ (endowed with the Clifford conjugation) satisfies $\Cc(\R_3)=2$. As $s_{\R_3}=6$, then 
\begin{equation*}
\Cc(\R_3)=2=\left\lceil \frac{8}{7} \right\rceil=\left\lceil \frac{\dim(\R_3)}{s_{\R_3}+1} \right\rceil. \sqbullet
\end{equation*}
\end{example}

\subsection{Cartan $*$-algebras with Cartan number one}

By the inequality $\Rr(A)\leq \Cc(A)$, we have $\Rr(A)=1$ when $\Cc(A)=1$. Thus, by Example \ref{Cartannumber1}(i), we deduce the following (which was already contained in \cite[Thm.7.10]{gsst} in a stronger form):

\begin{prop}\label{division0}
The division algebras $\HH$ and $\O$ satisfy $\Rr(\HH)=\Rr(\O)=1$.
\end{prop}

The purpose of this section is to show that the division algebras $\C$, $\HH$ and $\O$ endowed with the standard conjugations are the only alternative Cartan $*$-algebras that satisfy $\Cc(A)=1$, namely, to show the following:

\begin{thm}\label{division1}
$\Cc(A)=1$ if and only if $A$ is a division algebra endowed with the standard conjugation.
\end{thm}

We start with the following lemma.

\begin{lem}\label{goodSCP}
If $\Cc(A)=1$, then $e_K\in\sph_A$ for each $K\in\mc{P}(n)\setminus\{\varnothing\}$.
\end{lem}
\begin{proof}
Let $\Cc$ be a SCP such that $\Cc_1=\{\varnothing\}$. Then, $(1,e_K)$ is a Cartan pair for each $K\in\mc{P}(n)\setminus\{\varnothing\}$, so $e_K\in\sph_A$ for each $K\in\mc{P}(n)\setminus\{\varnothing\}$. Thus, we may reduce to show: \textit{If $\Cc(A)=1$, then there exists a SCP $\Cc$ such that $\Cc_1=\{\varnothing\}$}.

Let $\Cc'$ be a SCP such that $|\Cc'_1|=1$. Assume that $\Cc'_1=\{V\}$ for some $V\in \mc{P}(n)\setminus\{\varnothing\}$. In particular, as $\Cc'$ is a SCP, then $(e_V,1)\in\Cc'$. Consider the set
$
\Cc:=\{(1,e_{M\ds V}) : (e_V,e_M)\in \Cc'\}.
$
By Lemma \ref{permutazione}, we have that each element of $\Cc$ is a Cartan pair. As $\mc{P}(n)$ endowed with the symmetric difference $\ds$ is a group, the map 
$
\mc{P}(n)\setminus\{V\}\to \mc{P}(n)\setminus\{\varnothing\},\, M\mapsto M\ds V
$
is a bijection. We conclude that $\Cc$ is a SCP, as required.
\end{proof}

We are ready to show Theorem \ref{division1}.

\begin{proof}[Proof of Theorem \ref{division1}]
By Example \ref{Cartannumber1}(i), we only need to show that if $\Cc(A)=1$, then $A$ is a division algebra endowed with the standard conjugation. By Lemma \ref{goodSCP}, as $\Cc(A)=1$, we have $e_K\in\sph_A$ for each $K\in \mc{P}(n)\setminus\{\varnothing\}$. In particular, $e_K^c=-e_K$ for each $K\in \mc{P}(n)\setminus\{\varnothing\}$. Thus, the $*$-involution acts on $A$ as
\begin{equation}\label{stconj}\textstyle
x:=x_{\varnothing}+\sum_{K\in\mc{P}(n)\setminus\{\varnothing\}}x_Ke_K\mapsto x^c=x_{\varnothing}-\sum_{K\in\mc{P}(n)\setminus\{\varnothing\}}x_Ke_K
\end{equation}
for each $x\in A$. As $e_K\in\sph_A$ for each $K\in\mc{P}(n)\setminus\{\varnothing\}$, we also have  $e_K^2=-1$ for each $K\in\mc{P}(n)\setminus\{\varnothing\}$. By the Moufang identity \eqref{Mou}, it follows that
\begin{align*}
-1&\textstyle=e_{K\ds H}^2=(\sigma(K,H)e_Ke_H)^2=(e_Ke_H)(e_Ke_H)=\sigma(K,H)\sigma(H,K)(e_Ke_H)(e_He_K)\\
&\textstyle=\sigma(K,H)\sigma(H,K)e_K(e_He_H)e_K=-\sigma(K,H)\sigma(H,K)e_K^2=\sigma(K,H)\sigma(H,K)
\end{align*}
for each $K,H\in \mc{P}(n)\setminus\{\varnothing\}$ such that $K\neq H$. We deduce that
$$
e_Ke_H=\sigma(K,H)\sigma(H,K)e_He_K=-e_He_K
$$
for each $K,H\in \mc{P}(n)\setminus\{\varnothing\}$ such that $K\neq H$. Thus,
\begin{align*}
xx^c&\textstyle=\Big(x_{\varnothing}+\sum_{K\in\mc{P}(n)\setminus\{\varnothing\}}x_Ke_K\Big)\Big(x_{\varnothing}-\sum_{K\in\mc{P}(n)\setminus\{\varnothing\}}x_Ke_K\Big)=x_{\varnothing}^2-\Big(\sum_{K\in\mc{P}(n)\setminus\{\varnothing\}}x_Ke_K\Big)^2\\
&\textstyle=\sum_{K\in\mc{P}(n)}x_K^2-\sum_{K,H\in \mc{P}(n)\setminus\{\varnothing\}, \, K\neq H} x_Kx_He_Ke_H=\sum_{K\in\mc{P}(n)}x_K^2
\end{align*}
for each $x\in A$. In particular, $xx^c\in \R\setminus\{0\}$ for each $x\in A\setminus\{0\}$. We deduce
$$
x\frac{x^c}{xx^c}=\frac{x^c}{xx^c}x=1
$$
for each $x\in A$ such that $x\neq 0$. We conclude that $A$ is a division algebra endowed with the conjugation \eqref{stconj}, which is the standard conjugation on $A$, as required.
\end{proof}

\subsection{A sufficient condition to have $\Rr(A)\geq 2$}

The following proposition provides a sufficient condition in order to guarantee that $\Rr(A)\geq 2$.

\begin{prop}\label{maggioredi2}
Assume that $\Cc(A)\geq2$. If 
\begin{equation}\label{buonacondition}
\{K\in \mc{P}(n) : e_K^c=-e_K\}\subset \{K\in \mc{P}(n) : e_K^2=-1\},
\end{equation}
then $\Rr(A)\geq 2$.
\end{prop}
\begin{proof}
As $\Cc(A)\geq 2$, there exists $L\in\mc{P}(n)\setminus\{\varnothing\}$ such that $e_L\not\in \sph_A$, otherwise the set $\Cc=\{(1,e_L) : L\in\mc{P}(n)\setminus\{\varnothing\}\}$ would be a SCP such that $|\Cc_1|=1$ and $\Cc(A)=1$. Let $H\in \mc{P}(n)$ be any element. Consider the slice polynomial function $f(x):=xe_{L\ds H}$ defined on $\Qq_A$. As $f$ is not a constant function, in order to conclude, it is enough to show: \textit{The real component $f_H$ of $f$ is constantly equal to zero.} 

Let $J\in\sph_A$. As $\{e_K\}_{K\in \mc{P}(n)}$ is a basis of $A$ as a real vector space, $t(J)=0$ and $t$ is a real linear map, it follows that
$
J=\sum_{t(e_K)=0}x_Ke_K
$
for suitable $x_K\in \R$. By \eqref{buonacondition}, $t(e_L)=0$, so
$$\textstyle
Je_{L\ds H}=\sum_{t(e_K)=0}x_Ke_Ke_{L\ds H}\in \bigoplus_{K\neq H} \R e_K
$$
for each $J\in\sph_A$, because $K\ds(L\ds H)=H$ if and only if $K=L$. As $L\neq\varnothing$, it follows that $L\ds H\neq H$ and
$$\textstyle
f(\alpha+\beta J)=(\alpha+\beta J)e_{L\ds H}=\alpha e_{L\ds H} +\beta Je_{L\ds H}\in \bigoplus_{K\neq H} \R e_K
$$
for each $\alpha,\beta\in \R$ and $J\in\sph_A$ such that $\alpha+\beta J\in\Omega$. We conclude that $f_H$ is constantly equal to zero, as required.
\end{proof}

By the inequality $\Rr(A)\leq \Cc(A)$ and the previous proposition, we deduce straightforwardly the following:

\begin{cor}\label{2}
If $\Cc(A)=2$ and $A$ satisfies condition \eqref{buonacondition}, then $\Rr(A)=2$.
\end{cor}

Observe that a Cartan algebra $A$ endowed with the $*$-involution introduced in \eqref{*Cartan} always satisfies \eqref{buonacondition}. In particular, condition \eqref{buonacondition} is satisfied by the Clifford algebras $\R_n$ endowed with the Clifford conjugation and the Clifford algebras $\R_{n,0}$ endowed with the reversion (see Example \ref{exa:reversion}). By Theorem \ref{division1} and Proposition \ref{maggioredi2}, we deduce straightforwardly the following:

\begin{cor}\label{corCliffRn}
The Clifford algebra $\R_{p,q}$ endowed with the $*$-involution introduced in \eqref{*Cartan} satisfies $\Rr(\R_{p,q})\geq 2$ for each $p,q\in\mathbb{N}$ such that $n:=p+q\geq 2$ and $(p,q)\neq(0,2)$. In particular,
\begin{itemize}
\item[{\rm(i)}] The Clifford algebra $\R_n$ endowed with the Clifford conjugation satisfies $\Rr(\R_n)\geq 2$ for each $n\geq 3$.
\item[{\rm(ii)}] The Clifford algebra $\R_{n,0}$ endowed with the reversion satisfies $\Rr(\R_{n,0})\geq 2$ for each $n\geq 2$.
\end{itemize}
\end{cor}

The next example shows that $\Rr(\R_3)=2$. 

\begin{example}\label{RR3}
\textit{The algebra Clifford algebra $\R_3$ (endowed with the Clifford conjugation) satisfies $\Rr(\R_3)=2$}. In fact, by Example \ref{Cartannumber1}(ii), $\Cc(\R_3)=2$, so by Corollary \ref{2}, $\Rr(\R_3)=2$, because $\R_3$ satisfies condition \eqref{buonacondition}. An explicit example of a non-constant slice regular function with constant real part is given by the slice polynomial function $f(x):=xe_{123}$. $\sqbullet$
\end{example}

The following example shows that Proposition \ref{maggioredi2} is false without assumption \eqref{buonacondition}. Moreover, it provides an example of an associative Cartan $*$-algebras $A$ that satisfies $\Rr(A)<\Cc(A)$.

\begin{example}\label{R11}
Consider the Clifford algebra $\R_{2,0}$ endowed with the Clifford conjugation. Observe that as $e_K^c=-e_K$ for each $K\in\mc{P}(2)\setminus\{\varnothing\}$, condition \eqref{buonacondition} is not satisfied. Moreover, observe that $e_K\in\sph_{\R_{2,0}}$ if and only if $K=\{1,2\}$. We have $\Cc(\R_{2,0})=2$, because $\Cc:=\{(1,e_{12}),(e_1,e_2)\}$ is a SCP, and it is easy to check that $e_Ke_{K\ds \{1\}}\not\in \sph_{\R_{2,0}}$ for each $K\in \mc{P}(2)$. Let us show that $\Rr(\R_{2,0})=1$.

Each $x\in\R_{2,0}$ can be written as $x:=x_0+x_1e_1+x_2e_2+x_{12}e_{12}$ for some $x_0,x_1,x_2,x_{12}\in\R$. Recall that $e_{12}=e_1e_2$ and that $e_1e_2=-e_2e_1$. If $x^2=-1$, then
$$
(x_0^2+x_1^2+x_2^2-x_{12}^2)+2x_0x_1e_1+2x_0x_2e_2+2x_0x_{12}e_{12}=-1.
$$
As $e_K^c=-e_K$ for each $K\in\mc{P}(2)\setminus\{\varnothing\}$, we deduce that $\sph_{\R_{2,0}}$ is the real algebraic set
\begin{equation}\label{sferaconto}
\sph_{\R_{2,0}}=\{x_0=0, \, x_{12}^2-1=x_1^2+x_2^2\}.
\end{equation}
Let $\Omega\subset \Qq_{\R_{2,0}}$ be a circular slice domain and $f:\Omega \to \R_{2,0}$ a slice regular function. Assume that the real component $f_K$ of $f$ is constant for some $K\in\mc{P}(2)$. Up to substitute $f$ with $fe_K$, we may assume that the real part $f_{\varnothing}$ is constant. Moreover, up to substitute $f$ with $f-f_{\varnothing}$ (which is still a slice regular function, because $f_\varnothing$ is constant), we may assume, in addition, that $f_{\varnothing}\equiv0$. Let $D\subset \C$ be the connected open subset invariant under conjugation (see Assumption \ref{assSimm}) such that $\Omega=\Omega_D$ and $F:D\to \R_{2,0}\otimes_\R\C$ the stem function such that $f=\mc{I}(F)$. Let $F_K:=F_{K,1}+\iota F_{K,2}$ for $K\in\mc{P}(2)$ be the complex components of $F$. Then,
\begin{multline*}
f(\alpha+\beta J)=F_{\varnothing,1}(z)+F_{1,1}(z)e_1+F_{2, 1}(z)e_2+F_{12, 1}(z)e_{12}\\
+J(F_{\varnothing,2}(z)+F_{1,2}(z)e_1+F_{2, 2}(z)e_2+F_{12, 2}(z)e_{12})
\end{multline*}
for each $\alpha,\beta\in\R$ and $J\in\sph_{\R_{2,0}}$ such that $\alpha+\beta J\in \Omega$, where $z=\alpha+i \beta$. By \eqref{sferaconto}, we deduce
$$
f_{\varnothing}(\alpha+\beta J)=F_{\varnothing,1}(z)+x_1F_{1,2}(z)+x_2F_{2,2}(z)-x_{12}F_{12,2}(z)
$$
for each $\alpha,\beta\in\R$ and $J=x_1e_1+x_2e_2+x_{12}e_{12}\in\sph_{\R_{2,0}}$ such that $\alpha+\beta J\in \Omega$, where $z=\alpha+i \beta$. In particular, as $f_{\varnothing}\equiv 0$, then
\begin{equation}\label{contoR2}
F_{\varnothing,1}(z)=x_{12}F_{12,2}(z)-x_1F_{1,2}(z)-x_2F_{2,2}(z)
\end{equation}
for each $z\in D$ and $(x_1,x_2,x_{12})\in X:=\{x_{12}^2-1=x_1^2+x_2^2\}$. As $(0,0,1)\in X$, we deduce $F_{\varnothing,1}(z)=F_{12,2}(z)$ for each $z\in D$. By the fact that 
$
(1,0,\sqrt{2}),(\sqrt{2},0,\sqrt{3})\in X,
$
we have  
$$
(\sqrt{2}-1)F_{\varnothing,1}(z)=\tfrac{\sqrt{3}-1}{\sqrt{2}}F_{\varnothing,1}(z)
$$
for each $z\in D$. In particular, $F_{\varnothing,1}\equiv0$. As the vectors $(0,0,1),(1,0,\sqrt{2}),(0,1,\sqrt{2})\in X$ are linearly independent, by \eqref{contoR2}, we have that also $F_{12,2}$, $F_{1,2}$ and $F_{2,2}$ are identically zero. As $f$ is slice regular, the stem function $F$ is holomorphic. We deduce that $F$ is constant, because $D$ is connected. We conclude that $f$ is constant, as required. $\sqbullet$
\end{example}

We end this section by studying the reconstructing number of the algebra $\mathbb{S}\O$ of split octonions. 

\begin{example}
Let $\mathbb{S}\O=\HH+\ell\HH$ be the algebra of split octonions. If $\mathbb{S}\O$ is endowed with the $*$-involution defined as $(q+\ell p)^c=q^c+\ell p$ for each $q,p\in\HH$ (see Example \ref{splitoctonions2}(iii)), then condition \eqref{buonacondition} is satisfied. By Example \ref{Cartannumber1}(iii), we have $\Cc(\mathbb{S}\O)=2$. Thus, by Corollary \ref{2}, we deduce $\Rr(\mathbb{S}\O)=2$.

Consider now $\mathbb{S}\O$ endowed with the $*$-involution defined as  $(q+\ell p)^c=q^c-\ell p$ for each $q,p\in\HH$ (see Example \ref{splitoctonions}). In this case condition \eqref{buonacondition} is no more satisfied and in order to determine $\Rr(\mathbb{S}\O)$ we need some calculation. In order to lighten the notation, we set $e_0:=1$, $e_1:=\ell$, $e_2:=j$, $e_3:=i$, $e_4:=\ell j$, $e_5:=\ell i$, $e_6:=ji$, $e_7:=\ell(ij)$. In particular, $\{e_0,\ldots,e_7\}$ is a real basis of $\mathbb{S}\O$. Let $x:=x_0+x_1e_1+\ldots+x_7e_7\in \mathbb{S}\O$. If $x^2=-1$, we deduce:
$$
(x_0^2+x_1^2-x_2^2-x_3^2+x_4^2+x_5^2-x_6^2+x_7^2)+2x_0\sum_{s=1}^7 x_se_s=-1.
$$
In particular, $\sph_{\mathbb{S}\O}$ is the real algebraic set
$$
\sph_{\mathbb{S}\O}=\{x_0=0,\, x_2^2+x_3^2+x_6^2-1=x_1^2+x_4^2+x_5^2+x_7^2\}.
$$
Arguing as in Example \ref{R11}, with the necessary changes due to the non-associativity of $\mathbb{S}\O$, we conclude that $\Rr(\mathbb{S}\O)=1$. We omit the explicit calculation, because it is simple but long. $\sqbullet$
\end{example}

It is worthwhile to notice that the previous example also shows that the reconstructing number depends, as expected, by the involved $*$-involution.

\subsection{Cartan graph}\label{Sgrafo}

Recall that a \textit{(simple unoriented) graph} is a pair $\Gg=(\Vv,\Ee)$, where $\Vv$ is a set, called \textit{set of vertices}, and $\Ee$ is a subset of $\{\{v,w\} : v,w\in \Vv, v\neq w\}$, called \textit{set of edges}. A graph $\Gg$ is called \textit{regular of degree} $k$ if, for each vertex $v\in \Vv$, there exist exactly $k$ vertices $w\in \Vv$ such that $\{v,w\}\in \Ee$. A \textit{dominating set of} $\Gg$ is a subset $\Ww$ of the set of vertices $\Vv$ such that, for each vertex $v\in\Vv\setminus\Ww$, there exists a vertex $w\in\Ww$ such that $\{v,w\}\in\Ee$. Each graph $\Gg$ has at least one dominating set, in fact, by definition $\Ww:=\Vv$ is always a dominating set for $\Gg$. The \textit{domination number $\gamma(\Gg)$ of} $\Gg$ is the minimum integer $m$ such that there exists a dominating set $\Ww$ of cardinality $m$.

To an alternative Cartan $*$-algebra $A$ we can naturally associate a graph. Define $\Vv_A:=\{e_K\}_{K\in\mc{P}(n)}$ and 
$$
\Ee_A:=\{\{e_K,e_H\} : (e_K,e_H) \text{ is a Cartan pair}\}.
$$ 
As $(e_K,e_K)$ is not a Cartan pair for each $K\in\mc{P}(n)$, we deduce that $\Gg_A:=(\Vv_A,\Ee_A)$ is a (simple) graph (see Figure \ref{FigCartan}). We call the graph $\Gg_A$ the \textit{Cartan graph of} $A$. Let $s_A$ be the integer introduced in \eqref{SA}. By Lemma \ref{lemSA}, for each $K\in\mc{P}(n)$, the sets $\{H\in\mc{P}(n) : (e_K,e_H) \text{ is a Cartan pair}\}$ have the same cardinality $s_A$. Thus, the Cartan graph $\Gg_A$ is a regular graph of degree $s_A$. As $s_A\geq 1$, for each $K\in\mc{P}(n)$ there exists $H\in\mc{P}(n)\setminus\{K\}$ such that $\{e_K,e_H\}\in\Ee_A$. In particular, $\Gg_A$ does not contain isolated vertices. A natural question at this point is to wonder if $\Gg_A$ is always connected. The following example shows that $\Gg_A$ is not always connected. 

\begin{example}
Consider the Clifford algebra $\R_{2,0}$ (endowed with the Clifford conjugation). It is straightforward to see that the Cartan graph $\Gg_{\R_{2,0}}$ of $\R_{2,0}$ is the the one of Figure \ref{Cartan1}, which has 2 connected components. $\sqbullet$
\end{example}
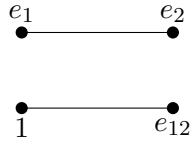
\begin{figure}[htbp]
    \centering
 \begin{tikzpicture}[scale=2]
        \draw (0,0)--(1,0);
        \draw (0,0.5)--(1,0.5);
        \fill (0,0) circle (0.04);
        \fill (0,0.5) circle (0.04);
        \fill (1,0) circle (0.04);
        \fill (1,0.5) circle (0.04);
        
        \node[below] at (0,0) {$1$};
        \node[below] at (1,0) {$e_{12}$};
        \node[above] at (0,0.5) {$e_{1}$};
        \node[above] at (1,0.5) {$e_{2}$};
\end{tikzpicture}
    
    \caption{The Cartan graph of $\R_{2,0}$.}
    \label{Cartan1}
\end{figure}

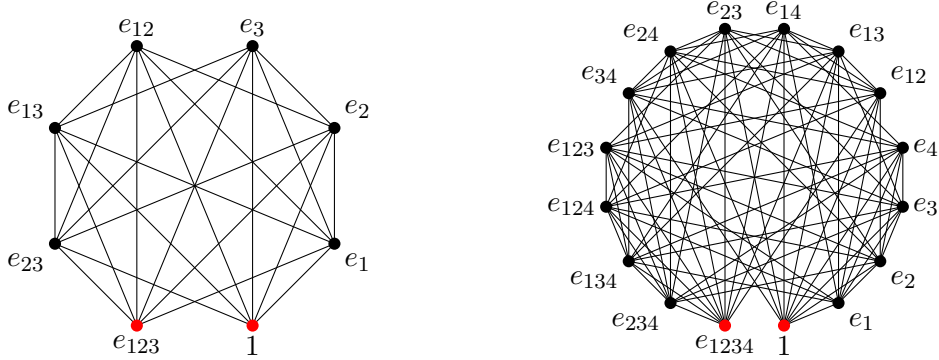
\begin{figure}[htbp]
    \centering
    \begin{subfigure}{0.45\textwidth}
        \centering
        \begin{tikzpicture}[scale=2]
\foreach \i/\ang in {
    8/247.5,
    1/292.5,
    2/337.5,
    3/22.5,
    4/67.5,
    5/112.5,
    6/157.5,
    7/202.5
}{
    \coordinate (n\i) at (\ang:1);
}

    \foreach \j in {2,3,4,5,6,7}{
            \draw (n1) -- (n\j);
}

    \foreach \j in {3,4,5,6,8}{
            \draw (n2) -- (n\j);
}

    \foreach \j in {4,5,7,8}{
            \draw (n3) -- (n\j);
}

   \foreach \j in {6,7,8}{
            \draw (n4) -- (n\j);
}

   \foreach \j in {6,7,8}{
            \draw (n5) -- (n\j);
}

   \foreach \j in {7,8}{
            \draw (n6) -- (n\j);
}

  \foreach \j in {8}{
            \draw (n7) -- (n\j);
}

\foreach \i in {2,...,7}{
\fill (n\i) circle (0.04);
}

\fill[red] (n1) circle (0.04);

\fill[red] (n8) circle (0.04);

\node[below] at (n1) {$1$};
\node[below right] at (n2) {$e_1$};
\node[above right] at (n3) {$e_2$};
\node[above] at (n4) {$e_3$};
\node[above] at (n5) {$e_{12}$};
\node[above left]  at (n6) {$e_{13}$};
\node[below left] at (n7) {$e_{23}$};
\node[below]  at (n8) {$e_{123}$};
\end{tikzpicture}
    \end{subfigure}
    \begin{subfigure}{0.45\textwidth}
        \centering
        \begin{tikzpicture}[scale=2]

\foreach \i/\ang in {
    16/258.75,
    1/281.25,
    2/303.75,
    3/326.25,
    4/348.75,
    5/11.25,
    6/33.75,
    7/56.25,
    8/78.75,
    9/101.25,
    10/123.75,
    11/146.25,
    12/168.75,
    13/191.25,
    14/213.75,
    15/236.25
}{
    \coordinate (n\i) at (\ang:1);
}

\foreach \j in {2,...,10,11}{
            \draw (n1) -- (n\j);
}

\foreach \j in {3,...,8}{
            \draw (n2) -- (n\j);
}

\foreach \j in {12,13,14}{
            \draw (n2) -- (n\j);
}

\foreach \j in {4,5,6,9,10,12,13,15}{
            \draw (n3) -- (n\j);
}

\foreach \j in {5,7,9,11,12,14,15}{
            \draw (n4) -- (n\j);
}

\foreach \j in {8,10,11,13,14,15}{
            \draw (n5) -- (n\j);
}

\foreach \j in {7,8,9,10,12,13,16}{
            \draw (n6) -- (n\j);
}

\foreach \j in {8,9,11,12,14,16}{
            \draw (n7) -- (n\j);
}

\foreach \j in {10,11,13,14,16}{
            \draw (n8) -- (n\j);
}

\foreach \j in {10,11,12,15,16}{
            \draw (n9) -- (n\j);
}

\foreach \j in {11,13,15,16}{
            \draw (n10) -- (n\j);
}

\foreach \j in {14,15,16}{
            \draw (n11) -- (n\j);
}

\foreach \j in {13,14,15,16}{
            \draw (n12) -- (n\j);
}

\foreach \j in {14,15,16}{
            \draw (n13) -- (n\j);
}

\foreach \j in {15,16}{
            \draw (n14) -- (n\j);
}

\foreach \j in {16}{
            \draw (n15) -- (n\j);
}

\foreach \i in {2,...,15}{
\fill (n\i) circle (0.04);
}

\fill[red] (n1) circle (0.04);

\fill[red] (n16) circle (0.04);

\node[below]       at (n1) {$1$};
\node[below right] at (n2) {$e_1$};
\node[below right]       at (n3) {$e_2$};
\node[right]       at (n4) {$e_3$};
\node[right] at (n5) {$e_4$};

\node[above right]       at (n6) {$e_{12}$};
\node[above right]       at (n7) {$e_{13}$};

\node[above]       at (n8) {$e_{14}$};
\node[above]       at (n9) {$e_{23}$};

\node[above left]  at (n10) {$e_{24}$};
\node[above left]        at (n11) {$e_{34}$};

\node[left]        at (n12) {$e_{123}$};
\node[left]        at (n13) {$e_{124}$};
\node[below left]  at (n14) {$e_{134}$};
\node[below left]       at (n15) {$e_{234}$};
\node[below]       at (n16) {$e_{1234}$};
        \end{tikzpicture}
    \end{subfigure}
    
    % Didascalia globale
    \caption{The Cartan graph of $\R_3$ (left) and $\R_4$ (right) with a distinguished dominating set (labelled with red bullets).}
    \label{FigCartan}
\end{figure}

The dominating number $\gamma(\Gg_A)$ of the Cartan graph $\Gg_A$ is strictly related to the Cartan number $\Cc(A)$. In fact, it holds the following:

\begin{prop}\label{grafo}
$\gamma(\Gg_A)=\Cc(A)$.
\end{prop}
\begin{proof}
We show the two inequalities $\gamma(\Gg_A)\leq \Cc(A)$ and $\gamma(\Gg_A)\geq \Cc(A)$ separately.
 
$\bullet\, \, \gamma(\Gg_A)\leq \Cc(A)$. Let $\Cc$ be a SCP and $\Cc_1$ and $\Cc_2$ the sets introduced in \eqref{partition}. Assume that $\Cc(A)=|\Cc_1|$ (see \eqref{C1}). Let $\Ww:=\{e_L : L\in\Cc_1\}$ and $e_M\in\Vv_A\setminus \Ww$. As $\{\Cc_1,\Cc_2\}$ is a partition of $\mc{P}(n)$, we have $M\in\Cc_2$. Thus, there exists $L\in\Cc_1$ such that $(e_L,e_M)\in \Cc$. In particular, $(e_L,e_M)$ is a Cartan pair, so $\{e_L,e_M\}\in \Ee_A$. We deduce that $\Ww$ is a dominating set for the graph $\Gg_A$. Thus, $\gamma(\Gg_A)\leq |\Ww|=|\Cc_1|=\Cc(A)$, as required. 

$\bullet\, \, \gamma(\Gg_A)\geq \Cc(A)$. Let $\Ww$ be a dominating set of $\Gg_A$ such that $\gamma(\Gg_A)=|\Ww|$. We claim: \textit{The set}
$
\Cc:=\{(e_L,e_M): e_L\in \Ww, e_M\in \{e_K\}_{K\in\mc{P}(n)}\setminus\Ww \text{ and } (e_L,e_M) \text{ is a Cartan pair}\}
$
\textit{is a SCP}. 

Clearly, $\Cc_1\cap \Cc_2=\varnothing$ and each element of $\Cc$ is a Cartan pair. Thus, we only need to show that $\Cc_1\cup \Cc_2=\mc{P}(n)$. By the definition of dominating set, for each $e_M\in \{e_K\}_{K\in\mc{P}(n)}\setminus\Ww$, there exists $e_L\in\Ww$ such that $\{e_L,e_M\}\in\Ee_A$. Thus, $(e_L,e_M)$ is a Cartan pair, so $(e_L,e_M)\in \Cc$. In particular, $M\in\Cc_2$. 

 Choose now $e_L\in\Ww$. By Lemma \ref{sqrt-1}, there exists at least one $M'\in\mc{P}(n)$ such that $(e_L,e_M')$ is a Cartan pair. We deduce that there exists $e_M\in \{e_K\}_{K\in\mc{P}(n)}\setminus\Ww$ such that $(e_L,e_M)$ is a Cartan pair, because, otherwise $\Ww\setminus\{e_L\}$ would be a dominating set of $\Gg_A$ and 
$$
\gamma(\Gg_A)\leq|\Ww\setminus\{e_L\}|=|\Ww|-1<|\Ww|=\gamma(\Gg_A),
$$ 
which is a contradiction. Thus, $(e_L,e_M)\in \Cc$. In particular, $L\in\Cc_1$. We deduce that
$$
\mc{P}(n)=\{K : e_K\in \Ww\}\cup\{K : e_K\not\in\Ww\}\subset \Cc_1\cup\Cc_2\subset \mc{P}(n).
$$
We conclude that, $\Cc_1\cup \Cc_2=\mc{P}(n)$, so $\Cc$ is a SCP, as claimed.

As $\Cc$ is a SCP, we have $|\Cc_1|\geq \Cc(A)$ (see \eqref{C1}). We conclude that $\gamma(\Gg_A)=|\Ww|\geq|\Cc_1|\geq \Cc(A)$, as required.
\end{proof}

We deduce straightforwardly the following:

\begin{cor}\label{ricoprire}
Let $\Ww$ be a subset of $\{e_K\}_{K\in\mc{P}(n)}$. If for each $e_M\in\{e_K\}_{K\in\mc{P}(n)}\setminus\Ww$ there exists $e_L\in\Ww$ such that $(e_L,e_M)$ is a Cartan pair, then $\Cc(A)\leq |\Ww|$.
\end{cor}
\begin{proof}
The set $\Ww$ is a dominating set for the Cartan graph $\Gg_A$. Thus, by Proposition \ref{grafo}, we deduce $\Cc(A)=\gamma(\Gg_A)\leq |\Ww|$, as required.
\end{proof}

The dominating number of a graph is an important invariant studied by many authors. For further information and details, we refer the reader to the monograph \cite{H} and the references contained therein. In the literature, there exist several upper bounds for the dominating number of regular graphs depending on their degree \cite{B, B2, R, S}. Thanks to Proposition \ref{grafo}, we can use these upper bounds to provide upper bounds for the reconstructing number $\Rr(A)$ of $A$. For instance, in \cite[Thm.7]{B}, it is shown that for a regular graph $\Gg=(\Vv,\Ee)$ of degree $\geq 6$ it holds
$$
\gamma(\Gg)\leq \frac{127}{418}|\Vv|\simeq 0.30382775 |\Vv|.
$$
In particular, we deduce the following (compare it with Theorem \ref{upper}):

\begin{prop}\label{domin}
Assume that $s_A\geq 6$. Then,
$$
\Rr(A)\leq \Cc(A)\leq \frac{127}{418}\dim(A).
$$
\end{prop}

Consider  the Clifford algebra $\R_{p,q}$ endowed with the Clifford conjugation. We end this section with the following:

\begin{cor}\label{stimaRpq}
For each $p,q\in\mathbb{N}$ such that $n:=p+q\geq 5$, it holds
$$
\Rr(\R_{p,q})\leq \Cc(\R_{p,q})\leq  \frac{127}{418}\dim(\R_{p,q}).
$$
\end{cor}
\begin{proof}
In view of Proposition \ref{domin}, we only need to show that, if $n\geq 5$, then $s_{\R_{p,q}}\geq 6$. By Corollary \ref{Clifford2}, we have that $e_k\in\sph_{\R_{p,q}}$ for each $k\in \{p+1,\ldots,n\}$ (and there are exactly $n-p=q$ of them), and that $e_K\in\sph_{\R_{p,q}}$ for each $K\in\mc{P}(n)$ such that $|K|=2$ and either $K\subset\{1,\ldots,p\}$ or $K\subset \{p+1,\ldots,n\}$ (and there are exactly $\tbinom{p}{2}+\tbinom{q}{2}$ of them). As $s_{\R_{p,q}}=|\{K\in\mc{P}(n) : e_K\in\sph_{\R_{p,q}}\}|$, we deduce that
\begin{align*}
s_{\R_{p,q}}&\geq q+\binom{p}{2}+\binom{q}{2}=\frac{p^2+(q+1)^2-(p+q+1)}{2}\\
&=\frac{(p+q+1)^2+(p-q-1)^2-2(p+q+1)}{4}=\frac{(n+1)^2+(p-q-1)^2-2(n+1)}{4}\\
&\geq \frac{(n+1)^2-2(n+1)}{4}=\frac{n^2-1}{4}.
\end{align*}
We deduce that if $n\geq 5$, then $s_{\R_{p,q}}\geq 6$, as required. 
\end{proof}

\section{On the reconstructing number of Clifford algebras}\label{Scliff}

\textit{Let $p,q\in\mathbb{N}$ with $n:=p+q\geq 2$. We consider the Clifford algebras $\R_{p,q}$ endowed with the Clifford conjugation. Recall that $\R_n:=\R_{(0,n)}$.}

In this section we improve the estimations of Theorem \ref{upper} and Proposition \ref{stimaRpq} for the reconstructing number $\Rr(\R_{p,q})$. We start with the study of $\Rr(\R_{p,q})$ in full generality, that is, for each $n:=p+q\geq 2$. Later, we focus on the algebras $\R_n$ improving significantly our initial results. 

\subsection{On the reconstructing number of the Clifford algebras $\R_{p,q}$}

In view of the inequa\-lity $\Rr(\R_{p,q})\leq \Cc(\R_{p,q})$, we start by studying the Cartan numbers $\Cc(\R_{p,q})$. Recall that, given an integer $m\geq 0$, we denote by $\m{m}$ the remainder of $m$ divided by 4. Moreover, given $K\in\mc{P}(n)$, we denote by $\m{K}$ the remainder of $|K|$ divided by 4. 

\begin{prop}[Cartan number of the Clifford algebras $\R_{p,q}$]\label{Rpq}
It holds
\begin{equation}\label{CRpq}
\Cc(\R_{p,q})\leq\frac{n^3+5n+6}{6}.
\end{equation}
\end{prop}
\begin{proof}
If $n=p+q\leq 3$, the statement is trivially true, because the right hand side of \eqref{CRpq} is equal to the dimension of $\R_{p,q}$. Thus we may assume in what follows that $n\geq 4$. Observe that, as 
$$
\frac{n^3+5n+6}{6}= 1+n+\binom{n}{2}+\binom{n}{3},
$$
the right hand side of the previous inequality is exactly the number of subset $H$ of $\{1,\ldots,n\}$ such that $|H|\leq 3$. Thus, by Corollary \ref{ricoprire}, we may reduce to show: \textit{For each $K\in\mc{P}(n)$ such that $|K|\geq 4$, there exists $H\in \mc{P}(n)$ such that $|H|\leq 3$ and $(e_H,e_K)$ is a Cartan pair.}

For each $K\in\mc{P}(n)$, let $K_+$ and $K_-$ be the sets introduced in \eqref{+-}. We consider all the possible cases $(\m{K_+},\m{K_-})\in \{(k,h) : k,h=0,1,2,3\}$ with $|K|\geq 4$ and for each of them we have to find $H\in\mc{P}(n)$ with $|H|\leq 3$ such that $e_He_K\in\sph_{\R_{p,q}}$. By Corollary \ref{Clifford2}, $e_He_K\in \sph_{\R_{p,q}}$ if and only if 
$$
(\m{(H\ds K)_+},\m{(H\ds K)_-})\in\{(0,1), (0,2), (2,0), (2,3)\}.
$$
Let $K\in\mc{P}(n)$ be such that $|K|\geq 3$. 
 
$\bullet$ $(\m{K_+},\m{K_-})\in\{(0,1), (0,2), (2,0), (2,3)\}$. Then, the pair $(1,e_K)$ is a Cartan pair.

$\bullet$ $(\m{K_+},\m{K_-})\in\{(1,1), (1,2), (3,0), (3,3)\}$. There exists $k\in\{1,\ldots,p\}\cap K$. The pair $(e_k,e_K)$ is a Cartan pair. 

$\bullet$ $(\m{K_+},\m{K_-})\in\{(2,1), (0,3)\}$. There exists $k\in\{p+1,\ldots,n\}\cap K$. The pair $(e_k,e_K)$ is a Cartan pair. 

$\bullet$ $(\m{K_+},\m{K_-})\in\{(1,3),(3,1)\}$. There exist $k\in\{1,\ldots,p\}\cap K$ and $h\in\{p+1,\ldots,n\}\cap K$. The pair $(e_{kh},e_K)$ is a Cartan pair. 

$\bullet$ $(\m{K_+},\m{K_-})=(2,2)$. There exists $H\subset\{1,\ldots,p\}\cap K$ with $|H|=2$. The pair $(e_H,e_K)$ is a Cartan pair. 

$\bullet$ $(\m{K_+},\m{K_-})=(0,0)$. Observe that $K\neq \varnothing$, because $|K|\geq 3$. Then, either there exists $H\subset\{1,\ldots,p\}\cap K$ with $|H|=2$ or there exists $H\subset \{p+1,\ldots,n\}\cap K$ with $|H|=2$. In both cases $(e_H,e_K)$ is a Cartan pair.

$\bullet$ $(\m{K_+},\m{K_-})=(3,2)$. There exist $H\subset\{1,\ldots,p\}\cap K$ such that $|H|=3$. The pair $(e_H, e_K)$ is a Cartan pair.

$\bullet$ $(\m{K_+},\m{K_-})=(1,0)$. As $\m{K_+}=1$, then either $|K_+|=1$ or $|K_+|\geq 5$. If $|K_+|\geq 5$, then there exists $H\subset \{1,\ldots,p\}\cap K$ such that $|H|=3$ and $(e_H,e_K)$ is a Cartan pair. Assume now that $|K_+|=1$. If $(p,q)\neq (1,n-1)$, then $p\geq 2$, because $n=p+q\geq 4$. In particular, there exists $k\in \{1,\ldots,p\}\setminus K$ and the pair $(e_k,e_K)$ is a Cartan pair. If $(p,q)=(1,n-1)$, as $n=p+q\geq 4$, then there exists $H\subset \{2,\ldots,n-1\}$ such that $|H|=2$ and $(e_{H\cup K},e_H)$ is a Cartan pair. 
\end{proof}

By the previous result and the inequality $\Rr(\R_{p,q})\leq \Cc(\R_{p,q})$, we deduce straightforwardly the following:

\begin{cor}[Reconstructing number of the Clifford algebras $\R_{p,q}$]\label{reconRpq}
It holds
\begin{equation}\label{RRpq}
\Rr(\R_{p,q})\leq\frac{n^3+5n+6}{6}.
\end{equation}
\end{cor}

Considering all the algebras $\R_{p,q}$ is a very general situation and our previous estimation is quite `rough'. Nevertheless, it is worthwhile to notice that \eqref{RRpq} is already much sharper, when $n=p+q$ is big enough, with respect to both the inequalities of Theorem \ref{upper} and Proposition \ref{stimaRpq}.

\subsection{On the Cartan number of the Clifford algebras $\R_n$}

In this section, we study the Cartan number of the Clifford algebras $\R_n$. We start with the following:

\begin{example}[Cartan number of $\R_4$]\label{R4}
The set 
$
\Cc:=\{(1,e_K) : |K| = 1,2\}\cup \{(e_{1234},e_K) : |K|=3\}
$
is a SCP for $\R_4$, so $\Cc(\R_4)\leq 2$. Moreover, by Theorem \ref{division1}, we have $\Cc(\R_4)\geq 2$. We deduce that $\Cc(\R_4)=2$. $\sqbullet$
\end{example}

As we already know that $\Cc(\R_2)=\Cc(\HH)=1$ and $\Cc(\R_3)=2$ (see Example \ref{Cartannumber1}(i) and (ii)), the previous example allow us to assume, in what follows, that $n\geq 5$. 

For each subset $\Cc$ of $\{(e_L,e_M)\}_{L,M\in\mc{P}(n)}$, let $\Cc_1$ and $\Cc_2$ be the sets introduced in \eqref{partition}. We start by determining the Cartan number of $\R_n$ when  $[n]_4=1$, which is the easiest case to approach. 

\begin{prop}\label{lem41}
Let $n\geq 5$ be such that $[n]_4=1$. Then, $\Cc(\R_n)=2$.
\end{prop}
\begin{proof}
Let $n\geq 5$ be an integer such that $\m{n}=1$. By Corollary \ref{Clifford3}, we have $e_K\in\sph_{\R_n}$ if and only if $[K]_4\in\{1,2\}$. Let $K\in\mc{P}(n)\setminus\{1, e_{1\ldots n}\}$. If $[K]_4\in\{1,2\}$, then $(1,e_K)$ is a Cartan pair. While, if $[K]_4\in\{0,3\}$, then $(e_{1\ldots n},e_K)$ is a Cartan pair. It follows that there exists a SCP such that $|\Cc_1|=2$, so $\Cc(\R_n)\leq 2$. By Theorem \ref{division1}, we have $\Cc(\R_n)\geq2$. We conclude that $\Cc(\R_n)=2$, as required.
\end{proof}

Next, we study the cases $[n]_4\in\{0,2,3\}$.

\begin{prop}\label{lem4}
Let $n\geq 6$ be such that $[n]_4\in\{0,2,3\}$. Then, $\Cc(\R_n)\leq 4$.
\end{prop}
\begin{proof}
In order to conclude, it is enough to show that there exist $H_1,H_2,H_3\in \mc{P}(n)$ such that for each $K\in\mc{P}(n)\setminus\{\varnothing,H_1,H_2,H_3\}$ at least one among $(1,e_K), (e_{H_1},e_K)$, $(e_{H_2},e_K)$ or $(e_{H_3},e_K)$ is a Cartan pair. Recall that, by Corollary \ref{Clifford3}, we have $e_K\in\sph_{\R_n}$ if and only if $[K]_4\in\{1,2\}$. In particular, $(1,e_K)$ is a Cartan pair for each $K\in\mc{P}(n)$ such that $[K]_4\in\{1,2\}$. Thus, we are reduced to show: \textit{There exist $H_1,H_2,H_3\in \mc{P}(n)$ with the following property: For each $K\in\mc{P}(n)\setminus\{\varnothing, H_1,H_2,H_3\}$ such that $[K]_4\in\{0,3\}$, at least one among $(e_{H_1},e_K)$, $(e_{H_2},e_K)$ or $(e_{H_3},e_K)$ is a Cartan pair.} We consider the cases $[n]_4\in\{0,2,3\}$, separately.

\noindent{\sc Case $[n]_4=0$.} We take $H_1=\{1\}$, $H_2=\{2,\ldots,n\}$ and $H_3=\{1,\ldots,n\}$. Observe that $\m{H_2}=3$. Let $K\in \mc{P}(n)$. If $\m{K}=3$, then $(e_{H_3}, e_K)$ is a Cartan pair. Assume that $\m{K}=0$. We only have the following two possibilities:
\begin{itemize}
\item $1\in K$, then $\m{H_2\ds K}=1$, so $(e_{H_2}, e_K)$ is a Cartan pair.
\item $1\not\in K$, then $\m{H_1\ds K}=1$, so $(e_{H_1}, e_K)$ is a Cartan pair.
\end{itemize}
We conclude that, if $\m{n}=0$, then $\Cc(\R_n)\leq 4$, as required.

\noindent {\sc Case $\m{n}=2$.} We take $H_1=\{1\}$, $H_2=\{2,\ldots,n\}$ and $H_3=\{1,\ldots,n\}$. Observe that $\m{H_2}=1$. Let $K\in \mc{P}(n)$. If $\m{K}=0$, then $(e_{H_3}, e_K)$ is a Cartan pair. Assume that $\m{K}=3$. We only have the following two possibilities:
\begin{itemize}
\item $1\in K$, then $\m{H_1\ds K}=1$, so $(e_{H_1}, e_K)$ is a Cartan pair.
\item $1\not\in K$, then $\m{H_2\ds K}=2$, so $(e_{H_2}, e_K)$ is a Cartan pair.
\end{itemize}
We conclude that, if $\m{n}=2$, then $\Cc(\R_n)\leq 4$, as required.

\noindent {\sc Case $\m{n}=3$.} We take $H_1=\{1\}$, $H_2=\{2\}$ and $H_3=\{1,2\}$. Let $K\in \mc{P}(n)$ be such that $\m{K}\in\{0,3\}$. We only have the following possibilities:
\begin{itemize}
\item If $\{1,2\}\cap K=\varnothing$, then $(e_{H_3}, e_K)$ is a Cartan pair. 
\item If $\{1,2\}\subset K$, then $(e_{H_3}, e_K)$ is a Cartan pair. 
\item If $|\{1,2\}\cap K|=1$ and $\m{K}=0$, then either $(e_{H_1},e_K)$ is a Cartan pair if $1\not\in K$, or $(e_{H_2},e_K)$ is a Cartan pair if $2\not\in K$.
\item If $|\{1,2\}\cap K|=1$ and $\m{K}=3$, then either $(e_{H_1},e_K)$ is a Cartan pair if $1\in K$, or $(e_{H_2},e_K)$ is a Cartan pair if $2\in K$.
\end{itemize}
We conclude that, if $\m{n}=3$, then $\Cc(\R_n)\leq 4$, as required.
\end{proof}

Putting all the previous results together (see Example \ref{Cartannumber1}(i) and (ii), Theorem \ref{division1}, Example \ref{R4}, Proposition \ref{lem41} and Proposition \ref{lem4}), we deduce the following:

\begin{thm}[Cartan number of the Clifford algebras $\R_n$]\label{CartanRn}
For each $n\geq 2$, it holds:
\begin{itemize}
\item[{\rm(i)}] $\Cc(\R_2)=\Cc(\HH)=1$.
\item[{\rm(ii)}] $\Cc(\R_3)=\Cc(\R_4)=2$.
\item[{\rm(iii)}] If $\m{n}=1$, then $\Cc(\R_n)=2$.
\item[{\rm(iv)}] If $n\geq 6$ and $\m{n}\in\{0,2,3\}$, then $2\leq \Cc(\R_n)\leq 4$.
\end{itemize}
\end{thm}

We end this section by showing that $\Cc(\R_6)=4$.

\begin{example}[Cartan number of $\R_6$]
\textit{It holds $\Cc(\R_6)=4$.} 

By Lemma \ref{lem4}, $\Cc(\R_6)\leq 4$. Thus, we only need to show that $\Cc(\R_6)>3$. Let $\Cc$ be a SCP for $\R_6$ such that $|\Cc_1|=\Cc(\R_6)$. Let $V\in \Cc_1$. As $\mc{P}(n)$ endowed with the symmetric difference is a group, the map 
$$
\{e_K\}_{K\in\mc{P}(n)}\to \{e_K\}_{K\in\mc{P}(n)},\quad e_K\mapsto e_{K\ds V}
$$
is a bijection. Thus, by Proposition \ref{permutazione}, the set 
$
\Cc':=\{(e_{L\ds V}, e_{M\ds V}): (L,M)\in \Cc\}
$
is a SCP such that $\varnothing\in \Cc'_1$ and $|\Cc_1'|=|\Cc_1|$. We deduce that there exists a SCP $\Cc$ such that $\varnothing\in \Cc_1$ and $|\Cc_1|=\Cc(\R_6)$. By Corollary \ref{Clifford3}, $(1,e_K)$ is a Cartan pair for each $K\in \mc{P}(n)$ such that $\m{K}\in\{1,2\}$. In particular, in order to conclude, it is enough to show: \textit{For each $H,W\in \mc{P}(n)\setminus\{\varnothing\}$ such that $H\neq W$, there exists $K\in\mc{P}(n)\setminus \{\varnothing, H, W\}$ such that $\m{K}\in\{0,3\}$, and both the pairs $(e_H,e_K)$ and $(e_W,e_K)$ are not Cartan pairs.}

Up to a permutation of the indices, we may assume that 
$$
e_H\in \{e_1,e_{12},e_{123},e_{1234},e_{12345},e_{123456}\}.
$$
We consider all these cases separately. Recall that, by Corollary \ref{Clifford3}, $(e_V,e_K)$ is a Cartan pair if and only if $\m{V\ds K}\in\{1,2\}$. In particular, we have to show that for each $H\in \{\{1\},\{12\},\{123\},\{1234\},\{12345\},\{123456\}\}.$ and each $W\in \mc{P}(n)\setminus\{\varnothing, H\}$, there exists $K\in \mc{P}(n)\setminus\{\varnothing. H,W\}$ with $\m{K}\in\{0,3\}$ such that $\m{H\ds K},\m{H\ds K}\in\{0,3\}$. 

\noindent(I) $e_H=e_1$. Up to a permutation of the indices, we only have to consider the following cases:
$$
e_W\in \{e_2, e_{12}, e_{23}, e_{123}, e_{234}, e_{1234}, e_{2345}, e_{12345}, e_{23456}, e_{123456}\}.
$$
\begin{itemize}
\item If $e_W\in \{e_{12}, e_{23}, e_{123}, e_{234}, e_{2345}, e_{12345}, e_{123456}\}$, then the set $K:=\{2,5,6\}$ satisfies $K\not\in\{\varnothing, H,W\}$, $\m{K}=3$ and both the pairs $(e_H,e_K)$ and $(e_W,e_K)$ are not Cartan pairs.
\item If $e_W\in  \{e_2,  e_{1234}, e_{23456}\}$, then the set $K:=\{1,2,5,6\}$ satisfies $K\not\in\{\varnothing, H,W\}$, $\m{K}=0$ and both the pairs $(e_H,e_K)$ and $(e_W,e_K)$ are not Cartan pairs.
\end{itemize}

\noindent(II) $e_H=e_{12}$. By the previous case and up to a permutation of the indices, we only have to consider the following cases:
$$
e_W\in \{e_{13}, e_{34}, e_{123}, e_{134}, e_{345}, e_{1234}, e_{1345}, e_{3456}, e_{12345}, e_{13456}, e_{123456}\}.
$$
\begin{itemize}
\item If $e_W\in \{e_{13}, e_{123}, e_{134}, e_{345}, e_{1345}, e_{3456},e_{12345}e_{123456}\}$, then the set $K:=\{1,5,6\}$ satisfies $K\not\in\{\varnothing, H,W\}$, $\m{K}=3$ and both the pairs $(e_H,e_K)$ and $(e_W,e_K)$ are not Cartan pairs.
\item If $e_W\in\{e_{34},e_{1234}\}$, then the set $K:=\{1,4,6\}$ satisfies $K\not\in\{\varnothing, H,W\}$, $\m{K}=3$ and both the pairs $(e_H,e_K)$ and $(e_W,e_K)$ are not Cartan pairs.
\item If $e_W=e_{13456}$, then the set $K:=\{2,5,6\}$ satisfies $K\not\in\{\varnothing, H,W\}$, $\m{K}=3$ and both the pairs $(e_H,e_K)$ and $(e_W,e_K)$ are not Cartan pairs.
\end{itemize}

\noindent(III) $e_H=e_{123}$. By the previous cases and up to a permutation of the indices, we only have to consider the following cases:
$$
e_W\in \{e_{124}, e_{145}, e_{456}, e_{1234}, e_{1245}, e_{1456}, e_{12345}, e_{12456}, e_{123456}\}.
$$
\begin{itemize}
\item If $e_W\in \{e_{124}, e_{145}, e_{456}, e_{1234}, e_{12345}\}$, then the set $K:=\{1,2,5,6\}$ satisfies $K\not\in\{\varnothing, H,W\}$, $\m{K}=0$ and both the pairs $(e_H,e_K)$ and $(e_W,e_K)$ are not Cartan pairs.
\item If $e_W\in\{e_{1245},e_{123456}\}$, then the set $K:=\{1,5,6\}$ satisfies $K\not\in\{\varnothing, H,W\}$, $\m{K}=3$ and both the pairs $(e_H,e_K)$ and $(e_W,e_K)$ are not Cartan pairs.
\item If $e_W\in\{e_{1456},e_{12456}\}$, then the set $K:=\{3,5,6\}$ satisfies $K\not\in\{\varnothing, H,W\}$, $\m{K}=3$ and both the pairs $(e_H,e_K)$ and $(e_W,e_K)$ are not Cartan pairs.
\end{itemize}

\noindent(IV) $e_H=e_{1234}$. By the previous cases and up to a permutation of the indices, we only have to consider the following cases:
$$
e_W\in \{e_{1235}, e_{1256}, e_{12345}, e_{12356}, e_{123456}\}.
$$
\begin{itemize}
\item If $e_W\in \{e_{1235}, e_{12345}, e_{123456}\}$, then the set $K:=\{1,2,6\}$ satisfies $K\not\in\{\varnothing, H,W\}$, $\m{K}=0$ and both the pairs $(e_H,e_K)$ and $(e_W,e_K)$ are not Cartan pairs.
\item If $e_W=e_{1256}$, then the set $K:=\{3,4,5,6\}$ satisfies $K\not\in\{\varnothing, H,W\}$, $\m{K}=0$ and both the pairs $(e_H,e_K)$ and $(e_W,e_K)$ are not Cartan pairs.
\item If $e_W=e_{12356}$, then the set $K:=\{1,4,5,6\}$ satisfies $K\not\in\{\varnothing, H,W\}$, $\m{K}=3$ and both the pairs $(e_H,e_K)$ and $(e_W,e_K)$ are not Cartan pairs.
\end{itemize}

\noindent(V) $e_H=e_{12345}$. By the previous cases and up to a permutation of the indices, we only have to consider the following cases:
$$
e_W\in \{e_{12346}, e_{123456}\}.
$$
The set $K:=\{1,5,6\}$ satisfies $K\not\in\{\varnothing, H,W\}$, $\m{K}=3$ and both the pairs $(e_H,e_K)$ and $(e_W,e_K)$ are not Cartan pairs.

\noindent(VI) As the case $e_H=e_{123456}$ is included in the previous cases, we conclude that $\Cc(\R_6)>3$, as required. $\sqbullet$
\end{example}

\subsection{On the reconstructing number of the Clifford algebras $\R_n$}

By Corollary \ref{corCliffRn} and Example \ref{R4}, it follows $\Rr(\R_4)=2$. At this point we have (see Corollary \ref{division0}, Example \ref{RR3}, Corollary \ref{corCliffRn}, Lemma \ref{lem41} and Lemma \ref{lem4}):
\begin{itemize}
\item $\Rr(\R_2)=1$ and $\Rr(\R_3)=\Rr(\R_4)=2$,
\item if $n\geq 5$ and $\m{n}=1$, then $\Rr(\R_n)=2$,
\item if $n\geq 6$ and $\m{n}\in\{0,2,3\}$, then $2\leq \Rr(\R_n)\leq 4$.
\end{itemize}

The purpose of this section is to improve the inequality $2\leq \Rr(\R_n)$ for $\m{n}\in\{0,2,3\}$ when $n\geq 6$. 

\begin{prop}\label{maggiore3}
Let $n\geq 6$ be such that $\m{n}\in\{0,2,3\}$. Then, $\Rr(\R_n)\geq 3$.
\end{prop}
\begin{proof}
Let $\Omega\subset \Qq_{\R_n}$ be a circular slice domain. We have to show that, for each $K,H\in\mc{P}(n)$ such that $K\neq H$, there exists a non-constant slice regular function $f:\Omega\to \R_n$ such that the real components $f_K$ and $f_H$ are constant. Up to substitute $f$ with $fe_H$, we may reduce to show: \textit{For each $K\in \mc{P}(n)\setminus\{\varnothing\}$, there exists a non-constant slice regular function $f:\Omega\to \R_n$ such that the real components $f_{\varnothing}$ and $f_K$ are constant.} In particular, it is enough to show: \textit{For each $K\in \mc{P}(n)\setminus\{\varnothing\}$, there exists $H\in \mc{P}(n)$ such that the non-constant slice polynomial function $f(x):=x e_H$ has constant real components $f_{\varnothing}$ and $f_K$.}

Let $H\in \mc{P}(n)$ and $f(x):=xe_H$. By Corollary \ref{Clifford3}, we have 
$$
\{W\in\mc{P}(n) : e_W^c=-e_W\}=\{W\in\mc{P}(n) : e_W^2=-1\}.
$$
In particular, each imaginary unit $J\in\sph_{\R_n}$ can be written (uniquely) as
$
J=\sum_{e_W^2=-1} x_W e_W
$
for suitable $x_W\in \R$. Thus,
$$\textstyle
f(\alpha+\beta J)=\alpha e_H +\beta Je_H=\alpha e_H +\beta\sum_{e_W^2=-1} x_W e_We_H=\alpha e_H +\beta\sum_{e_W^2=-1} \sigma(W,H)x_W e_{W\ds H}.
$$
for each $\alpha,\beta\in \R$ and $J=\sum_{e_W^2=-1}x_We_W\in\sph_{\R_n}$. By Corollary \ref{Clifford3}, we deduce that all the real components $f_K$ of $f$ are constant for each 
$$
K\not\in \{H\}\cup \{ W\ds H : \m{W}=1,2\}.
$$
In particular, if $H\neq \varnothing$ and $\m{H}=0,3$, then $f_{\varnothing}$ is constant. Thus, in order to conclude it is enough to show: \textit{For each $K\in\mc{P}(n)\setminus\{\varnothing\}$, there exists $H\in\mc{P}(n)\setminus\{\varnothing,K\}$ such that $\m{H}=0,3$ and}
\begin{equation}\label{cond}
K\not\in \{ W\ds H : \m{W}=1,2\}.
\end{equation}

Let $K\in\mc{P}(n)\setminus\{\varnothing\}$. Observe that $K\not\in \{ W\ds H : \m{W}=1,2\}$ if and only if $\m{K\ds H}\neq1,2$. We consider the cases $[n]_4\in\{0,2,3\}$ separately.

\noindent{\sc Case $[n]_4=0$.} We consider several cases accordingly to $\m{K}$.
\begin{itemize}
\item $\m{K}\in\{0,1\}$. If $K=\{1,\ldots,n\}$, we take $H:=\{1,2,3,4\}$. If $|K|<n$, then we take $H:=\{1,\ldots,n\}\setminus K$. Clearly, in both cases, $H\neq \varnothing$ and $\m{H}=0$. Moreover, as $n\geq 6$, then $H\neq K$. As in both cases $\m{K\ds H}=0$, then \eqref{cond} holds.
\item $\m{K}=2$. As $K\neq \varnothing$, then there exists $k\in K$. We take $H:=(\{1,\ldots,n\}\setminus K)\cup \{k\}$. Clearly, $H\not\in\{ \varnothing, K\}$. Moreover, $\m{H}=3$. As $K\ds H=\{1,\ldots,n\}\setminus\{k\}$, then $\m{K\ds H}=3$, so \eqref{cond} holds.
\item $\m{K}=3$. As $n\geq 6$ then there exists $H\in\mc{P}(n)$ with $|H|=4$ such that either $H\subset K$ or $H\subset \{1,\ldots,n\}\setminus K$. In both cases $H\not\in\{ \varnothing, K\}$ and $\m{H}=0$. As $\m{K\ds H}=3$, then \eqref{cond} holds.
\end{itemize}
We conclude that, if $\m{n}=0$, then $\Rr(\R_n)\geq 3$, as required.

\noindent{\sc Case $[n]_4=2$.} We consider several cases accordingly to $\m{K}$.
\begin{itemize}
\item $\m{K}\in\{0,3\}$. Then there exists $k,h\in\{1,\ldots,n\}\setminus K$ such that $h\neq k$. Moreover, as $K\neq \varnothing$, there exists $s,r\in K$ such that $s\neq r$. We take $H:=\{k,h,s,r\}$. Clearly $H\not\in\{ \varnothing, K\}$ and $\m{H}=0$. As in both cases $\m{K\ds H}\in\{0,3\}$, then \eqref{cond} holds.
\item $\m{K}=1$. If $|K|<n-1$, as $n\geq 6$, then there exists $H\subset \{1,\ldots,n\}\setminus K$ with $|H|=3$. If $|K|=n-1$, then there exists $k,h\in K$ such that $k\neq h$ and $s\in\{1,\ldots,n\}\setminus K$. In this case we take $H:=\{h,k,s\}$. Clearly $H\not\in\{ \varnothing, K\}$ and $\m{H}=3$. As in both cases $\m{K\ds H}=0$, then \eqref{cond} holds.
\item $\m{K}=2$. If $|K|>2$, then there exists $H\subset K$ such that $|H|=3$. If $|K|=2$, as $n\geq 6$, then there exists $k,h\in\{1,\ldots,n\}\setminus K$ such that $k\neq h$. Let $s\in K$. In this case we take $H:=\{k,h,s\}$. Clearly $H\not\in\{ \varnothing, K\}$ and $\m{H}=3$. As in both cases $\m{K\ds H}=3$, then \eqref{cond} holds.
\end{itemize}
We conclude that, if $\m{n}=2$, then $\Rr(\R_n)\geq 3$, as required.

\noindent{\sc Case $[n]_4=3$.} We consider several cases accordingly to $\m{K}$.
\begin{itemize}
\item $\m{K}\in\{0,3\}$. If $K=\{1,\ldots,n\}$, we take $H:=\{1,2,3,4\}$. If $|K|<n$, then we take $H:=\{1,\ldots,n\}\setminus K$. Clearly $H\not\in\{ \varnothing, K\}$ and $\m{H}=0,3$. Moreover, as $n\geq 6$, then $H\neq K$. As in both cases $\m{K\ds H}=3$, then \eqref{cond} holds.
\item $\m{K}=1$. If $|K|<n-2$, then there exists $H\subset \{1,\ldots,n\}\setminus K$ such that $|H|=3$. If $|K|=n-2$, then as $n\geq 6$, there exists $k,h\in K$ such that $k\neq h$. Let $s\in \{1,\ldots,n\}\setminus K$. We take $H=\{k,h,s\}$. In both cases $H\not\in\{ \varnothing, K\}$ and $\m{H}=3$. Moreover, $\m{K\ds H}=0$, so \eqref{cond} holds.
\item $\m{K}=2$. If $|K|>2$, then there exists $H\subset K$ such that $\m{H}=3$. If $|K|=2$, as $n\geq 6$, then there exists $k,h\in \{1,\ldots,n\}\setminus K$ such that $k\neq h$. Let $s\in K$. We take $H=\{k,h,s\}$. In both cases $H\not\in\{ \varnothing, K\}$ and $\m{H}=3$. Moreover, $\m{K\ds H}=3$, so \eqref{cond} holds.
\end{itemize}
We conclude that, if $\m{n}=3$, then $\Rr(\R_n)\geq 3$, as required.
\end{proof}

By the inequality $\Rr(\R_n)\leq \Cc(\R_n)$, we deduce straightforwardly the following:

\begin{cor}
Let $n\geq 6$ be such that $\m{n}\in\{0,2,3\}$. Then, $\Cc(\R_n)\geq 3$.
\end{cor}

Putting all the previous results together, we obtain the following:

\begin{thm}[Reconstructing number of the Clifford algebras $\R_n$]\label{reconRn}
For each $n\geq 2$, it holds:
\begin{itemize}
\item[{\rm(i)}] $\Rr(\R_2)=\Rr(\HH)=1$.
\item[{\rm(ii)}] $\Rr(\R_3)=\Rr(\R_4)=2$.
\item[{\rm(iii)}] If $\m{n}=1$, then $\Rr(\R_n)=2$.
\item[{\rm(iv)}] If $n\geq 6$ and $\m{n}\in\{0,2,3\}$, then $3\leq \Rr(\R_n)\leq 4$.
\end{itemize}
\end{thm}

\section{Applications to slice-Nash functions}\label{S8}

\subsection{Slice-Nash functions}\label{SsliceNash}

\textit{Let $A$ be a finite dimensional alternative $*$-algebra with unity $1\neq 0$ such that $\sph_A\neq\varnothing$. Let $F:=F_1+\iota F_2:D\to A\otimes_\R\C$ be a stem function defined on a non-empty open set $D\subset \C$ and $f:\Omega\to A$ the corresponding slice function. We assume, in addition, that $D$ is invariant under conjugation (see Assumption \ref{assSimm}).}

As already mentioned in \S\ref{IntroNash}, the first and second author of this paper introduced in \cite{bc} the notion of Nash functions for slice regular functions of one quaternionic or octonionic variable. Even if in \cite{bc} the focus is only on quaternions and octonions, the definitions and most of the results of \cite{bc} extends verbatim to slice regular functions $f$ defined on open circular sets $\Omega\subset\Qq_A$. In this section, we briefly recall the definitions and main results of \cite{bc} that we need for our purposes. 

The left multiplication by an element of $\sph_A$ induces a complex structure on $A$. More precisely, for a fixed imaginary unit $J\in\sph_A$, the sum of $A$ together with the complex scalar multiplication $\C_J\times A\to A$ sending the pair $(\gamma,a)\in \C_J\times A$ into the product $\gamma a$ in $A$ defines a structure of $\C_J$-vector space on $A$ (if $A$ is alternative, but not associative, this fact follows by Artin's theorem). Define the integers $d_A:=\dim(A)-1$ and $u_A:=\tfrac{\dim(A)}{2}-1$. Observe that $u_A$ is equal to the dimension of $A$ as a complex vector space minus 1. If $\{1,J_1,\ldots,J_{u_A}\}$ is a basis of $A$ as a $\C_J$-vector space, then $\Bb_J:=\{1,J, J_1,JJ_1,\ldots,J_{u_A},JJ_{u_A}\}$ is a basis of $A$ as a $\R$-vector space \cite[Lem.2.3]{gp3}. A real vector basis of $A$ of this form is called a \textit{splitting basis of $A$ associated to $J$}. 

\textit{Let $J\in\sph_A$ be an imaginary unit and $\Bb_J:=\{J_0:=1,J, J_1,JJ_1,\ldots,J_{u_A},JJ_{u_A}\}$ a splitting basis of $A$ associated to $J$.} 

As $\Bb_J$ is a basis of $A$ as a $\R$-vector space, there exist unique real valued functions $F_{s,\ell}^{J,J_k}:D\to  \R$ for $s,\ell=1,2$ and $k=0,\ldots,u_A$ such that
$$\textstyle
F_s=\sum_{k=0}^{u_A}(F_{s,1}^{J,J_k}+F_{s,2}^{J,J_k}J)J_k
$$
for $s=1,2$. Define
$
F_{J,J_k}^{\ell}:=F_{1,\ell}^{J,J_k}+\iota F_{2,\ell}^{J,J_k}:D\to \C
$
for each $\ell=1,2$ and $k=0,\ldots,u_A$. We have
$$\textstyle
F=F_1+\iota F_2=\sum_{k=0}^{u_A}(F_{J,J_k}^{1}+F_{J,J_k}^{2}J)J_k.
$$

We start by recalling the following:

\begin{defn}[Stem-Nash function {\cite[Def.5.2]{bc}}]
A stem function $F:D\to A\otimes_\R\C$ is called a \textit{stem-Nash function on} $D$ if there exist an imaginary unit $J\in \sph_A$ and a splitting basis $\Bb_J$ of $A$ associated to $J$ such that the functions $F_{J,J_k}^{\ell}$ are complex Nash functions on $D$ for each $\ell=1,2$ and each $k=0,\ldots,u_A$. $\sqbullet$
\end{defn}

As complex Nash functions are, in particular, holomorphic (see Definition \ref{defNash}), we have that, if $F$ is a stem-Nash function on $D$, then the functions $F_{J,J_k}^{\ell}$ are holomorphic functions for each $\ell=1,2$ and $k=0,\ldots,u_A$. We deduce that (see also \cite[Rmk.3(2)]{gp}): \textit{If $F$ is a stem-Nash function, then $F$ is a holomorphic stem function.}

Next, we recall the definition of slice-Nash functions.

\begin{defn}[Slice-Nash functions {\cite[Def.5.7]{bc}}]\label{defslicenash}
We say that $f:=\mathcal{I}(F):\Omega\to A$ is a \textit{slice-Nash function on} $\Omega$ if the stem function $F$ is a stem-Nash function on $D$. $\sqbullet$
\end{defn}

As stem-Nash functions are, in particular, holomorphic, we deduce the following:

\begin{remark}\label{NashReg}
If $f$ is a slice-Nash function on $\Omega$, then $f$ is slice regular on $\Omega$. The converse, clearly, it is not true (see, for instance, Example \ref{exnonash} below). $\sqbullet$
\end{remark}

For each $J\in\sph_A$, denote by $f_J$ the restriction of $f$ to $\Omega_J:=\Omega\cap \C_J$. Let $\Bb_J:=\{J_0:=1,J, J_1,JJ_1,\ldots,J_{u_A},JJ_{u_A}\}$ be a splitting basis of $A$ associated to $J$. By the splitting lemma \cite[Lem.2.4]{gp3}, there exist unique holomorphic functions  $f_1^{J,J_k},f^{J,J_k}_2:\Omega_J\to \C_J$ for $k=0,\ldots,u_A$ such that 
$$\textstyle
f_J(x)=\sum_{k=0}^{u_A}(f_1^{J,J_k}(x)+f^{J,J_k}_2(x)J)J_k
$$
for each $x\in\Omega_J$. 

Let $\Bb:=\{\e_0,\ldots,\e_{d_A}\}$ be any basis of $A$ as a $\R$-vector space (here we are not requiring that $1\in\Bb$). Define
$$
\Omega_{\Bb}:=\{(x_0,\ldots,x_{d_A})\in \R^{d_A+1} : x_0\e_0+\ldots+x_{d_A}\e_{d_A}\in \Omega\}\subset \R^{d_A+1}.
$$
Let $f:\Omega\to A$ be a slice function. Then, there exist unique real valued functions $f_\ell^{\Bb}:\Omega\to \R$ for $\ell=0,\ldots,d_A$ such that $f(x)=f_0^{\Bb}(x)\e_0+\ldots+f_{d_A}^{\Bb}(x)\e_{d_A}$ for each $x\in \Omega$. After the identification of $A$ with $\R^{d_A+1}$ induced by the basis $\Bb$, the functions $f_\ell^{\Bb}$ can be regarded as functions of real variables
$$
f^{\Bb,\R}_{\ell}:\Omega_{\Bb}\to \R, \quad (x_0,\ldots,x_{d_A})\mapsto f_\ell^{\Bb}(x_0\e_0+\ldots+x_{d_A}\e_{d_A}).
$$
In order to lighten the exposition, up to a light abuse of notation, we will write $\Omega$ instead of $\Omega_{\Bb}$ and $f_\ell$ instead of $f_\ell^{\Bb,\R}$. This will create no confusion, as the situation will be always clear by the context. 

In \cite{bc}, we provided the following characterisation for slice-Nash functions. Even if the proofs contained in \cite{bc} are only for $A=\HH, \O$, it is easy to check that the same proofs, with obvious modifications, holds true also in this more general context.

\begin{thm}[Characterisation of slice-Nash functions {\cite[Thm. 5.9, Thm.5.10]{bc}}]\label{charNash}
The following are equivalent:
\begin{itemize}
\item[{\rm(i)}] $f$ is a slice-Nash function on $\Omega$.
\item[{\rm(ii})] There exist $J\in\sph_A$ and a splitting basis $\Bb_J$ of $A$ associated to $J$ such that the functions  $f_1^{J,J_k},f^{J,J_k}_2$ are $\C_J$-Nash functions on $\Omega_J$ for each $k=0,\ldots,u_A$. 
\item[{\rm(iii)}] For each $J\in\sph_A$ and each splitting basis $\Bb_J$ of $A$ associated to $J$, the functions $f_1^{J,J_k},f^{J,J_k}_2$ are $\C_J$-Nash functions on $\Omega_J$ for each $k=0,\ldots,u_A$. 
\item[{\rm(iv})] $f$ is slice regular and there exists a basis $\Bb$ of $A$ as a $\R$-vector space such that the real components $f_\ell$ of $f$ are real Nash functions on $\Omega$ for each $\ell=0,\ldots,d_A$.
\item[{\rm(v)}] $f$ is slice regular and for each basis $\Bb$ of $A$ as a $\R$-vector space, the real components $f_\ell$ of $f$ are real Nash functions for each $\ell=0,\ldots,d_A$.
\end{itemize}
\end{thm}

We end this section with the following:

\begin{remark}
From another perspective, Savi, in the recent work \cite{sa}, studied quaternionic and octonionic algebras from a model theoretic point of view. Slice regular functions which are definable in the theories he introduced are precisely those slice regular functions whose real components are real Nash functions. In particular, by Theorem \ref{charNash}, it follows that the class of definable slice regular functions are exactly the slice-Nash functions defined on open circular sets that are, in addition, semialgebraic. $\sqbullet$
\end{remark}

\subsection{Nash number}

In this section, \textit{$A$ denotes an alternative Cartan $*$-algebra equipped with the basis $\{e_K\}_{K\in\mc{P}(n)}$ whose product is induced by a function $\sigma$, such that $\sph_A\neq \varnothing$}.

Even if the definition of the Nash number $\Nn(A)$ of $A$ is already contained in Definition \ref{Nashnumber}, we recall it here for the reader's convenience. 

\begin{defn}[Nash number]
The \textit{Nash number $\Nn(A)$ of $A$} is the minimum integer $\ell$ that satisfies the following property: there exists a subset $\Kk$ of $\mc{P}(n)$ of cardinality~$\ell$ such that, for each open circular set $\Omega\subset \Qq_A$ and each slice regular function $f:\Omega\to A$, if the real components $f_K$ of $f$ are real Nash functions for each $k\in\Kk$, then $f$ is a slice-Nash function.~$\sqbullet$
\end{defn}

The Nash number $\Nn(A)$ is well-defined. In fact, by Theorem \ref{charNash}, if all the real components of $f$ are real Nash functions, then $f$ is a slice-Nash function. \textit{Let $\Omega=\Omega_D\subset \Qq_A$ be an open circular set, $F=\sum_{K\in\mc{P}(n)}F_Ke_K:D\to A\otimes_\R\C$ a stem function and $f=\I(F):\Omega\to A$ the corresponding slice function. Let $\{f_K\}_{K\in\mc{P}(n)}$ be the real components of $f$, see \eqref{equa}.} Recall that $D$ is invariant under conjugation, see Assumption \ref{assSimm}.

The purpose of this section is to show the following:

\begin{thm}\label{mainNash}
It holds $\Nn(A)\leq \Cc(A)$.
\end{thm}
\begin{proof}
Let $\Cc$ be a SCP, and $\Cc_1$ and $\Cc_2$ the sets introduced in \eqref{partition}. Assume that $\Cc(A)=|\Cc_1|$ (see \eqref{C1}) and that $f_L$ is a real Nash function for each $L\in\Cc_1$. For each $(e_L,e_M)\in\Cc$, let 
$$
f_{(e_L,e_M)}:=\mc{I}(F_L+F_M\sigma(L,L)\sigma(M,L)e_{L\ds M})
$$
be the slice regular function introduced in \eqref{vecchie}. By \eqref{vecchie2}, \eqref{uguaSw} and Proposition \ref{slicepres}(ii), we have
$$
f_{(e_L,e_M)}(x)=f_L(x)+f_M(x)\sigma(L,L)\sigma(M,L)e_{L\ds M}
$$
for each $x\in \Omega_{e_{L\ds M}}$. We start by showing: \textit{$f_{(e_L,e_M)}$ is a slice-Nash function for each $(e_L,e_M)\in \Cc$.}

Let $(e_L,e_M)\in\Cc$. By Propositions \ref{slicepres}(iii), the function
$
f_L\sigma(L,L)+f_M\sigma(M,L)e_{L\ds M}
$
is holomorphic on $\Omega_{e_{L\ds M}}$. Thus, the function
$$
f_L+f_M\sigma(L,L)\sigma(M,L)e_{L\ds M}:\Omega_{e_{L\ds M}}\to \C_{e_{L\ds M}}
$$
is holomorphic for each $(e_L,e_M)\in \Cc$. By \cite[Thm.1.3]{c}, we deduce that the function $f_L+\sigma(L,L)\sigma(M,L)f_Me_{L\ds M}$ is a complex Nash function on $\Omega\cap\C_{e_{L\ds M}}$, because $f_L$ is a real Nash function on $\Omega_{e_{L\ds M}}$. As $(e_L,e_M)$ is a Cartan pair, we have $e_{L\ds M}\in\sph_A$. Let $m\in L\ds M$. In particular, 
$
\{e_K,e_{L\ds M}e_K : m\not\in K\}
$
is a splitting basis of $A$ associated to $e_{L\ds M}$. By Theorem \ref{charNash}, we deduce that the function $f_{(e_L,e_M)}$ is a slice-Nash function, as required. 

Let $(e_L,e_M)\in\Cc$. The stem function of $f_{(e_L,e_M)}$ is $F_L+F_M\sigma(L,L)\sigma(M,L)e_{L\ds M}$. As $f_{(e_L,e_M)}$ is a slice-Nash function, $F_L$ and $F_M$ are complex Nash function. As $\Cc_1\cup \Cc_2=\mc{P}(n)$, we deduce that $F=\sum_{K\in\mc{P}(n)}F_Ke_K$ is a stem-Nash function. We conclude that $f$ is a slice-Nash function, as required.
\end{proof}

By the previous result and Example \ref{Cartannumber1}(i), we deduce straightforwardly the following:

\begin{cor}
The division algebras $\HH$ and $\O$ satisfy $\Nn(\HH)=\Nn(\O)=1$.
\end{cor}

\subsection{On the Nash number of Clifford algebras}\label{NashRpq}

\textit{Let $p,q\in\mathbb{N}$ with $n:=p+q\geq 2$. We consider the Clifford algebras $\R_{p,q}$ endowed with the Clifford conjugation. Recall that $\R_n:=\R_{(0,n)}$.}

In this section, we make use of Theorem \ref{mainNash} and the result of \S\ref{Scliff} to deduce the analogous of Theorem \ref{reconRpq} and \ref{reconRn} for the Nash number $\Nn(\R_{p,q})$ of the Clifford algebras $\R_{p,q}$. By Proposition \ref{Rpq} and Theorem \ref{mainNash}, we deduce straightforwardly the following:

\begin{thm}[Nash number of the Clifford algebras $\R_{p,q}$]\label{Nash1}
It holds
\begin{equation}\label{betterstima}
\Nn(\R_{p,q})\leq \frac{n^3+5n+6}{6}.
\end{equation}
\end{thm}

In order to show the analogous of Theorem \ref{reconRn} for the Nash number $\Nn(\R_{n})$ of the Clifford algebras $\R_{n}$, we need first to adapt some of the results of \S\ref{Scliff} to this context. We start with the following:

\begin{example}\label{exnonash}
Let $\exp(x)$ be the slice regular function induced by the holomorphic stem function $z\mapsto e^z:=1\otimes e^z$. The function $f(x):=\exp(x)e_{123}$, both if defined on the quadratic cone $\Qq_{\R_3}$ of $\R_3$ or the quadratic cone $\Qq_{\R_4}$ of $\R_4$, is a slice regular function which is not a slice-Nash function. Arguing as in the proof of Proposition \ref{maggioredi2}, it follows that the real component $f_{\varnothing}$ of $f$ is constantly equal to zero. In particular, $f_{\varnothing}$ is a real Nash function. We deduce that $\Nn(\R_3)\geq 2$ and $\Nn(\R_4)\geq 2$. $\sqbullet$
\end{example}

Recall that, given an integer $m\geq 0$, we denote by $\m{m}$ the remainder of $m$ divided by 4. Next, we adapt the argument of Proposition \ref{maggiore3} to show its analogous for $\Nn(\R_n)$.

\begin{prop}\label{maggiore3Nash}
Let $n\geq 6$ be such that $\m{n}\in\{0,2,3\}$. Then, $\Nn(\R_n)\geq 3$.
\end{prop}
\begin{proof}
The proof is exactly the same as the one of Proposition \ref{maggiore3}, except that in this case one has to consider the slice regular function $f(x):=\exp(x)e_H$ instead of the slice polynomial function $f(x):=x e_H$.
\end{proof}

Analogously, one can show that $\Nn(\R_n)\geq 2$ if $\m{m}=1$. Using the previous results, by Theorem \ref{CartanRn} and Theorem \ref{mainNash}, we deduce straightforwardly the following:

\begin{thm}[Nash number of the Clifford algebras $\R_n$]\label{Nash2}
For each $n\geq 2$, it holds:	
\begin{itemize}
\item[{\rm(i)}] $\Nn(\R_2)=\Nn(\HH)=1$.
\item[{\rm(ii)}] $\Nn(\R_3)=\Nn(\R_4)=2$.
\item[{\rm(iii)}] If $\m{n}=1$, then $\Nn(\R_n)=2$,
\item[{\rm(iv)}] If $n\geq 6$ and $\m{n}\in\{0,2,3\}$, then $3\leq \Nn(\R_n)\leq 4$.
\end{itemize}
\end{thm}

%%%

\bibliographystyle{amsalpha}

\end{document}